\documentclass{amsart} 

\usepackage{fullpage}

\usepackage[usenames,dvipsnames,svgnames,table]{xcolor}

\usepackage{graphicx} 

\usepackage{amsmath, amssymb, mathrsfs}

\usepackage[all, cmtip]{xy} 

\usepackage{tikz}
 \usetikzlibrary{cd}
 \tikzset{
  symbol/.style={
    draw=none,
    every to/.append style={
      edge node={node [sloped, allow upside down, auto=false]{$#1$}}}
      }
      }

\usepackage{xcolor}

\usepackage{amsthm}
	\theoremstyle{definition} 
	\newtheorem{defn}{Definition}[section]
	
	\theoremstyle{plain} 
	
	\newtheorem{thm}[defn]{Theorem}
	\newtheorem*{mthm}{Main Theorem} 
	\newtheorem{lem}[defn]{Lemma}
	\newtheorem{prop}[defn]{Proposition}
	\newtheorem{cor}[defn]{Corollary}
	\newtheorem{conj}[defn]{Conjecture}
	\theoremstyle{remark} 
	\newtheorem{rmk}[defn]{Remark}

\usepackage{enumitem}

\renewcommand{\AA}{\mathbb A}

\newcommand{\CC}{\mathbb C}

\newcommand{\QQ}{{\mathbb Q}}

\newcommand{\ZZ}{{\mathbb Z}}

\newcommand{\bQ}{{\mathbb Q}}
\newcommand{\bZ}{{\mathbb Z}}

\newcommand{\cF}{{\mathcal F}}

\newcommand{\cH}{{\mathcal H}}

\newcommand{\cO}{{\mathcal O}}

\newcommand{\cS}{{\mathcal S}}
\newcommand{\cT}{{\mathcal T}}
\newcommand{\cU}{{\mathcal U}}
\newcommand{\cV}{{\mathcal V}}

\newcommand{\frakX}{\mathfrak X}

\newcommand{\frakZ}{\mathfrak Z}

\newcommand{\frakh}{\mathfrak h}

\newcommand{\frakm}{\mathfrak m}

\newcommand{\frakz}{\mathfrak z}

\newcommand{\sH}{\mathscr H}
\newcommand{\sI}{\mathscr I}

\newcommand{\sW}{\mathscr W}

\newcommand{\Qbar}{\overline{\QQ}}
\newcommand{\Zbar}{\overline{\ZZ}}

\newcommand{\Zpbar}{\Zbar_p}

\newcommand{\Qpbar}{\Qbar_p}

\newcommand{\dR}{\mathrm{dR}}

\usepackage{colonequals}
\newcommand{\defeq}{\colonequals} 

\newcommand{\isom}{\cong} 
\newcommand{\congr}{\equiv} 

\newcommand{\inj}{\hookrightarrow}

\newcommand{\directsum}{\oplus} 
\newcommand{\tensor}{\otimes} 
\newcommand{\Directsum}{\bigoplus} 
\newcommand{\Tensor}{\bigotimes} 

\DeclareMathOperator\Hom{Hom} 
\DeclareMathOperator\GL{GL} 
\DeclareMathOperator\GSp{GSp} 
\DeclareMathOperator\Sp{Sp} 

\DeclareMathOperator\Spec{Spec} 

\newcommand{\et}{{\acute{\mathrm{e}}\mathrm{t}}} 
\newcommand{\mot}{\mathrm{mot}} 
\newcommand{\ord}{\mathrm{ord}} 
\newcommand{\Iw}{\mathrm{Iw}} 

\DeclareMathOperator\Ind{Ind} 

\newcommand{\id}{\mathrm{id}} 
\newcommand{\pr}{\mathrm{pr}} 

\DeclareMathOperator\ch{ch} 
\DeclareMathOperator\Gal{Gal} 
\DeclareMathOperator\Frob{Frob} 
\DeclareMathOperator\Sym{Sym} 
\let\det\relax
\DeclareMathOperator{\det}{det} 

\newcommand{\bran}{\mathrm{br}} 
\newcommand{\Eis}{\mathrm{Eis}} 
\newcommand{\Symbl}{\mathrm{Symbl}}
\DeclareMathOperator\cores{cores} 
\DeclareMathOperator{\Vol}{Vol}

\usepackage[OT2,T1]{fontenc}
\DeclareSymbolFont{cyrletters}{OT2}{wncyr}{m}{n}
\DeclareMathSymbol{\Sha}{\mathalpha}{cyrletters}{"58} 

\usepackage[bookmarks=false, 
        colorlinks, 
        linkcolor=red, 
        anchorcolor=red, 
        citecolor=red, 
        urlcolor=red] 
        {hyperref}  

\newcounter{counter}
\begin{document}
\title{Euler systems for $\GSp_4\times\GL_2$}

\author{Chi-Yun Hsu}
\address{Department of Mathematics, University of California, Los Angeles, Los Angeles, CA 90095, USA}
\email{cyhsu@math.ucla.edu}

\author{Zhaorong Jin}
\address{Department of Mathematics, Princeton University, Princeton, NJ, 08544}
\email{zhaorong.jin93@gmail.com} 

\author{Ryotaro Sakamoto}
\address{Department of Mathematics Faculty of Science and Technology, Keio University Yagami Campus: 3-14-1 Hiyoshi, Kohoku-ku, Yokohama, Kanagawa, 2238522, Japan }
\email{r-sakamoto@keio.jp}

\thanks{C.H. was partially supported by a Government Scholarship to Study Abroad from Taiwan}

\thanks{R.S. was supported by JSPS KAKENHI Grant Number 20J00456}

\date{\today}


\begin{abstract}
For a non-endoscopic cohomological cuspidal automorphic representation of $\GSp_4\times \GL_2$, assumed to be $p$-ordinary, we construct an Euler system for the Galois representation associated to it.
Both the construction and the verification of tame norm relations are based on Novodvorsky's integral formula for the $L$-function of $\GSp_4\times \GL_2$.
\end{abstract}
\maketitle

\vspace*{6pt}\tableofcontents


\section{Introduction}

\subsection{Recent progress in Euler systems}
The theory of Euler systems is one of the most powerful tools available for studying the arithmetic of Galois representations, and especially for controlling Selmer groups. It was first introduced by Kolyvagin \cite{Kolyvagin-ES}, inspired by works of Thaine \cite{Thaine} and his own \cite{Kolyvagin-MW}, and later formalized by Rubin \cite{Rubin-ES}.
However, the construction of Euler systems is a difficult problem, and there are relatively few known examples of Euler systems. 

A common strategy for constructing Euler systems is pushing forward cohomology classes on certain special cycles of Shimura varieties of varying levels.
The prototypical example is Kolyvagin's Euler system of Heegner points \cite{Kolyvagin-MW}, which uses codimension one cycles, the Heegner points, on modular curves, although one should note that this is not exactly an Euler system in the sense of \cite{Rubin-ES}, but rather a so-called ``anticyclotomic'' Euler system.
Another related construction is Kato's Euler system \cite{Kato-ES}, which utilizes the cup product of Siegel units on modular curves.

Recently a series of examples have been constructed using this technique, including the work of \cite{LLZ-RS} (embedding of Shimura varieties corresponding to $\mathrm{GL}_2\hookrightarrow \mathrm{GL}_2 \times_{\mathrm{GL}_1} \mathrm{GL}_2$),
\cite{LLZ-Hilbert} ($\mathrm{GL}_2\hookrightarrow \mathrm{Res}_{F/\mathbb{Q}}\mathrm{GL}_2 $ with $F$ a real quadratic field), and \cite{LSZ} ($\mathrm{GL}_2 \times_{\mathrm{GL}_1} \mathrm{GL}_2 \hookrightarrow \mathrm{GSp}_4$). The common starting point for all these recent constructions is Siegel units (or Eisenstein classes, the higher-weight analogue) realized in the appropriate motivic cohomology of (products of) modular curves. Pushing forward these cohomology classes via the corresponding closed embedding of Shimura varieties, one gets motivic cohomology classes in the target Shimura variety. The embeddings are suitably ``twisted'' in order to get the desired levels in the Shimura variety, which can be then realized as one base Shimura variety base-changed to various cyclotomic extensions of $\mathbb{Q}$. One then considers the \'etale realization of such motivic cohomology classes to find elements in \'etale cohomology. Finally, an application of the Hochschild--Serre spectral sequence yields Galois cohomology classes satisfying the desired norm relations.

In \cite{LLZ-RS} and \cite{LLZ-Hilbert}, the proof of the norm relations involves explicit yet laborious double coset computations. In \cite{LSZ}, a novel approach using a multiplicity-one argument in local representation theory is adopted, bypassing the complicated calculations on $\mathrm{GSp}_4$ that would have been practically intractable.
More specifically, Loeffler, Skinner and Zerbes use methods of smooth representation theory to reduce the tame norm-compatibility statement to a far easier, purely local statement involving Bessel models of unramified representations of $\mathrm{GSp}_4(\mathbb{Q}_\ell)$. This reduction is possible thanks to a case of the local Gan--Gross--Prasad conjecture due to Kato, Murase and Sugano \cite{KMS}, showing that the space of $\mathrm{SO}_4(\mathbb{Q}_\ell)$-invariant linear functionals on an irreducible spherical representation of $\mathrm{SO}_4(\mathbb{Q}_\ell)\times \mathrm{SO}_5(\mathbb{Q}_\ell)$ has dimension as most one. Note that $\mathrm{SO}_4$ and $\mathrm{SO}_5$ are the projectivization of $\GL_2\times_{\GL_1}\GL_2$ and $\GSp_4$, respectively. This technique promises to be applicable in many other settings where local multiplicity one results of this type are known. In \cite{Grossi}, the norm relations for the Euler system of Hilbert modular surfaces as first constructed in \cite{LLZ-Hilbert} is improved and reproved using this type of multiplicity one arguments. 

Exactly the same multiplicity one result, between $\mathrm{SO}_4$ and $\mathrm{SO}_5$, can also be applied in our setting of the natural embedding $\mathrm{GL}_2 \times_{\mathrm{GL}_1} \mathrm{GL}_2 \hookrightarrow \mathrm{GSp}_4 \times_{\mathrm{GL}_1} \mathrm{GL}_2$, as suggested in \cite{LSZ} and \cite{LZ-Arizona}.

\subsection{Outline of main results and the proofs}
Let $G=\GSp_4\times_{\GL_1}\GL_2$.
Let $\Pi=\Pi_f\tensor\Pi_\infty$ be a cuspidal automorphic representation of $G(\AA_f)$ with $\Pi_\infty$ a unitary discrete series of weight $(k_1,k_2,k)$ with $k_1\geq k_2\geq 3$, $k\geq2$.
Let $K\subset G(\AA_f)$ be the level of $\Pi$, and $\Sigma$ be the set of primes at which $K$ is not hyperspecial.
Let $W_\Pi$ be the $p$-adic Galois representation associated to $\Pi$, namely for each $\ell\notin\Sigma\cup\{p\}$, 
\[P_\ell(\ell^{-s})  = L(\Pi_\ell,s-\frac{k_1+k_2+k-4}2)^{-1},\]
where $P_\ell(X) = \det(1-\Frob_\ell^{-1} \cdot X \mid W_\Pi)$ is the characteristic polynomial of the geometric Frobenius $\Frob_\ell^{-1}$ on $W_\Pi$.

We state our main theorem, which is a combination of Definition \ref{def:Gal class} and Proposition \ref{wild}, \ref{tame}. 
\begin{mthm}
Assume that $\Pi$ is non-endoscopic.
Assume that $\Pi$ is unramified and (Borel) ordinary at $p$.
Fix an integer $e>1$ prime to $\Sigma\cup\{2,3,p\}$.
Let $r$ be an integer such that $\max(0,1-k_2+k) \leq r\leq \min(k_1-k_2, k-2)$.
Then there exists a Galois-stable lattice $T_\Pi^\ast \subset W_\Pi^\ast$ so that an Euler system for $T_\Pi^\ast(2-k_2-r)$ exists.
More precisely, for each square-free integer $M\geq1$ prime to $\Sigma\cup \{p\}$ and an integer $m\geq0$, there exists ${}_ez_{Mp^m}^{[\Pi,r]}\in H^1(\QQ(\mu_{Mp^m}),T_\Pi^\ast(2-k_2-r))$
satisfying
\[
\begin{array}{cl}
\text{(wild norm relation)} &
    \cores^{\QQ(\mu_{\ell Mp^{m+1}})}_{\QQ(\mu_{Mp^m})} {}_ez_{Mp^{m+1}}^{[\Pi,r]}
=  {}_e z_{Mp^m}^{[\Pi,r]},\\ \\
\text{(tame norm relation)} &
    \cores^{\QQ(\mu_{\ell Mp^m})}_{\QQ(\mu_{Mp^m})} {}_ez_{\ell Mp^m}^{[\Pi,r]}
= P_\ell(\ell^{1-k_2-r}\sigma_\ell^{-1}) \cdot {}_e z_{Mp^m}^{[\Pi,r]}.
\end{array}
\]
where $\sigma_\ell$ is the arithmetic Frobenius of $\ell$ in $\Gal(\QQ(\mu_{Mp^m}/\QQ)$.
\end{mthm}

We give a brief overview of how the Euler system classes ${}_e z_{Mp^m}^{[\Pi,r]}$ are constructed, and how to show the norm relations.
The method is similar to \cite{LSZ}.

\subsubsection*{Construction of Euler system classes}
Let $H = \mathrm{GL}_2 \times_{\mathrm{GL}_1} \mathrm{GL}_2$ and
$G = \mathrm{GSp}_4 \times_{\mathrm{GL}_1} \mathrm{GL}_2$.
For each open compact $U\subset G(\AA_f)$, let $Y_G(U)$ denote the Shimura variety for $G$ of level $U$.
For each of certain parameters $(a,b,c,r)\in \ZZ^4$, we will construct a $G(\mathbb{A}_{f})$-equivariant map, the symbol map (Section \ref{sec:symbl_inf})
\[
\Symbl\colon
\cS(\AA_f^2) \otimes_{\mathcal{H}(H(\mathbb{A}_{f}))} \mathcal{H}(G(\mathbb{A}_{f}))  
\rightarrow
\varinjlim_U H^5_\mathrm{mot}(Y_G(U),\mathscr{D}).
\]
Here $\cS(\AA_f^2)$ is the space of Schwartz functions on $\AA_f^2$, regarded as an $H(\mathbb{A}_{f})$-representation via the projection to its first factor $H \twoheadrightarrow \mathrm{GL}_2$, $\mathcal{H}(-)$ denotes the Hecke algebra, 
$\mathscr{D}$ is some relative Chow motive over $Y_G$ arising from algebraic representations of $G$.
The construction of $\Symbl$ records how to construct Euler system classes by pushing forward Eisenstein classes as mentioned above: The space of Schwartz functions $\cS(\AA_f^2)$ parametrizes Eisenstein classes, and the Hecke algebra $\cH(G(\AA_f))$ gives the freedom to twist the pushforward of Eisenstein classes.

Let $K\subset G(\mathbb{A}_{f})$ be a level subgroup, unramified outside $p$ and a finite set $\Sigma\not\ni p$ of primes. For any integer $n$ coprime to $\Sigma$, the base-extension 
$Y_G(K) \times_{\Spec{\mathbb{Q}}} \Spec{\mathbb{Q}(\mu_n)}$
is again a Shimura variety for $G$ of some level smaller than $K$. This identification, along with explicit choices of $K_p\subset G(\QQ_p)$ and test data $\phi_{Mp^m}\in\cS(\AA_f^2), \xi_{Mp^m}\in \cH(G(\AA_f))$ as input to $\Symbl$, we can define the motivic Euler system classes (Definition \ref{local data})
\[
z_{Mp^m} \in H^5_\mathrm{mot}(Y_G(K)  \times_{\Spec{\mathbb{Q}}} \Spec{\mathbb{Q}(\mu_{Mp^m})},\mathscr{D}),
\]
for $M\geq1$ a square-free integer coprime to $\{p\}\cup\Sigma$ and an integer $m \geq0$ .

To get an Euler system associated to the Galois representation $W_\Pi^\ast$, note that $\Pi_f^* \otimes W_\Pi^*$ appears with multiplicity 1 as a direct summand of
$\varinjlim_U H^4_\mathrm{\et}(Y_G(U)_{\bar{\mathbb{Q}}},\mathscr{D})$, and our conditions on $\Pi$ would ensure it does not contribute to cohomology outside degree $4$. Choosing a vector $\varphi \in \Pi_f$ thus results in a homomorphism of Galois representations
$\Pi_f^* \otimes W_\Pi^* \rightarrow W_\Pi^* $, which factors through 
$(\Pi_f^*)^K \otimes W_\Pi^*$ if $\varphi$ is $K$-invariant. This, combined with the Hochschild--Serre spectral sequence, gives a map
\[
H^5_\mathrm{\et}(Y_G(K) \times_{\Spec{\mathbb{Q}}} \Spec{\mathbb{Q}(\mu_{Mp^m})},\mathscr{D})
\rightarrow
H^1(\mathbb{Q}(\mu_{Mp^m}),W_\Pi^*).
\]
The images of the \'etale realizations of $z_{Mp^m}$ give a collection of cohomology classes
$z^\Pi_{Mp^m} \in H^1(\mathbb{Q}(\mu_{Mp^m}),W_\Pi^*)$, and we will prove that they satisfy the norm relations of an Euler system.
If we carefully trace the parameters $(a,b,c,r)\in \ZZ^4$ and take care of integrality of the Euler system classes (which depends on an auxiliary choice of integer $e$), we will construct ${}_ez^{[\Pi,r]}_{Mp^m}$ as in the Main Theorem.

\subsubsection*{Norm relations}
We will focus on the tame norm relations. 
The wild norm relations can be dealt with in a similar fashion and the lack of Euler factor makes it easier. 
We have seen that the symbol map induces a $G(\mathbb{A}_{f})$-equivariant bilinear pairing
\[
\frakZ \colon \left( \cS(\AA_f^2) \otimes_{\mathcal{H}(H(\mathbb{A}_{f}))} \mathcal{H}(G(\mathbb{A}_{f})) \right)
\otimes \Pi_f
\rightarrow
H^1(\mathbb{Q}(\mu_{Mp^m}),W_\Pi^*),
\]
Via Frobenius reciprocity, $\frakZ$ corresponds to an $H(\mathbb{A}_{f})$-equivariant pairing 
\[
\mathfrak{z}\colon \cS(\AA_f^2) \otimes\Pi_f \rightarrow H^1(\mathbb{Q}(\mu_{Mp^m}),W_\Pi^*).
\]

 As the symbol map is constructed out of the Eisenstein class map, it can be shown that it suffices to study the $\frakZ$ and $\frakz$ replacing $\cS(\AA_f^2)$ by principal series $\tau$ of $\GL_2$. 
 Note that the classes $z^\Pi_{Mp^m}$ and $\cores^{\QQ(\mu_{\ell Mp^m})}_{\QQ(\mu_{Mp^m})}(z^\Pi_{\ell Mp^m})$ are given by images of the $\frakZ$, with different choices of test data only differing at $\ell$.
Hence it suffices to study $\frakZ_\ell$, or equivalently $\frakz_\ell$.
  Therefore we are in the situation of comparing the values of different local test data under the $H(\mathbb{Q}_\ell)$-equivariant pairing
\[
\mathfrak{z}_\ell\colon \tau_\ell \otimes \Pi_\ell \rightarrow \mathbb{C}.
\]
By a known case of the Gan--Gross--Prasad conjecture due to \cite{KMS}, the space of such $H(\mathbb{Q}_\ell)$-equivariant maps is $1$-dimensional, and for each principal series $\tau_\ell = I(\chi,\psi)$ we may construct an explicit basis element $\frakz_{\chi,\psi}$ using zeta integrals (defined just before Proposition \ref{z U action}). Thus the desired tame norm relations are reduced to explicit computations of local zeta integrals (Section \ref{sec:tame}).

\subsubsection*{Choice of local test data}
For the tame norm relation to actually hold, we need to choose two suitable local test data in $\tau_\ell \tensor \Pi_\ell$ so that the their image under $\frakz_{\chi,\psi}$ differ by a ratio of $L(0,\Pi_\ell)^{-1}$.
The test vector in $\tau_\ell$ will be some Siegel sections, which is fairly standard.
It is the test vectors in $\Pi_\ell$ that will play an important role.
Our $\Pi_\ell$ will be assumed to be generic.
Let $\varphi_0$ be the normalized spherical vector of $\Pi_\ell$. 
The local zeta integral at $\varphi_0$ can be rewritten as a double sum of Whittaker functions at certain values, and the Casselman--Shalika formula says the double sum would give the local $L$-factor $L(0,\Pi_\ell)$.
We can similarly write down the local zeta integral at Hecke translates of $\varphi_0$ as (different) double sums of Whittaker functions.
Playing with the double sums, we can find a correct Hecke translate of $\varphi_0$ at which the local zeta integral is $1$.

\subsection{Organization}
Finally we give a summary of the the contents of the paper.
Section \ref{sec:alg rep} to \ref{sec:rep theory} are largely background knowledge.
In Section \ref{sec:alg rep} we recall algebraic representation theory; in particular we investigate the branching law: how irreducible representations of $G$ decomposes as irreducible representations of $H$.
In Section \ref{sec:cohom} we fix notations for Shimura varieties, construct coefficient sheaves from the algebraic representations, and recall Eisenstein classes on modular curves.
In Section \ref{sec:symblmap}, we construct the symbol map according to the integral formula for $L$-functions of $G$.
In Section \ref{sec:rep theory}, we recall local representation theory, in preparation to study the local symbol maps.
Section \ref{sec:local formula} is the core of the paper: Section \ref{sec:tame}, \ref{sec:div} and \ref{sec:wild} contains the key computation needed for tame norm relation, integrality, and wild norm relation for Euler system classes.
In Section \ref{sec:ES}, we assemble the results to construct Euler system elements and prove their norm relations.

\subsection*{Acknowledgements}
The project originated from the Arizona Winter School in 2018.
We thank the organizers of the AWS for making the collaboration possible.
We thank David Loeffler and Sarah Zerbes for proposing the problem and their continuing guidance.
We thank Vishal Arul, Alex Smith, Nicholas Triantafillou, Jun Wang, and He Wei for their participation in the project group during the AWS.
The Winter School was supported by NSF grant DMS-1504537 and by the Clay
Mathematics Institute.

\section{General Notations}
\begin{itemize}
\item Let $\widehat{\ZZ} \defeq \varprojlim_n \ZZ/n\ZZ$ be the profinite completion of $\ZZ$ and $\AA_{f} \defeq \widehat{\ZZ} \otimes_{\ZZ} \QQ$ be the ring of the finite ad\`eles of $\QQ$. 
We fix an embedding $\overline{\QQ} \hookrightarrow \overline{\QQ}_{\ell}$ for any prime $\ell$. 

\item Let $J$ be the skew-symmetric $4 \times 4$ matrix over $\ZZ$ given by 
$
\begin{pmatrix}
&&& 1\\
&&1& \\
&-1&&\\
-1&&& 
\end{pmatrix}
$. 
Let $\GSp_{4}$ be the group scheme over $\ZZ$ defined by 
\[
\GSp_{4}(R) \defeq \{ (g, \mu) \in \GL_{4}(R) \times \GL_{1}(R) \colon g^t \cdot J \cdot g = \mu J \}
\] 
for any commutative unital ring $R$. 
Write $\mu \colon \GSp_{4} \to \GL_{1}$ for the symplectic multiplier map. 

\item Put $G \defeq \GSp_{4} \times_{\GL_{1}} \GL_{2}$ and $H \defeq \GL_{2} \times_{\GL_{1}} \GL_{2}$. 
Here the fiber product is for $\GSp_4 \rightarrow \GL_1$ being the symplectic multiplier map $\mu$ and for $\GL_2 \rightarrow \GL_1$ being the determinant map. 
We also denote by $\mu$ the composition maps $G \xrightarrow{{\mathrm pr}_{1}} \GSp_{4} \xrightarrow{\mu} \GL_{1}$ and $H \xrightarrow{{\mathrm pr}_{1}} \GL_{2} \xrightarrow{\det} \GL_{1}$. 

\item Let $\iota$ denote the embedding $H \inj G$ given by 
\[
(
\begin{pmatrix}
a & b 
\\
c & d \end{pmatrix}
, 
\begin{pmatrix}
a' & b' 
\\
c' & d' 
\end{pmatrix}
) \mapsto 
(
\begin{pmatrix}
a &&& b \\
& a' & b' &\\
 & c' & d' &\\
c &&& d
\end{pmatrix}
, 
\begin{pmatrix}
a' & b' 
\\
c' & d' 
\end{pmatrix}
). 
\]
\end{itemize}

\section{Algebraic Representations and Branching Laws}
\label{sec:alg rep}

\subsection{Algebraic representations}
We first recall the representations of $\GL_2$.
Write $t_1$ and $t_2$ for the characters of the diagonal torus of $\GL_2$, given by projection onto the two diagonal entries.
Note that $t_1+t_2$ is the determinant $\det$ restricting to the torus and that $\{ t_1, t_2 \}$ is a basis of the character group of the torus.
For an integer $k\geq0$, we denote by $\Sym^k$ the $k$-th symmetric power of the standard representation of $\GL_2$. 
Note that $\Sym^k$ is the unique (up to isomorphism) algebraic representation of $\GL_2$ with highest weight $kt_1$. 
Here the positivity of roots is defined with respect to the Borel subgroup of upper triangular matrices of $\GL_{2}$. 
It has dimension $k+1$ and the central character $x \mapsto x^k$. 
Also, the dual representation satisfies 
\[
(\Sym^k)^\ast \isom \Sym^k \tensor \det^{-k}. 
\]

Now we turn to the representations of $\GSp_4$.
Write $\chi_1, \ldots, \chi_4$ for the characters of the diagonal torus of $\GSp_4$ given by projection on to the four diagonal entries. 
Note that $\chi_1+\chi_4=\chi_2+\chi_3$ is the symplectic multiplier $\mu$ restricting to the torus and that $\{\chi_1, \chi_2, \mu \}$ is a basis of the character group of the torus. 
For integers $a,b\geq0$, we denote by $V^{a,b}$ the unique (up to isomorphism) irreducible algebraic representation of $\GSp_4$ whose highest weight, with respect to the Borel subgroup of upper triangular matrices of $\GSp_{4}$, is the character $(a+b)\chi_1+a\chi_2$. 
The representation $V^{a,b}$ has dimension $\frac16 (a+1)(b+1)(a+b+2)(2a+b+3)$ and the central character $x \mapsto x^{2a+b}$. 
The dual representation satisfies 
\[
(V^{a,b})^\ast \isom V^{a,b} \tensor \mu^{-2a-b}. 
\]

We may thus deduce the representation theory of $G = \GSp_{4} \times_{\GL_{1}} \GL_{2}$ from the above. 
For any integers $a,b,c\geq0$, set 
\[ 
W^{a,b,c} \defeq V^{a,b} \boxtimes \Sym^c. 
\]
This is an irreducible algebraic representations of $G$ with highest weight $(a+b)\chi_{1} + b\chi_{2} + ct_{1}$.

\subsection{Branching laws} 

We are interested in $\iota^\ast(W^{a,b,c})$, the restriction of $W^{a,b,c}$ to $H$ via the embedding $\iota: H \inj G$.
We have the following branching law.

\begin{prop}\label{prop:branching_laws}
For any integers $a,b,c\geq0$, we have 
\[
\iota^\ast(W^{a,b,c}) = \Directsum_{q=0}^a \Directsum_{r=0}^b \Directsum_{j=0}^{\min(a-q+r,c)} {\Sym^{a+b-q-r} \boxtimes \Sym^{a-q+r+c-2j}} \tensor \mu^{j+q}. 
\]
\end{prop}
\begin{proof}
We first observe that the embedding $\iota \colon H \inj G$ factors as
\[
H = \GL_2 \times_{\GL_1} \GL_2 \xrightarrow{(\id_{\GL_2}, \iota_1)} \GL_2 \times_{\GL_1} \GL_2 \times_{\GL_1} \GL_2 \xrightarrow{(\iota_2, \id_{\GL_2})} 
\GSp_4 \times_{\GL_1} \GL_2 = G, 
\]
where 
$\iota_1 \colon \GL_2 \inj \GL_2 \times_{\GL_1} \GL_2$ is the diagonal embedding 
and 
$\iota_2 = \pr_1\circ\iota \colon \GL_2 \times_{\GL_1} \GL_2 \inj \GSp_4$. 
The branching laws for the embeddings $\iota_1$ and $\iota_2$ are known; the former is the Clebsch--Gordan formula
\[
\iota_1^\ast(\Sym^k \boxtimes \Sym^\ell) = 
\Directsum_{j=0}^{\min(k,\ell)} \Sym^{k+\ell-2j} \tensor \mu^j, 
\]
and the latter is 
\[
\iota_2^\ast(V^{a,b}) = 
\Directsum_{q=0}^a \Directsum_{r=0}^b \Sym^{a+b-q-r} \boxtimes \Sym^{a-q+r} \tensor \mu^q 
\]
(see \cite[Proposition 4.3.1]{LSZ}). 
Hence $\iota^\ast(W^{a,b,c})$ is equal to
\begin{align*}
(\iota_1, \id_{\GL_{2}})^{*} \circ (\id_{\GL_{2}}, \iota_2)^{*} (W^{a,b,c}) 
&= \bigoplus_{q=0}^{a} \bigoplus_{r=0}^{b} {\mathrm Sym}^{a+b-q-r} \boxtimes \iota_{1}^{*} \left( {\mathrm Sym}^{a-q+r}  \boxtimes \Sym^c \right) \tensor \mu^{q}
\\
&= 
\Directsum_{q=0}^a \Directsum_{r=0}^b \Directsum_{j=0}^{\min(a-q+r,c)} \Sym^{a+b-q-r} \boxtimes \Sym^{a-q+r+c-2j} \tensor \mu^{q+j}. 
\end{align*}
\end{proof}

\begin{cor}\label{cor: branching}
Let $a,b,c\geq0$ be integers. 
Then the representation $\iota^\ast(W^{a,b,c})$ has an irreducible $H$-subrepresentation of the form $(\Sym^{d}\tensor \mu^{e}) \boxtimes 1$ for some integers $d \geq 0$ and $e$ if and only if $c \leq a+b$. 
Moreover, if $c \leq a+b$, then such $H$-subrepresentations are exactly
\[
(\Sym^{b+c-2r}\tensor \mu^{a+r}) \boxtimes 1,
\]
where $r$ is an integer such that $\max(0,-a+c) \leq r \leq \min(b,c)$.
\end{cor}
\begin{proof}
According to Proposition~\ref{prop:branching_laws}, 
the representation $\iota^{*}(W^{a,b,c})$ has an irreducible $H$-subrepresentation which is the pullback from the first factor of a $\GL_2$-irreducible representation is equivalent to that there exists non-negative integers $q \leq a$, $r \leq b$, and $j \leq \min(a-q+r,c)$ with $a-q+r+c-2j=0$. 
By the condition on $j$, the equality $a-q+r+c-2j=0$ implies that $a-q+r = c = j$. 
Since $0 \leq q \leq a$ and $r \leq b$, the inequality $c\leq a+b$ follows from the equality $a-q+r=c$. 

If $c\leq a+b$, $0\leq q \leq a$, $0\leq r\leq b$, and $a-q+r=c$, then  Proposition~\ref{prop:branching_laws} says that the representation $\iota^{*}(W^{a,b,c})$ has the irreducible $H$-subrepresentation 
\[
\Sym^{a+b-q-r} \boxtimes 1 \tensor \mu^{a+r} = (\Sym^{b+c-2r}\tensor \mu^{a+r}) \boxtimes 1. 
\]
Moreover, since $r=-a+c+q$ and $0\leq q \leq a$, we have $-a+c\leq r \leq c$. 
Combining the condition $0\leq r\leq b$, we obtain the desired range of $r$. 
\end{proof}

Following Corollary~\ref{cor: branching}, we will now show that there is actually a canonical injective homomorphism of $H$-representations
\[
{\mathrm br}^{[a,b,c,r]} \colon (\Sym^{b+c-2r}\tensor \det^{a+r}) \boxtimes 1 \hookrightarrow \iota^\ast(W^{a,b,c}),
\]
which will be referred to as the \emph{branching maps}.

We fix the following choices
\begin{itemize}
\item $u\in V^{1,0}$ with weight $\chi_1+\chi_2$, a highest weight vector, 
\item $u'\in V^{1,0}$ with weight $\mu$, 
\item $v\in V^{0,1}$ with weight $\chi_1$, a highest weight vector, 
\item $v'\in V^{0,1}$ with weight $\chi_2$, 
\item $w\in \Sym^1$ with weight $t_1$, a highest weight vector, 
\item $w'\in \Sym^1$ with weight $t_2$. 
\end{itemize}
In fact, $V^{0,1}$ is isomorphic to the standard representation of $\GSp_4 \subseteq \GL_4$, so we fix an isomorphism and choose $v$ to be the vector corresponding to $e_1$, where $e_1,\ldots, e_4$ is the standard basis for the standard representation. 
Under this isomorphism, we choose $v'$ to be $e_2$. 
Next, we identify $V^{1,0}$ with the $5$-dimensional irreducible subrepresentation of $\bigwedge^2 V^{0,1}$, and choose $u$ to be with $v\wedge v'$, or $e_1 \wedge e_2$. 
Then $u'$ is $e_1\wedge e_4-e_2\wedge e_3$. 
Finally, $\Sym^1$ is the standard representation of $\GL_2$. 
We take the representation space to be the subspace generated by $e_2, e_3$, and choose $w$ to be $e_2$. 
Then $w'$ is $e_3$. 
Note that the choice of $v$ determines all the other five vectors.

Now we can define
\[
v^{a, b, c, r} \defeq u^{c-r}\cdot v^{b-r} \cdot (u')^{a-c+r} \cdot (v')^r \cdot (w')^c. 
\]
It is easily seen that $v^{a,b,c,r}$ lies in $W^{a,b,c}$.
In addition, it has $G$-highest weight $(b+c-2r)\chi_1+c\chi_2-ct_1+(a+r)\mu$ and hence $H$-highest weight $(b+c-2r)t_1 \boxtimes 1 + (a+r)\mu$. 
This means that $v^{a,b,c,r}$ is a highest weight vector for the irreducible $H$-subrepresentation of $\iota^\ast(W^{a,b,c})$ which is isomorphic to $(\Sym^{b+c-2r}\tensor \mu^{a+r}) \boxtimes 1$.
Hence we can define
\begin{align} \label{align: definition of branching map}
\bran^{[a,b,c,r]} \colon (\Sym^{b+c-2r}\tensor \mu^{a+r}) \boxtimes 1 \inj \iota^\ast(W^{a,b,c}).
\end{align}
by sending the highest weight vector $(w^{b+c-2r} \tensor (w\wedge w')^{a+r})\boxtimes 1$ to $v^{a,b,c,r}$.
Note that changing our fixed choice of highest weight vector $v$ to a scalar multiple $c\cdot v$ only changes the the two highest weight vectors by the same scalar $c^{2a+b+c}$, and thus the branching map $\bran^{[a,b,c,r]}$ just defined is independent of the choice of $u$, $u'$, $v$, $v'$, $w$, and $w'$. 

\subsection{Integrality}
In this subsection, we discuss the integrality of branching maps.
Let $\lambda$ be a dominant integral weight of $G$, $W_\lambda$ the corresponding representation, and $w_\lambda \in W_\lambda$ a highest weight vector of $W_\lambda$. 
The pair $(W_\lambda,w_\lambda)$ is unique up to a unique isomorphism.

An \emph{admissible lattice} in $W_\lambda$ is a $\ZZ$-lattice $L$ such that
\begin{itemize}
\item the homomorphism $G\rightarrow \GL(W_\lambda)$ extends to a homomorphism $G \rightarrow \GL(L)$ of group schemes over $\ZZ$,
\item the intersection of $L$ with the highest weight space of $W_\lambda$ is $\ZZ w_\lambda$. 
\end{itemize}

It is known that there are finitely many admissible lattices. 
Let $W_{\lambda,\ZZ}$ be the maximal admissible lattice.
For two dominant integral weights $\lambda$ and $\lambda'$ of $G$, the image of $W_{\lambda,\ZZ} \tensor W_{\lambda',\ZZ}$ under the natural $G$-homomorphism
\[
W_\lambda \tensor W_{\lambda'} \rightarrow W_{\lambda+\lambda'}
\]
is an admissible lattice, so the map exist integrally
\[
W_{\lambda,\ZZ} \tensor W_{\lambda',\ZZ} \rightarrow W_{\lambda+\lambda',\ZZ}.
\]

Now we prove that the branching maps also exists integrally. 

\begin{prop}\label{prop: integral branching map}
Let $a$, $b$, and $c$ be non-negative integers with $c \leq a+b$. 
For any integer $r$ with $\max(0,-a+c) \leq r \leq \min(b,c)$, the branching map $\bran^{[a,b,c,r]}$ induces a morphism of integral representations 
\[
{\mathrm br}^{[a,b,c,r]} \colon (\mathrm{TSym}^{b+c-2r}\ZZ^2 \tensor \mu^{a+r}) \boxtimes 1 \inj \iota^\ast(W^{a,b,c}_\ZZ),
\]
where $\mathrm {TSym}^{b+c-2r}\ZZ^2$, the symmetric tensors, is the minimal admissible lattice in $\Sym^{b+c-2r}$.
\end{prop}
\begin{proof}
First of all, $L\defeq (\bran^{[a,b,c,r]})^{-1}(W^{a,b,c}_\ZZ)$ is a lattice of $\Sym^{b+c-2r} \tensor \mu^{a+r}$ stable under the $H$-action.
Also by our construction of branching maps, $L$ contains the highest weight vector $w^{b+c-2r} \tensor (w\wedge w')^{a+r} \boxtimes 1$.
Hence $L$ is an admissible lattice of $\Sym^{b+c-2r} \tensor \mu^{a+r}$ and contains the minimal admissible lattice.
\end{proof}

\section{Cohomology and Eisenstein classes}
\label{sec:cohom}
\subsection{Modular varieties}
We introduce some notations for modular varieties.  
For $K^{\GL_2}\subset \GL_2(\AA_f)$ a open compact subgroup, We let $Y_{\GL_2}(K^{\GL_2})$ denote the modular curve over $\QQ$ of level $K$. 
It has $\CC$-points
\[
Y_{\GL_2}(K^{\GL_2})(\CC) = \GL_2(\QQ)\backslash  \frakh^{\pm}\times \GL_2(\AA_f)/ K^{\GL_2}
\]
where $\frakh^{\pm}$ denotes the union of the upper and lower half planes in $\CC$.
Also let $Y_{\GL_2} \defeq \varprojlim_{K^{\GL_2}} Y_{\GL_2}(K^{\GL_2})$.
Similarly, we denote by $Y_H$ and $Y_G$ the canonical model over $\QQ$ for the Shimura varieties of $H$ and $G$, respectively.

For any open compact subgroup $K\subset G(\AA_f)$, the embedding $\iota \colon H \hookrightarrow G$ induces a morphism of varieties over $\QQ$
\[
\iota \colon Y_{H}(K\cap H(\AA_f)) \rightarrow Y_{G}(K).
\]

\subsection{Coefficient sheaves}

In this section, let $a,b,c,r\geq0$ be integers with $c \leq a+ b$ and $\max(0, -a+c) \leq r \leq \min (b, c)$. 

By \cite[Th\'eor\`em 8.6]{Anc15} (see \cite[Section 6]{LSZ}), the algebraic $G$-representation $W^{a,b,c}$ 
over $\QQ$ induces a $G(\AA_{f})$-equivariant relative Chow motive $\sW_{\QQ}^{a,b,c}$ over $Y_{G}$. 
We also have a $\GL_2(\AA_{f})$-equivariant relative Chow motive $\sH^k_\QQ$ over $Y_{\GL_2}$ associated to the representation $(\Sym^k)^\ast=\Sym^k \tensor \det^{-k}$ 
of $\GL_2$ (Note the dual.)  
Then, by the work of Torzewski \cite{Tor}, the branching map ${\mathrm br}^{[a,b,c,r]}$ in (\ref{align: definition of branching map}) gives a morphism of $H(\AA_{f})$-equivariant relative Chow motives 
\[
{\mathrm br}^{[a,b,c,r]} \colon \mathscr{H}^{b+c-2r}_\QQ \boxtimes 1 \hookrightarrow \iota^\ast\left( \mathscr{W}_{\QQ}^{a,b,c,*}(-a-r)[-a-r] \right),
\]
where for any integer $m$, the symbol $[m]$ means twisting by the character $\|\mu(-)\|^{m}$ of $G(\AA_{f})$.  
Note that the $G(\AA_{f})$-equivariant relative Chow motive associated to $\mu$ is $\QQ(-1)[-1]$ (see \cite[Lemma 6.2.2]{LSZ}).

\subsection{Eisenstein classes for $\GL_2$}

For any space $X$ and field $F$, let $\cS(X,F)$ denote the space of Schwartz functions on $X$ with values in $F$.
Let $\cS_0(\AA_f^2,F)\subset\cS(\AA_f^2,F)$ be the subspace of functions $\phi$ such that $\phi(0,0)=0$.
For $e>1$ an integer, let ${}_e\cS_0(\AA_f^2,\ZZ)\subset\cS_0(\AA_f^2,\QQ)$ denote the subgroup of functions of the form $\phi=\phi^{(e)}\cdot\prod_{\ell\mid e}\ch(\ZZ_\ell^2)$, where $\phi^{(e)}\in\cS((\AA_f^{(e)})^2,\QQ)$.

Recall that we have the Siegel units on modular curves $Y_{\GL_2}$.
\begin{thm} 
There is a $\GL_2(\AA_f)$-equivariant map 
\[
 \cS_0(\AA_f^2,\QQ) \rightarrow 
  H^1_\mot(Y_{\GL_2}, \QQ(1))= \cO^\times(Y_{\GL_2}) \tensor_{\ZZ}\QQ, 
  \phi\mapsto g_\phi
\]
so that if $\phi=\ch((a,b)+N\hat\ZZ^2)$ for an integer $N\geq1$ and $(a,b)\in\QQ^2-N\ZZ^2$, then $g_\phi$ is the Siegel unit $g_{a/N,b/N}$.
\end{thm}
\begin{proof}
See \cite[Proposition 7.1.1]{LSZ}.
\end{proof}

There is also an integral version of Siegel units.
\begin{thm}
If $e$ is a positive integer prime to $6$, then there is a map 
\[
 {}_e\cS_0(\AA_f^2,\ZZ)\rightarrow H^1_\mot(Y_{\GL_2},\ZZ(1))=\cO^\times(Y_{\GL_2}),
 \phi \mapsto {}_eg_\phi
\]
so that ${}_eg_\phi = (e^2-\begin{pmatrix}e&\\&e\end{pmatrix}^{-1})g_\phi$ as elements of $H^1_\mot(Y_{\GL_2},\QQ(1))$.
\end{thm}
\begin{proof}
See \cite[Proposition 7.1.2]{LSZ}.
\end{proof}

For higher weights, there is the map of Eisenstein symbol, being an analogue of Siegel units.
\begin{thm}
 For an integer $k\geq1$, there is a $\GL_2(\AA_f)$-equivariant map
\[
 \cS(\AA_f^2,\QQ) \rightarrow H^1_\mot(Y_{\GL_2}, \sH^k_\QQ(1)), \phi \mapsto \Eis^k_{\mot, \phi}
\]
the \emph{motivic Eisenstein symbol}, satisfying the following property:
the pullback of the de Rham realization $r_\dR(\Eis^k_{\mot,\phi})$ to the upper half plane is the $\sH^k_\QQ$-valued differential $1$-form
\[
 -F^{(k+2)}_\phi(\tau)(2\pi i dz)^k (2\pi i d\tau)
\]
where $F^{(k+2)}_\phi$ is the Eisenstein series defined by 
\[
 F^{(k+2)}_\phi(\tau) = \frac{(k+1)!}{(-2\pi i)^{k+2}} \sum_{\substack{x,y\in \QQ \\ (x,y)\neq (0,0)}} \frac{\hat\phi(x,y)}{(x\tau+y)^{k+2}}.
\]
To unify notations, for $k=0$ with the further assumption that $\phi\in \cS_0(\AA_f^2,\QQ)$, we also define $\Eis^0_{\mot,\phi} \defeq g_\phi$.
\end{thm}
\begin{proof}
See \cite[Theorem 7.2.2]{LSZ}.
\end{proof}

To have an integral version of Eisenstein symbols, we need to go the \'etale realization.
Let $p$ be a prime number.
Let $\sH^k_{\QQ_p}$ denote the $p$-adic realization of the Chow motive $\sH^k_\QQ$.
Also let $\sH^k_{\ZZ_p}$, the \'etale $\ZZ_p$-sheaf associated to $\mathrm{TSym}^k\ZZ^2\tensor \det^{-k}$, the minimal admissible lattice in $(\Sym^k)^\ast = \Sym^k\tensor \det^{-k}$.
We have $\sH^k_{\QQ_p} = \sH^k_{\ZZ_p}\tensor_{\ZZ_p}\QQ_p$.
For any open compact subgroup $K^{\GL_2}\subset G(\AA_f)$, we have an \'etale realization map
\[
r_\et \colon H^1_\mot(Y_{\GL_2}(K^{\GL_2}), \sH^k_\QQ(1)) \rightarrow H^1_\et(Y_{\GL_2}(K^{\GL_2}), \sH^k_{\QQ_p}(1)).
\]


\begin{thm}\label{thm:intEis}
 If $k\geq1$ is an integer and $e$ is a positive integer prime to $6p$, then for each neat compact open subgroup $K\subset \GL_2(\AA_f^{(ep)} \times \prod_{\ell\mid ep} \ZZ_\ell)$, there is a map 
\[
 {}_{ep}\cS(\AA_f^2,\ZZ_p)^K \rightarrow H^1_\et(Y_{\GL_2}(K),\sH^k_{\ZZ_p}(1)),
 \phi\mapsto{}_e\Eis_{\et,\phi}^k
\]
so that ${}_e\Eis_{\et,\phi}^k = (e^2-e^{-k}\begin{pmatrix}e&\\&e\end{pmatrix}^{-1})r_\et(\Eis_{\mot,\phi}^k)$
as elements in $H^1_\et(Y_{\GL_2}(K),\sH^k_{\QQ_p}(1))$.
Again to unify notations, for $k=0$ and $\phi\in {}_e\cS_0(\AA_f^2,\ZZ)$, we put ${}_e\Eis_{\et,\phi}^0\defeq r_\et({}_eg_\phi)$.
\end{thm}
\begin{proof}
See \cite[Proposition 7.2.4]{LSZ}.
\end{proof}

The next few theorems tell us that to understand the map of Eisenstein symbol in terms of $\GL_2(\AA_f)$-representations, it suffices to consider the principal series of $\GL_2$.
\begin{defn}
Let $k\geq0$ be an integer.
Given a (finite-order) character $\nu\colon \QQ_{>0}^\times\backslash\AA_f^\times\rightarrow \CC^\times$ such that $\nu(-1)=(-1)^k$, let $I(|\cdot|^{k+1/2}\nu, |\cdot|^{-1/2})$ be the space of smooth functions 
    $f\colon\mathrm{GL}_2(\AA_f) \rightarrow \mathbb{C}$
    such that
    \[
    f(\begin{pmatrix}
     a & b \\ & d \\
    \end{pmatrix} g)
    =
    |a|^{k+1}\nu(a)|d|^{-1} f(g).
    \]
The space $I(|\cdot|^{k+1/2}\nu, |\cdot|^{-1/2})$ is equipped with a $\mathrm{GL}_2(\AA_f)$-action of right translation.
cf. Definition \ref{defn:PP_GL2}.
\end{defn}

\begin{thm} \label{Eis&PP0}
 There is a $\GL_2(\AA_f)$-equivariant isomorphism
\[
 \partial_0 \colon \cO^\times(Y_{\GL_2})/(\QQ^\text{ab})^\times \tensor \CC \rightarrow I^0(|\cdot|^{1/2},|\cdot|^{-1/2}) \directsum \Directsum_{\nu\neq1} I(|\cdot|^{1/2}\nu, |\cdot|^{-1/2}).
\]
Here $I^0(|\cdot|^{1/2},|\cdot|^{-1/2})$ is the kernel of $I(|\cdot|^{1/2},|\cdot|^{-1/2}) \rightarrow \CC, f\mapsto \int_{\GL_2(\AA_f)/B(\AA_f)} f \; dg$, a codimension one sub-representation of $I(|\cdot|^{1/2},|\cdot|^{-1/2})$.
\end{thm}
\begin{proof}
See \cite[Theorem 7.3.2]{LSZ}.
\end{proof}

\begin{thm} \label{Eis&PPk}
 For an integer $k\geq1$, there is a $\GL_2(\AA_f)$-equivariant surjective map
\[
 \partial_k\colon H^1_\mot(Y_{\GL_2},\sH^k_\QQ(1))\tensor \CC \rightarrow \Directsum_\nu I(|\cdot|^{k+1/2}\nu, |\cdot|^{-1/2}).
\]
Moreover, $\partial_k$ is an isomorphism on the image of the Eisenstein symbol map $\phi\mapsto \Eis_{\mot,\phi}^k$.
\end{thm}
\begin{proof}
See \cite[Theorem 7.3.3]{LSZ}.
\end{proof}

\begin{thm} \label{description of partial_k}
Let $k\geq0$ be an integer.
Let $\nu=\prod_\ell \nu_\ell$ be a finite-order character of $\QQ^{\times}_{>0}\backslash\AA_f^\times$ so that $\nu(-1)=(-1)^k$.
Write $\cS(\AA_f^2,\CC)^\nu\subset \cS(\AA_f^2,\CC)$ for the subspace of functions on which $\hat\ZZ^\times$ acts via $\nu$.
Then for $\phi=\prod_\ell \phi_\ell \in\cS(\AA_f^2,\CC)^\nu$ (when $k=0$, $\nu=1$ further assume that $\phi(0,0)=0$),
 \[
  \partial_k(\Eis_{\mot,\phi}^k) = \frac{2(k+1)!L(k+2,\nu)}{(-2\pi i)^{k+2}}\prod_\ell f_{\hat\phi_\ell, \chi_\ell,\psi_\ell},
 \]
where $\chi_\ell=|\cdot|^{k+1/2}\nu_\ell,\psi_\ell=|\cdot|^{-1/2}$, and $f_{\hat\phi_\ell, \chi_\ell,\psi_\ell}$ is the Siegel section defined in Proposition \ref{Siegel}.
\end{thm}
\begin{proof}
See \cite[Proposition 7.3.4]{LSZ}.
\end{proof}

\section{The symbol map} 
\label{sec:symblmap}
\subsection{Motivation: Novodvorsky's integral formula for $L$-functions}
Let $\Pi=\prod_\ell \Pi_\ell$ (resp. $\pi=\prod_\ell \pi_\ell$) be a generic irreducible admissible automorphic representation of $\GSp_4(\AA)$ (resp. $\GL_2(\AA)$), and assume $\Pi$ is cuspidal.
In \cite{Nov77} (see also \cite{Sou}), Novodvorsky constructs the automorphic $L$-function $L(s,\Pi\tensor\pi)\defeq L(s,\Pi\tensor\pi, r)$ where $r$ is the natural $8$-dimensional representation of the $L$-group of $\GSp_4\times \GL_2$, which is isomorphic to $\GSp_4(\CC)\times \GL_2(\CC)$.
The automorphic $L$-function $L(s,\Pi\tensor\pi)$ is constructed as the product of local $L$-functions $L(s,\Pi_\ell\tensor\pi_\ell)$, and the local $L$-function $L(s,\Pi_\ell\tensor\pi_\ell)$ is defined as the common denominator of minimal degree of the integrals $\sI(\varphi_{\Pi_\ell}\tensor\varphi_{\pi_\ell}, \phi,s)$ defined below.

The generic assumption we made on $\Pi$ and $\pi$ means that they admit Whittaker models.
The definition will be recalled in Section \ref{sec:tame}.
Let $\Psi$ be an unramified character of $\QQ_\ell$.
Let $W_{\varphi_{\Pi_\ell}}$ denote a Whittaker function for $\varphi_{\Pi_\ell}\in\Pi_\ell$ with respect to $\Psi$, and $W_{\varphi_{\pi_\ell}}$ denote a Whittaker function for $\varphi_{\pi_\ell}\in\pi_\ell$ with respect to $\Psi^{-1}$.

Recall that we have an embedding $ H = \GL_2\times_{\GL_1} \GL_2 \stackrel{\iota}{\inj} G = \GSp_4\times_{\GL_1}\GL_2$.
\begin{defn}[Local zeta integral] \label{def:local_zeta_int}

Let $\phi$ be a locally constant complex function on $\QQ_\ell^2$.
For $\varphi_{\Pi_\ell}\in\Pi_\ell$ and $\varphi_{\pi_\ell}\in\pi_\ell$, define
\[
 \sI(\varphi_{\Pi_\ell}\tensor\varphi_{\pi_\ell}, \phi,s) \defeq 
\int_{CN\backslash H(\QQ_\ell)} W_{\varphi_{\Pi_\ell}}(h_1,h_2) W_{\varphi_{\pi_\ell}}(h_2)f^{\phi}(h_1;\omega,s) dh.
\]
Here $h=(h_1,h_2)\in H(\QQ_\ell)$, $dh$ is the Haar measure on $H(\QQ_\ell)$ normalized so that $H(\ZZ_\ell)$ has volume $1$,
$C=\{(tI_2,tI_2) \mid t \in \QQ_\ell^{\times} \}\subset H_(\QQ_\ell)$, 
$N\subset H(\QQ_\ell)$ is the standard maximal unipotent subgroup,
$\omega_{\ell}=\omega_{\Pi_\ell}\omega_{\pi_\ell}$ is the product of the central characters of $\Pi_\ell$ and $\pi_\ell$, and
$f^\phi(h_1;\omega,s) = f_{\phi,\psi,\chi}(h_1,s)$ defined in Proposition \ref{Siegel} below, with $\psi = |\cdot|^{-1/2}, \chi = \omega_\ell^{-1}|\cdot|^{1/2}$.
With these particular $\psi, \chi$, we have $f^{\phi}(h_1;\omega,s) = \frac{|\det h_1|^s}{L(2s,\omega_\ell)} \int_{\QQ_\ell^\times}  \phi((0,x)h_1)\omega_\ell(x)\left|x\right|^{2s} d^\times x$.
\end{defn}

In fact, let $\varphi_\Pi=\prod_\ell \varphi_{\Pi_\ell}\in\Pi$ and $\varphi_\pi=\prod_\ell \varphi_{\pi_\ell}\in\pi$, then the global integral
\begin{align*}
&\int_{C(\AA_\QQ)H(\QQ)\backslash H(\AA_\QQ)} \varphi_{\Pi}(h_1,h_2) \varphi_{\pi}(h_2)E^\phi(h_1;\omega,s)\left|\det h_1\right|^s dh \\ 
= 
&\int_{(CN\backslash H)(\AA_\QQ)} W_{\varphi_{\Pi}}(h_1,h_2) W_{\varphi_{\pi}}(h_2)f^{\phi}(h_1;\omega,s) dh
\\
\end{align*}
decomposes into the product of local integrals $\sI(\varphi_{\Pi_\ell}\tensor\varphi_{\pi_\ell}, \phi,s)$ over all $\ell$.
Here $E^{\phi}(h_1;\omega,s)=\sum_{\gamma}f^{\phi}(\gamma h_1;\omega,s)$ is the Eisenstein series for $\GL_2$ of Jacquet defined in \cite[19.2]{Jac}, and $\omega=\omega_{\Pi}\omega_\pi$ is the product of the central characters of $\Pi$ and $\pi$.


\subsection{The symbol map at finite level} \label{sec:symbl_fin}
Motivated by Novodvorsky's integral formula, we ought to construct an Euler system making use of Eisenstein symbols coming from the first $\GL_2$-factor of $H$.

\begin{lem} \label{lem:VolV}
 Let $U'\subset U \subset G(\AA_f)$ be open compact subgroups, and define $V'\defeq U'\cap H(\AA_f)$ and $V\defeq U\cap H(\AA_f)$.
 We have the commutative diagram
\begin{center}
 \begin{tikzcd}
    Y_H(V')  \ar[r,"\iota_{U'}"] \ar[d,"\pr^{V'}_{V}"'] & Y_G(U') \ar[d,"\pr^{U'}_U"]\\  
    Y_H(V)  \ar[r,"\iota_{U}"']  & Y_G(U).       
 \end{tikzcd}
\end{center}
This induces maps on the cohomologies
\begin{center}
 \begin{tikzcd}
    H^1_\mot(Y_H(V'), \iota^\ast\sW_\QQ^{a,b,c,\ast}) \ar[r,"\iota_{U',\ast}"] \ar[d,"(\pr^{V'}_{V})_\ast"'] 
    & 
    H^5_\mot(Y_G(U'), \sW_\QQ^{a,b,c,\ast}(2))
    \ar[d,"(\pr^{U'}_U)_\ast"]
    \\  
    H^1_\mot(Y_H(V), \iota^\ast\sW_\QQ^{a,b,c,\ast})
    \ar[r,"\iota_{U,\ast}"']  
    & 
    H^5_\mot(Y_G(U), \sW_\QQ^{a,b,c,\ast}(2)).  
 \end{tikzcd}
\end{center}
Let $x\in H^1_\mot(Y_H(V), \iota^\ast\sW_\QQ^{a,b,c,\ast})$.
Then 
 \[
 \Vol(V') \cdot (\pr^{U'}_{U})_\ast \circ \iota_{U',\ast}(x) = \Vol(V)\cdot  \iota_{U,\ast}(x) 
 \]
 in $ H^5_\mot(Y_G(U)\sW_\QQ^{a,b,c,\ast}(2))$.
\end{lem}
\begin{proof}
 The functoriality of pushforward maps gives
 \[
 (\pr^{U'}_U)_\ast \circ \iota_{U',\ast}  =  \iota_{U,\ast} \circ  (\pr^{V'}_V)_\ast.
 \]
 Now we apply maps on both sides of the above identity to $x$.
 Since $x$ is fixed by $V$, the right hand side is 
 \[
 [V:V']\cdot \iota_{U,\ast}(x) = \frac{\Vol V}{\Vol V'}\cdot  \iota_{U,\ast}(x),
 \]
 which proves the claim.
\end{proof}

Let $U\subset G(\AA_f)$ be an open compact subgroup.
For any $U'\subset U\subset G(\AA_f)$, let $V'\defeq U'\cap H(\AA_f)$ and $V^{\prime(1)}\defeq \pr_1(V') \subset \GL_2(\AA_f)$. 
Here $\mathrm{pr}_1 \colon H \longrightarrow \GL_2$ denotes the first projection. 
Let $a,b,c\geq0$ be integers satisfying $c\leq a+b$. 
Let $r$ be an integer such that $\max(0,-a+c) \leq r \leq \min(b,c)$.
We have the following composition of maps (when $b+c-2r=0$ we need to replace $\cS$ by $\cS_0$ everywhere)
\begin{alignat*}{3}
\cS(\AA_f^2,\QQ)^{V^{\prime(1)}}
& \xrightarrow{\Eis_{\mot}^{b+c-2r}} && H^1_\mot(Y_{\GL_2}(V^{\prime(1)}), \sH^{b+c-2r}_\QQ(1)) 
\\
& \xrightarrow{(\pr_1)^\ast} && H^1_\mot(Y_H(V'), \sH^{b+c-2r}_{\QQ}(1) \boxtimes 1) 
\\
& \xrightarrow{{\mathrm br}^{[a,b,c,r]}} && H^1_{\mathrm mot}\left(Y_H(V'), \iota^{*}\left(\mathscr{W}^{a,b,c,\ast}_{\QQ}(1-a-r)[-a-r] \right) \right)
\\
& \xrightarrow{\Vol (V')\cdot\iota_{\ast}} && H^5_\mot(Y_G(U'), \mathscr{W}^{a,b,c,\ast}_{\QQ}(3-a-r))[-a-r]
\\
& \xrightarrow{(\pr^{U'}_U)_\ast}  &&H^5_\mot(Y_G(U), \mathscr{W}^{a,b,c,\ast}_{\QQ}(3-a-r))[-a-r],
\end{alignat*}
which is equal to 
\[
\Vol(V')\cdot \iota_\ast \circ \mathrm{br}^{[a,b,c,r]} \circ (\pr_1)^\ast \circ \Eis^{b+c-2r}_\mot \circ (\pr^{V'}_{V})_\ast
\]
by functoriality of pushforward maps.

By Lemma~\ref{lem:VolV}, these maps are compatible under $\cS(\AA_f^2, \QQ)^{V^{\prime\prime(1)}}\supset \cS(\AA_f^2, \QQ)^{V^{\prime(1)}}$ for $U''\subset U'$ and hence gives rise to 
\[
\Symbl^{[a,b,c,r]}_U\colon \cS(\AA_f^2,\QQ)=\varinjlim_{U'} \cS(\AA_f^2,\QQ)^{V^{\prime(1)}} \rightarrow H^5_\mot(Y_G(U), \mathscr{W}^{a,b,c,\ast}_{\QQ}(3-a-r))[-a-r].
\]
Moreover, $\Symbl^{[a,b,c,r]}_U=(\pr^{U'}_U)_\ast \circ \Symbl^{[a,b,c,r]}_{U'}$ by construction. 

\subsection{Hecke Algebra} \label{subsec:Hecke}
Let $\cH_\ell(G)$ be the Hecke algebra of locally constant, compactly supported $\CC$-valued functions on $G(\QQ_\ell)$.
Fix a left-invariant Haar measure $dg$ on normalized so that $\Vol G(\ZZ_\ell)=1$. 
The product structure on $\cH_\ell(G)$ is given by convolution: given $\xi_1,\xi_2\in \cH_\ell(G)$, 
\[
 (\xi_1 \xi_2)(\--) \defeq \int_{G(\QQ_\ell)} \xi_1(g)\xi_2(g^{-1}\--) \; dg. 
\]

It is well-known that if $\sigma$ is a smooth representation of $G(\QQ_\ell)$, then $\sigma$ can also be viewed as a left $\cH_\ell(G)$-module: given $\varphi\in \sigma$ and $\xi\in \cH_\ell(G)$, 
\[
   \xi \cdot \varphi\defeq \int_{G(\QQ_\ell)} \xi(g) (g \cdot \varphi) \; dg. 
\]
We view $\cH_\ell(G)$ as a left $G(\QQ_\ell)$-representation and a right $H(\QQ_\ell)$-representation by
\[
 (g\cdot \xi\cdot h)(\--) = \xi(h \-- g),
\]
where $g\in G(\QQ_\ell), h\in H(\QQ_\ell)$ and $\xi\in\cH_\ell(G)$.
This makes $\cH_\ell(G)$ into a left $\cH_\ell(G)$-module and a right $\cH_\ell(H)$-module, but note that the left $\cH_\ell(G)$-module structure is given by $(\xi_1\cdot\xi_2)(\--) = \int_{G(\QQ_\ell)} \xi_1(g) \xi_2 (\-- g) \; dg$ rather than the one given by convolution on the left. 

The reason for this choice of left $G(\QQ_\ell)$-action (rather than $(g\cdot \xi)(\--) = \xi(g^{-1}\--)$) will be clear in Lemma \ref{U'equiv}.

\begin{lem}\label{dual Hecke}
Let $\sigma$ be a smooth representation of $G(\mathbb{Q}_\ell)$, and $\sigma^\vee$ its contragredient.
For every $\xi\in\mathcal{H}_\ell(G)$, let $\xi'\in\mathcal{H}_\ell(G)$ be the function defined by $\xi'(g) = \xi(g^{-1})$. 
Then for any $\Phi \in \sigma^\vee$ and $\varphi \in \sigma $, we have
\[
\Phi(\xi\cdot\varphi) = (\xi' \cdot \Phi)(\varphi).
\]
\end{lem}
\begin{proof}
By linearity of $\Phi$ we have
\[
\Phi(\xi\cdot\varphi) 
= \Phi(\int_{G(\mathbb{Q}_\ell)}\xi(g)(g\cdot \varphi) dg)
= \int_{G(\mathbb{Q}_\ell)}\xi(g)\Phi(g\cdot\varphi)dg.
\]
On the other hand,
\[
(\xi'\cdot\Phi)(\varphi) 
= \left(\int_{G(\mathbb{Q}_\ell)}\xi(g^{-1})(g\cdot\Phi)dg\right)(\varphi)
= \int_{G(\mathbb{Q}_\ell)}\xi(g^{-1})\Phi(g^{-1}\cdot\varphi)dg,
\]
which equals $\Phi(\xi\cdot\varphi)$ as required.
\end{proof}

Let $\tau$ be a smooth representation of $\GL_2(\QQ_\ell)$, and $\sigma$ be a smooth representation of $G(\QQ_\ell)$.
Regard $\tau$ as a smooth representation of $H(\QQ_\ell)$ through the first projection $\pr_1 \colon H(\QQ_\ell) \rightarrow \GL_2(\QQ_\ell)$.
Let 
\[
 \frakX(\tau,\sigma^\vee) = \Hom_{\cH_\ell(G)}(\cH_\ell(G)\tensor_{\cH_\ell(H)} \tau, \sigma^\vee).
\]

\begin{lem} \label{U'equiv}
 Let $\frakZ\in \frakX(\tau, \sigma^\vee)$ and $\xi_1,\xi_2\in \cH_\ell(G)$.
 Then 
\[
\xi_2'\cdot \frakZ(\xi_1\tensor \phi) = \frakZ((\xi_1 \xi_2) \tensor \phi). 
\]
\end{lem}
\begin{proof}
 By linearity and $G(\QQ_\ell)$-equivariance of $\frakZ$,
\begin{align*}
 \xi_2'\cdot \frakZ(\xi_1\tensor \phi)
&=\int_{G(\QQ_\ell)} \xi_2(g^{-1}) \left( g\cdot \frakZ(\xi_1\tensor \phi) \right)\; dg \\
&=\frakZ \left(\left(\int_{G(\QQ_\ell)} \xi_2(g^{-1}) ( g\cdot \xi_1) \;dg\right) \tensor \phi \right) \\
&=\frakZ \left(\left(\int_{G(\QQ_\ell)} \xi_2(g^{-1}) \xi_1(\--g ) \;dg\right) \tensor \phi \right) \\
&= \frakZ((\xi_1 \xi_2) \tensor \phi).
\end{align*}
\end{proof}

The formula in the above lemma will appear a lot in our computation with Hecke actions.
From now on, we will change notation to write $\frakZ\colon \tau\tensor_{\cH_\ell(H)}\cH_\ell(G)\rightarrow \sigma^\vee$ instead of $\frakZ\colon \cH_\ell(G)\tensor_{\cH_\ell(H)} \tau \rightarrow \sigma^\vee$ to signify our focus on the right $\cH_\ell(G)$-action on $\cH_\ell(G)$ given by right convolution (and to align our notation with that used by \cite{LSZ}.)
Hence the formula becomes
\begin{align} \label{eq:U'equiv}
    \xi_2'\cdot \frakZ(\phi \tensor \xi_1) = \frakZ(\phi\tensor \xi_1\xi_2).
\end{align}
Despite the notation, the tensor is still with respect to the left $\cH_\ell(H)$-module $\tau$ and the right $\cH_\ell(H)$-module $\cH_\ell(G)$.

\subsection{The symbol map at infinite level} \label{sec:symbl_inf}
Let $\cH(G(\AA_f))=\Tensor'_\ell \cH_\ell(G)$ be the Hecke algebra of locally constant, compactly supported $\CC$-valued functions on $G(\AA_f)$.
We view $\cS(\AA_f^2,\QQ)$ as a left $\cH(H(\AA_f))$-module through the first projection $\pr_1\colon H \rightarrow \GL_2$ and extend $\cH(G(\AA_f))$-linearly the symbol maps $\Symbl_U^{[a,b,c,r}$ at different level to define
\[
 \Symbl^{[a,b,c,r]} \colon  \cS(\AA_f^2,\QQ)\tensor_{\cH(H(\AA_f))} \cH(G(\AA_f))  \rightarrow  H^5_\mot(Y_G, \mathscr{W}^{a,b,c,\ast}_{\QQ}(3-a-r))[-a-r].
\]
We want to have an explicit description of $\Symbl^{[a,b,c,r]}$ as what we had for the finite level $\Symbl^{[a,b,c,r]}_U$.
First, because of the relation $\Symbl^{[a,b,c,r]}_U=(\pr^{U'}_U)_\ast \circ \Symbl^{[a,b,c,r]}_{U'}$, for any open compact subgroup $U\subset G(\AA_f)$, we have
\[
\Symbl^{[a,b,c,r]}(\phi \otimes \ch(U)) =\Symbl^{[a,b,c,r]}_U(\phi).
\]
Then because of the $\cH(G(\AA_f))$-equivariance of $\Symbl^{[a,b,c,r]}$, for generators $\ch(gU)\in \cH(G(\AA_f))^U$, we have 
\begin{align*}
\Symbl^{[a,b,c,r]}(\phi\tensor\ch(gU)) 
&=\Symbl^{[a,b,c,r]}(\phi\tensor\frac1{\Vol U}\ch(gUg^{-1})\ch(gU))\\
&=\frac1{\Vol U}\ch(Ug^{-1})\cdot \Symbl^{[a,b,c,r]}(\phi\tensor \ch(gUg^{-1}))\\
&=\frac1{\Vol U}\ch(g^{-1}\cdot gUg^{-1})\cdot \Symbl^{[a,b,c,r]}(\phi\tensor \ch(gUg^{-1}))\\
&=g^{-1}\cdot \Symbl^{[a,b,c,r]}(\phi\tensor \ch(gUg^{-1}))\\
&=g^{-1} \cdot \Symbl^{[a,b,c,r]}_{gUg^{-1}}(\phi).
\end{align*}
Here the second equality uses the formula (\ref{eq:U'equiv}). 

\subsection{Integrality of the symbol map} \label{sec:symbl_int}
We examine more carefully the integralitiy of the symbol map. 
First we need to define a variant of the $\Symbl^{[a,b,c,r]}$. 
Let $e$ be an integer prime to $6p$.
Let $U\subset G(\AA_f^{(ep)}\times \prod_{\ell\mid ep}\ZZ_\ell)$ be an open compact subgroup.
For any $U'\subset U\subset G(\AA_f)$, let $V'\defeq U'\cap H(\AA_f)$ and $V^{\prime(1)}\defeq \pr_1(V') \subset \GL_2(\AA_f)$.
Again we can consider the composition of maps (when $b+c-2r=0$ we replace $\cS$ by $\cS_0$ everywhere)
\begin{alignat*}{3}
{}_{ep}\cS(\AA_f^2,\ZZ_p)^{V^{\prime(1)}}
& \xrightarrow{{}_e\Eis^{b+c-2r}_\et} &&H^1_\et(Y_{\GL_2}(V^{\prime(1)}), \sH^{b+c-2r}_{\ZZ_p}(1)) 
\\
& \xrightarrow{(\pr_1)^\ast} && H^1_\et(Y_H(V'), \sH^{b+c-2r}_{\ZZ_p}(1) \boxtimes 1) 
\\
& \xrightarrow{{\mathrm br}^{[a,b,c,r]}} && H^1_\et\left(Y_H(V'), \iota^{*}\left(\mathscr{W}^{a,b,c,\ast}_{\ZZ_p}(1-a-r)[-a-r] \right) \right)
\\
& \xrightarrow{\Vol(V')\cdot \iota_{*}} && \Vol(V')\cdot H^5_\et(Y_G(U'), \mathscr{W}^{a,b,c,\ast}_{\ZZ_p}(3-a-r))[-a-r]
\\
& \xrightarrow{(\pr^{U'}_U)_\ast}  && \Vol(V')\cdot 
H^5_\et(Y_G(U), \mathscr{W}^{a,b,c,\ast}_{\ZZ_p}(3-a-r))[-a-r], 
\end{alignat*}
which gives rise to 
\[
{}_e\Symbl^{[a,b,c,r]}_U \colon {}_{ep} \cS(\AA_f^2,\ZZ_p) \rightarrow H^5_\et(Y_G(U), \mathscr{W}^{a,b,c,\ast}_{\QQ_p}(3-a-r))[-a-r]. 
\]
Note that the integrality of ${}_e\Symbl^{[a,b,c,r]}_U(\phi)$ is controlled by $\Vol V'$, where $V'=U'\cap H(\AA_f)$ for some $U'\subset U$ and $V'$ fixes $\phi$.

Let $\cH_{\ZZ_p}(G(\AA_f^{(p)}\times \ZZ_p))$ be the algebra under convolution of locally constant, compactly supported $\ZZ_p$-valued functions on $G(\AA_f^{(p)}\times \ZZ_p)$.
Again we can extend $\cH_{\ZZ_p}(G(\AA_f^{(p)}\times \ZZ_p))$-linearly to define
\[
{}_e \Symbl^{[a,b,c,r]} \colon
 {}_{ep}\cS(\AA_f^2, \ZZ_p) \tensor \cH_{\ZZ_p}(G(\AA_f^{(p)}\times\ZZ_p)) 
\rightarrow H^5_\et(Y_G, \sW_{\QQ_p}^{a,b,c,\ast}(3-a-r))[-a-r].
\]
Namely, ${}_e\Symbl^{[a,b,c,r]}(\phi\tensor\ch(U)) \defeq {}_e\Symbl^{[a,b,c,r]}_U(\phi)$,
and $\cH_{\ZZ_p}(G(\AA_f^{(p)})\times \ZZ_p)$-equivariance gives
\begin{align*}
{}_e\Symbl^{[a,b,c,r]}(\phi\tensor\ch(gU)) = g^{-1}\cdot {}_e\Symbl^{[a,b,c,r]}_{gUg^{-1}}(\phi).
\end{align*}
In particular, if $\phi$ is fixed by $gUg^{-1}\cap H(\AA_f)$, then the integrality of ${}_e\Symbl^{[a,b,c,r]}(\phi\tensor\ch(gU))$ is controlled by $\Vol(gUg^{-1}\cap H(\AA_f))$.
We summarize this in the following Proposition:
\begin{prop} \label{prop:formalintegrality}
Let $\xi=\sum_\eta m_\eta \ch(\eta U) \in \cH_{\ZZ_p}(G(\AA_f^{(p)}\times \ZZ_p))$ and $\phi\in {}_{ep} \cS(\AA_f^2,\ZZ_p)$. 
If $m_\eta \cdot \Vol(\eta U\eta^{-1}\cap H(\AA_f))\in \ZZ_p$ for all $\eta$, then 
\[
{}_e\Symbl^{[a,b,c,r]}(\phi\tensor\xi) \in H^5_\et(Y_G(U),\sW_{\QQ_p}^{a,b,c,\ast}(3-a-r))[-a-r]
\] 
lies in the image of the cohomology $H^5_\et(Y_G(U),\sW_{\ZZ_p}^{a,b,c,\ast}(3-a-r))[-a-r]$ with integral coefficient.
\end{prop}

By Theorem~\ref{thm:intEis}, ${}_e\Symbl^{[a,b,c,r]}$ and $\Symbl^{[a,b,c,r]}$ are related by
\[
 {}_e\Symbl^{[a,b,c,r]}(\phi\tensor \xi)
= r_\et \circ \Symbl^{[a,b,c,r]} \left( 
(e^2-e^{-(b+c-2r)}
(\begin{pmatrix} e &\\ &e  \end{pmatrix}, \begin{pmatrix} 1 & \\&1 \end{pmatrix})^{-1}) \phi\tensor \xi\right)
\]
as elements in $H^5_\et(Y_G(U),\sW_{\QQ_p}^{a,b,c,\ast}(3-a-r))[-a-r]$,
for any $\xi\tensor \phi\in \cH_{\ZZ_p}(G(\AA_f^{(p)}\times\ZZ_p)) 
\tensor {}_{ep}\cS(\AA_f^2, \ZZ_p)$.

We will be choosing suitable $\phi\tensor \xi$'s, so that their image under the symbol map give the Euler system elements.
In the next two sections, we do some local representation theory, in order to know what $\phi\tensor \xi$ to choose to make norm relations of Euler system elements satisfied.

\section{Local Representation Theory}
\label{sec:rep theory}
We recall some results of local representation theory that will be needed to verify the norm relations of an Euler system. 
The main references are \cite{Ren10} and Chapter 4 of \cite{Bump}. Section 3 of \cite{LSZ} and Section 1 of \cite{Grossi} also have excellent summaries of the materials, and we follow their accounts closely.

\subsection{Principal series of $\mathrm{GL}_2(\mathbb{Q}_\ell)$}
Let $|\cdot|$ denote the standard $\ell$-adic norm on $\mathbb{Q}_\ell$ normalized such that $|\ell| = \ell^{-1}$. For a smooth character $\chi$ of $\mathbb{Q}_\ell^\times$, we write $L(\chi,s)$ for the local $L$-factor defined as
\[
L(\chi, s) = L(\chi|\cdot|,0) = 
\begin{cases}
(1-\chi(\ell)\ell^{-s})^{-1} & \mathrm{if}\ \chi|_{\mathbb{Z}_\ell^\times} =1, \\
1 & \mathrm{otherwise}.
\end{cases}
\]

\begin{defn} \label{defn:PP_GL2}
    Given two smooth characters $\chi$ and $\psi$ of $\mathbb{Q}_\ell^\times$, let $I(\chi,\psi)$ be the space of smooth functions 
    $f:\mathrm{GL}_2(\mathbb{Q}_\ell) \rightarrow \mathbb{C}$
    such that
    \[
    f(\begin{pmatrix}
     a & b \\ & d \\
    \end{pmatrix} g)
    =
    \chi(a)\psi(d)|a/d|^{1/2} f(g),
    \]
    equipped with a left  $\mathrm{GL}_2(\mathbb{Q}_\ell)$-action of right translation.
\end{defn}
This is the normalized induction from the standard Borel group. It is well-known that if $\chi/\psi \neq |\cdot|^{\pm 1}$, then $I(\chi,\psi)$ is an irreducible representation. If $\chi/\psi = |\cdot|$, then $I(\chi,\psi)$ has an irreducible codimension one invariant subspace. If $\chi/\psi = |\cdot|^{-1}$, then $I(\chi,\psi)$ has a one-dimensional invariant subspace and the quotient representation is irreducible.

Write $dx$ (respectively, $d^\times x$, $dg$) for the Haar measure on $\mathbb{Q}_\ell$ (respectively, $\mathbb{Q}_\ell^\times$, $\mathrm{GL}_2(\mathbb{Q}_\ell)$) normalized so that $\mathbb{Z}_\ell$ (respectively, $\mathbb{Z}_\ell^\times$, $\mathrm{GL}_2(\mathbb{Z}_\ell)$) has unit volume.
There is a natural pairing
$ I(\chi,\psi) \times I(\chi^{-1},\psi^{-1}) \rightarrow \mathbb{C}$
defined by
\[
\langle f_1, f_2 \rangle = 
\int_{\mathrm{GL}_2(\mathbb{Z}_\ell)} f_1(g)f_2(g) dg,
\]
under which $I(\chi^{-1},\psi^{-1})$ is identified with the dual of $I(\chi,\psi)$.

\begin{defn}
Let $\chi$ and $\psi$ be smooth characters of $\mathbb{Q}_\ell$.
A \emph{flat section} of the family of representations 
$I(\chi|\cdot|^{s_1}, \psi|\cdot|^{s_2})$ indexed by $s_1, s_2 \in \mathbb{C}$ is a function
$\mathrm{GL}_2(\mathbb{Q}_\ell)\times \mathbb{C}^2 \rightarrow \mathbb{C},\ 
(g,s_1,s_2) \mapsto f_{s_1,s_2}(g)$ such that for all fixed $s_1, s_2$, the function
$g \mapsto f_{s_1,s_2}(g)$ is in $I(\chi|\cdot|^{s_1}, \psi|\cdot|^{s_2}))$,
and the restriction of $f_{s_1,s_2}$ to $\mathrm{GL}_2(\mathbb{Z}_\ell)$ is independent of $s_1$ and $s_2$.
\end{defn}
\begin{rmk}
From the Iwasawa decomposition (Proposition 4.5.2 of \cite{Bump}), one sees that every $f \in I(\chi,\psi)$ extends to a unique flat section.
\end{rmk}

\subsection{Siegel sections}
We follow the account of Section 3.2 of \cite{LSZ}, See also Section 1.5 of \cite{Grossi}.

\begin{defn}
    Let $\mathcal{S}(\mathbb{Q}_\ell^2,\mathbb{C})$ be the space of Schwartz functions on $\mathbb{Q}_\ell^2$, and $\GL_2(\QQ_\ell)$ acts on $\mathcal{S}(\QQ_\ell^2, \CC)$ by right translation. For $\phi \in \mathcal{S}(\mathbb{Q}_\ell^2,\mathbb{C})$, we use $\hat{\phi}$ to denote its Fourier transform
    \[
    \hat{\phi}(x,y) = \int\int_{\QQ_\ell^2} e_\ell(xv-yu)\phi(u,v)dudv,
    \]
    where $e_\ell(x)$ is the standard additive character of $\mathbb{Q}_\ell$ taking $1/\ell^n$ to $\exp(2\pi i/\ell^n)$.
\end{defn}

\begin{prop}[\cite{LSZ}, Proposition 3.2.2]
\label{Siegel}
Let $\phi \in \mathcal{S}(\mathbb{Q}_\ell^2,\mathbb{C})$ and $\chi,\psi$ be characters of $\mathbb{Q}_\ell^\times$. 
Then the integral 
\[
f_{\phi,\chi,\psi}(g,s) \defeq \frac{\chi(\det g) |\det g|^{s+1/2}}{L(\chi/\psi,2s+1)} 
\int_{\QQ_\ell^\times} \phi((0,x)g) (\chi/\psi)(x) |x|^{2s+1} d^\times x
\]
defines a section of $I(\chi|\cdot|^s, \psi|\cdot|^{-s})$, so $f_{\phi,\chi,\psi}(g)\defeq f_{\phi,\chi,\psi}(g, 0)\in I(\chi,\psi)$ is well-defined.
Moreover, it satisfies
\[\begin{split}
f_{g\cdot\phi,\chi,\psi}(h) &= \chi(\det g)^{-1}|\det g|^{-1/2} f_{\phi,\chi,\psi}(hg), \\
f_{\widehat{g\cdot\phi},\chi,\psi}(h) &= \psi(\det g)^{-1}|\det g|^{-1/2} f_{\hat{\phi},\chi,\psi}(hg), \\
\end{split}\]
for all $g,h\in \GL_2(\QQ_\ell)$.
In particular, if $\psi = |\cdot|^{-1/2}$, the map
\[
\mathcal{S}(\mathbb{Q}_\ell^2,\mathbb{C}) \rightarrow I(\chi,\psi)
,\qquad
\phi \mapsto  f_{\hat{\phi},\chi,\psi}
\]
is $\mathrm{GL}_2(\mathbb{Q}_\ell)$-equivariant.
\end{prop}


\begin{defn}\label{phi_st} \hfill
\begin{itemize}
 \item 
For integers $s\geq 0$ and $t \geq 0$, define functions $\phi_{s,t} \in \mathcal{S}(\mathbb{Q}_\ell^2,\mathbb{C})$ by
\[
\phi_{s,t} \defeq
\begin{cases}
\ch(\ell^{s}\bZ_\ell \times (1+ \ell^{t}\bZ_\ell)) & \textrm{ if } t > 0, 
\\
\ch(\ell^{s}\bZ_\ell \times \bZ_\ell^{\times}) & \textrm{ if } s > 0 \textrm{ and } t=0, 
\\
\ch(\bZ_\ell \times \bZ_\ell) & \textrm{ if } s=t=0. 
\end{cases}
\]
\item 
For integers $s\geq t\geq 0$, define open compact subgroups $K_1^{\GL_2}(\ell^s,\ell^t)$ of $\GL_2(\QQ_\ell)$ by 
\[
K_1^{\GL_2}(\ell^s,\ell^t) \defeq \left\{\begin{pmatrix} a &b \\c& d \end{pmatrix}\in \GL_2(\bZ_\ell)\colon c \equiv 0 \bmod \ell^s, d \equiv 1 \bmod \ell^t \right\}.
\] 
Note that $K_1^{\GL_2}(\ell^s,\ell^t)$ is the stabilizer in $\GL_2(\QQ_\ell)$ of $\phi_{s,t}$ (when $s \geq t$)  and that 
\[
\Vol(K_1^{\GL_2}(\ell^s,\ell^t)) = 
\begin{cases}
\ell^{2-s-t}(\ell^{2}-1)^{-1}  & \textrm{ if } t >0, 
\\
\ell^{1-s}(\ell + 1)^{-1} & \textrm{ if } s > t = 0, 
\\
1 & \textrm{ if } s = t = 0.  
\end{cases}
\]

\item
For integers $s\geq t\geq 0$, define open compact subgroups $K^{H}_{1}(\ell^{s}, \ell^{t})$ of $H(\QQ_\ell)$ by 
\[
K^{H}_{1}(\ell^{s}, \ell^{t}) \defeq (K^{\mathrm{GL}_{2}}_{1}(\ell^{s}, \ell^{t}) \times \mathrm{GL}_{2}(\bZ_{\ell})) \cap H(\bZ_{\ell}). 
\]
Note that $K_1^H(\ell^s,\ell^t)$ is the stabilizer of $\phi_{s,t}$ in $H(\QQ_\ell)$ of $\phi_{s,t}$ (when $s \geq t$).

\end{itemize}
\end{defn}

\begin{lem}[\cite{LSZ}, Lemma 3.2.5] \label{Siegel-value}
 Let $\chi, \psi$ be unramified characters on $\bQ_\ell^\times$. 
 Denote by $B^{\GL_2}\subset \mathrm{GL}_2$ the standard Borel subgroup of upper triangular matrices.
 Then the function $f_{\phi_{t,0},\chi,\psi}$ is supported on $B^{\GL_2}(\mathbb{Q}_\ell)K_1^{\GL_2}(\ell^t,1)$, and
    \[
    f_{\phi_{t,0},\chi,\psi}(\id_2) = 
    \begin{cases}
    L(\chi/\psi,1)^{-1} & \mathrm{if} \ t > 0, \\
1 & \mathrm{if} \ t = 0.
    \end{cases}
    \]
\end{lem}

\subsection{Principal series of $\mathrm{GSp}_4(\mathbb{Q}_\ell)$}
Next we collect some basic results on principal series representations of $\mathrm{GSp}_4(\mathbb{Q}_\ell)$. We follow closely the account of \cite{LSZ}, Section 3.5. The standard reference is \cite{RS-GSp4}.

\begin{defn}
Let $\chi_1,\chi_2,\rho$ be smooth characters of $\mathbb{Q}_\ell^\times$ such that
\begin{equation}\label{character assumption}
|\cdot|^{\pm1} \notin \{\chi_1, \chi_2,\chi_1\chi_2,\chi_1/\chi_2\}.
\end{equation}
We define $\chi_1\times\chi_2\rtimes\rho$ to be the representation of $\mathrm{GSp}_4(\mathbb{Q}_\ell)$ given by the space of smooth functions $f\colon \mathrm{GSp}_4(\mathbb{Q}_\ell) \rightarrow \mathbb{C}$ satisfying
\[
f(\begin{pmatrix}
a & * & * & * \\ & b & * & * \\ & & cb^{-1} & * \\ &&& ca^{-1}
\end{pmatrix} g)
=
\frac{|a^2b|}{|c|^{3/2}} \chi_1(a) \chi_2(b) \rho(c) f(g),
\]
with $\mathrm{GSp}_4(\mathbb{Q}_\ell)$ acting by right translation.
The representation $\chi_1\times\chi_2\rtimes\rho$ is called an irreducible principal series.
\end{defn}

\begin{rmk}
This representation has central character $\chi_1\chi_2\rho^2$. The condition on $\chi_1,\chi_2$ implies it is irreducible and generic. In fact, this is the only type among the 11 groups of irreducible, admissible, non-supercuspidal representations of $\mathrm{GSp}_4(\mathbb{Q}_\ell)$ which is both generic and spherical (see table A.1 of \cite{RS-GSp4} and table 3 of \cite{RS-spherical}).
\end{rmk}

If $\eta$ is a smooth character of $\mathbb{Q}_\ell^\times$, we may regard it as a character of $\mathrm{GSp}_4(\mathbb{Q}_\ell)$ via the multiplier map. Then twisting $\chi_1\times \chi_2 \rtimes \rho$ by $\eta$ results in the representation $\chi_1\times \chi_2 \rtimes \rho\eta$.

\begin{lem}\label{twist equivalence}
Let $\sigma = \chi_1\times \chi_2 \rtimes \rho$ be an irreducible principal series as above, and $\eta$ a smooth character of $\mathbb{Q}_\ell^\times$. Then the twist $\sigma\otimes\eta$ of $\sigma$ by $\eta$ is equivalent to $\eta$ if and only if at least one of the following conditions is satisfied:
\begin{itemize}
    \item $\eta = 1$;
    \item $\eta = \chi_1$ and $\chi_1^2 = 1$;
    \item $\eta = \chi_2$ and $\chi_2^2 = 1$;
    \item $\eta = \chi_1\chi_2$ and $\chi_1^2 = \chi_2^2 = 1$.
\end{itemize}
\end{lem}
\begin{proof}
By Theorem 4.2 of \cite{ST93}, $\sigma\otimes\eta = \chi_1\times\chi_2\rtimes\rho\eta$ is isomorphic to $\sigma$ if and only if $(\chi_1,\chi_2,\rho)$ and $(\chi_1,\chi_2,\rho\eta)$ are in the same orbit of the Weyl group acting on characters of the diagonal torus. Using the explicit representatives of the 8-element Weyl group as given in Section 2.1 of \cite{RS-GSp4}, we see $(\chi_1,\chi_2,\rho\eta)$ must be one of the following:
$(\chi_1,\chi_2,\rho)$, $(\chi_2,\chi_1,\rho)$, $(\chi_1^{-1},\chi_2,\rho\chi_1)$, $(\chi_2,\chi_1^{-1},\rho\chi_1)$,
$(\chi_1,\chi_2^{-1},\rho\chi_2)$, $(\chi_2^{-1},\chi_1,\rho\chi_2)$, $(\chi_1^{-1},\chi_2^{-1},\rho\chi_1\chi_2)$, $(\chi_2^{-1},\chi_1^{-1},\rho\chi_1\chi_2)$.
The desired claim then follows immediately.
\end{proof}

\begin{defn}
Let $\sigma = \chi_1\times \chi_2 \rtimes \rho$ be an irreducible principal series as defined above. The local spin $L$-factor of $\sigma$ is defined as
\[
L(s,\sigma) = L(\sigma\otimes|\cdot|^s,0)
=L(\rho,s)L(\rho\chi_1,s)L(\rho\chi_2,s)L(\rho\chi_1\chi_2,s).
\]
\end{defn}

We record the following characterization of irreducible generic unramified representations.

\begin{prop}[\cite{LSZ}, Proposition 3.5.3; see also \cite{RS-GSp4}, Section 2.2)]
Let $\sigma = \chi_1\otimes \chi_2 \rtimes \rho$ be an irreducible principal series. Then $\sigma$ is unramified if and only if 
all three characters $\chi_1, \chi_2, \rho$ are all unramified. 
Moreover, every irreducible, generic, unramified representation of $\mathrm{GSp}_4(\mathbb{Q}_\ell)$ is isomorphic to $\chi_1\otimes \chi_2 \rtimes \rho$ for a unique Weyl-group orbit of unramified characters $(\chi_1,\chi_2,\rho)$ satisfying (\ref{character assumption}).
\end{prop}

\subsection{Multiplicity one} \label{sec:multione}
Let $\tau$ be a smooth representation of $\GL_2(\QQ_\ell)$, and $\sigma$ be a smooth representation of $G(\QQ_\ell)$.
Regard $\tau$ as a smooth representation of $H(\QQ_\ell)$ through the first projection $\pr_1 \colon H(\QQ_\ell) \rightarrow \GL_2(\QQ_\ell)$.
We defined in Section \ref{subsec:Hecke}
\[
 \frakX(\tau,\sigma^\vee) = \Hom_{\cH_\ell(G)}(\tau\tensor_{\cH_\ell(H)}\cH_\ell(G) , \sigma^\vee).
\]
As before, despite the notation, the tensor is with respect to the left $\cH_\ell(H)$-module $\tau$ and the right $\cH_\ell(H)$-module $\cH_\ell(G)$.

\begin{prop} 
\label{frakZ&frakz}
There is a canonical bijection between $\frakX(\tau,\sigma^\vee)$ and $\Hom_H \left( \tau \tensor \left.\sigma\right|_H, \CC \right)$.
More precisely, if $\frakZ\in\frakX(\tau,\sigma^\vee)$ corresponds to $\frakz\in \Hom_H(\tau \tensor \left.\sigma\right|_H, \CC)$, then
\[
 \frakZ(\xi \tensor \phi)(\varphi) = \frakz(\phi\tensor (\xi\cdot\varphi)).
\]
\end{prop}
\begin{proof}
Note that $\tau \tensor_{\cH_\ell(H)} \cH_\ell(G)$ is the left $\cH_\ell(G)$ module corresponding to the smooth $G(\QQ_\ell)$-representation $\mathrm{ind}_H^G \tau$, the compact induction of $\tau$ from $H$ to $G$ (\cite[III.2.6 Th\'eor\`eme]{Ren10}).
Hence
\begin{align*}
 \frakX(\tau,\sigma^\vee) 
&=\Hom_G \left(\mathrm{ind}_H^G \tau, \sigma^\vee \right) \\
&=\Hom_G \left((\mathrm{ind}_H^G \tau) \tensor \sigma, \CC \right) \\
&=\Hom_G \left(\sigma, (\mathrm{ind}_H^G \tau)^\vee \right) \\
&=\Hom_G \left(\sigma, \Ind_H^G (\tau^\vee) \right) \\
&=\Hom_H \left(\left.\sigma\right|_H, \tau^\vee \right) \\
&=\Hom_H \left( \tau\tensor \left.\sigma\right|_H, \CC \right).
\end{align*}
Here the fourth equality is \cite[III.2.7 Th\'eor\`eme]{Ren10} and the fifth equality is the Frobenius reciprocity (\cite[III.2.5 Th\'eor\`eme]{Ren10}).
\end{proof}

\begin{thm}[Kato--Murase--Sugano, {\cite[Theorem 3.7.5]{LSZ}}] \label{thm:multi1} 
Let $\sigma'$ be an irreducible unramified principal series of $\GSp_4(\QQ_\ell)$ with central character $\omega_{\sigma'}$.
Let $(\chi_1,\chi_2)$, $(\psi_1,\psi_2)$ be pairs of unramified characters of $\QQ_\ell^\times$ satisfying $\chi_1\chi_2\psi_1\psi_2\omega_{\sigma'}=1$.
Suppose that $\chi_1/\psi_1,\chi_2/\psi_2\neq |\cdot |^{-1}$ and $\chi_1/\psi_1,\chi_2/\psi_2$ do not equal to the same quadratic character.
Then 
\[
\dim \Hom_H(I_H(\underline{\chi},\underline{\psi})\tensor \left.\sigma'\right|_{H},\CC) \leq 1,
\]
where $I_H(\underline{\chi},\underline{\psi})$ is $I(\chi_1,\psi_1)\tensor I(\chi_2,\psi_2)$ regarded as an $H$-representation under restriction $H \hookrightarrow \GL_2\times \GL_2$.
\end{thm}

\begin{cor} \label{cor:multi1}
Let $\sigma=\sigma'\tensor\tau'$ be a representation of $G(\QQ_\ell)$ 
so that $\sigma'$ and $\tau'=I(\chi_2,\psi_2)$ are irreducible unramified principal series of $\GSp_4(\QQ_\ell)$ and $\GL_2(\QQ_\ell)$, respectively.
Let $\chi_1,\psi_1$ be unramified characters of $\QQ_\ell^\times$ satisfying $\chi_1\chi_2\psi_1\psi_2\omega_{\sigma'}=1$, where $\omega_{\sigma'}$ is the central character of $\sigma'$.
Suppose that $\chi_1/\psi_1,\chi_2/\psi_2\neq |\cdot |^{-1}$ and $\chi_1/\psi_1,\chi_2/\psi_2$ do not equal to the same quadratic character.
 Then $\dim \Hom_{H(\QQ_\ell)}(I(\chi_1,\psi_1) \tensor \left.\sigma\right|_H, \CC)\leq1$.
\end{cor}
\begin{proof}
Denote $I(\chi_1,\psi_1)$ by $\tau$.
Observe that $\tau \tensor \left.\sigma \right|_H$ is equal to $(\tau\tensor\tau') \tensor (\left.\sigma'\right|_H)$, where $\tau\tensor\tau'$ is a representation of $H$ with $H$ acting on $\tau$ through the first $\GL_2$-factor, and on $\tau'$ through the second $\GL_2$-factor.
Then the multiplicity one result follows from Theorem~\ref{thm:multi1}.
\end{proof}

\section{Local formula for norm relations}
\label{sec:local formula}
In this section, we derive formula in local representation theory that will be useful later for proving norm relations of Euler system.

\subsection{Double coset operators}

\begin{defn} \label{Kmn}
For integers $m\geq0$ and $n\geq0$, define open compact subgroups $K_{m,n}$ and $B_{m,n}$ of $G(\QQ_\ell)$ by 
\begin{align*}
 K_{m,n}\defeq
 \left\{ g \in G(\QQ_\ell) \colon
g \congr 
(\begin{pmatrix}
 \ast & \ast & \ast & \ast \\ & \ast & \ast & \ast \\
 & \ast & \ast & \ast \\ & & & \ast
\end{pmatrix},
\begin{pmatrix}
 \ast & \ast \\ \ast & \ast
\end{pmatrix})
\bmod \ell^n, 
\mu(g) \congr 1 \bmod \ell^m \right\}, \\
 B_{m,n}\defeq
 \left\{ g \in G(\QQ_\ell) \colon 
g \congr 
(\begin{pmatrix}
 \ast & \ast & \ast & \ast \\ & \ast & \ast & \ast \\
 &  & \ast & \ast \\ & & & 1
\end{pmatrix},
\begin{pmatrix}
 \ast & \ast \\  & \ast
\end{pmatrix})
\bmod \ell^n, 
\mu(g) \congr 1 \bmod \ell^m \right\}. 
\end{align*}
For $i \in \{1, 2\}$, we put 
\[
K_{m,n}^{(i)} := \mathrm{pr}_i(K_{m,n}) \,\,\, \textrm{ and } \,\,\, 
B_{m,n}^{(i)} := \mathrm{pr}_i(B_{m,n}). 
\]
Here $\mathrm{pr}_i$ denotes the $i$-th projection on $G = \GSp_4 \times_{\GL_1} \GL_2$. 
\end{defn}

Consider the following matrices in $G(\QQ_\ell)$.
Note that except for $u_3,u_4$, the other $u_i$'s are in $H(\QQ_\ell)$.
\begin{align*}
u_{0} &\defeq
 (\begin{pmatrix}
  1 &  &  &  \\
   & 1 &  &  \\
   &  &  1 & \\
   &  &  & 1
 \end{pmatrix},
 \begin{pmatrix}
  1 &  \\
   & 1 
 \end{pmatrix}), 
& u_{1} &\defeq
(\begin{pmatrix}
  \ell & & &\\
  & \ell & & \\
  & & 1 & \\
  & & & 1
 \end{pmatrix},
 \begin{pmatrix}
  \ell &\\
  & 1 
 \end{pmatrix}), \\
 u_{2} &\defeq  
 (\begin{pmatrix}
  \ell^2 & & &\\
  & \ell & & \\
  & & \ell & \\
  & & & 1
 \end{pmatrix},
 \begin{pmatrix}
  \ell &\\
  & \ell
 \end{pmatrix}), 
 & u_{3} &\defeq
 (\begin{pmatrix}
 \ell^2 & & & \\
 & \ell & &\\
 & & \ell &\\
 & & & 1 
  \end{pmatrix},
 \begin{pmatrix}
  \ell^2 & \\
  & 1
 \end{pmatrix}), 
 \\
 u_{4} &\defeq
 (\begin{pmatrix}
 \ell^2 & & & \\
 & \ell^2 & &\\
 & & 1 &\\
 & & & 1 
  \end{pmatrix},
 \begin{pmatrix}
  \ell & \\
  & \ell 
 \end{pmatrix}),  
 & u_{5} &\defeq 
 (\begin{pmatrix}
  \ell^3 & & &\\
  & \ell^2 & & \\
  & & \ell & \\
  & & & 1
 \end{pmatrix} ,
 \begin{pmatrix}
  \ell^2 &\\
  & \ell 
 \end{pmatrix}), \\
 u_{6} &\defeq 
 (\begin{pmatrix}
  \ell^4 & & &\\
  & \ell^2 & & \\
  & & \ell^2 & \\
  & & & 1
 \end{pmatrix} ,
 \begin{pmatrix}
  \ell^2 &\\
  & \ell^2 
 \end{pmatrix}).  
\end{align*}

Recall from Section \ref{subsec:Hecke} that $\cH_\ell(G)$ is the Hecke algebra of locally constant, compactly supported $\CC$-valued functions on $G(\QQ_\ell)$.
\begin{defn}\label{Hecke operators}
Let $m\geq 0$ and $n\geq 1$ be integers.
(Here we assume the stronger assumption $n\geq1$ than in Definition \ref{Kmn}).

\begin{itemize}
\item For $i\in\{0,\ldots, 6\}$, define Hecke operators 
\begin{align*}
 U_i(\ell) &\defeq \ch(K_{m,n} u_i G(\ZZ_\ell))\in \cH_\ell(G), \\
 U_i^{K_{1,0}}(\ell) &\defeq \ch(K_{m,n} u_i K_{1,0})\in \cH_\ell(G) \text{ for $m\geq1$}.
\end{align*}

\item For $i=1,2,5$, also define Hecke operators 
\[
 U_i^{B_{m,n}}(\ell) \defeq \frac1{\Vol B_{m,n}} \ch (B_{m,n} u_i B_{m,n}) \in \cH_\ell(G).
\]
\end{itemize}
\end{defn}

\begin{rmk} \label{indep of coset decomp}
 Let $u =  (\begin{pmatrix}
  \ell^{a_1} & & &\\
  & \ell^{a_2} & & \\
  & & \ell^{a_3} & \\
  & & & \ell^{a_4}
 \end{pmatrix} ,
 \begin{pmatrix}
  \ell^{b_1} &\\
  & \ell^{b_2}
 \end{pmatrix})$
with $a_{1} \geq a_{2} \geq a_{3} \geq a_{4} \geq 0$.
For any integers $m\geq0,n\geq0$, we have 
\[
u^{-1} K_{m,n} u \cap G(\ZZ_\ell) \subset K_{m,n}.
\]
In particular, the canonical map 
\[
K_{m,n} u K_{m,n}/K_{m,n} \longrightarrow K_{m,n} u G(\ZZ_\ell)/G(\ZZ_\ell)
\]
is a bijection.
Hence if $n \geq 1$, the double coset $K_{m,n} u G(\ZZ_\ell)$ dose not depend on the choice of the integers $m, n$ (see Lemma \ref{double coset decomposition} below). 
Furthermore if $m \geq 1$ and $n \geq 1$, then $K_{m,n} u K_{1,0}$ is also independent of $m,n$. 
So the definition of $U_i(\ell)$ (resp. $U_i^{K_{1,0}}(\ell)$) above is independent of $m\geq0, n\geq1$ (resp. $m\geq1,n\geq1$).
\end{rmk}

\begin{defn} \label{defn:J_i}
For $i \in \{1, \ldots, 6\}$, let $J_i$ be the set of left $K_{1,1}$-coset representatives of $K_{1,1}u_i K_{1,1}$ 
to be given by Lemma \ref{double coset decomposition} below.
By Remark \ref{indep of coset decomp}, $J_i$ is also the set of left $G(\ZZ_\ell)$-coset representatives of $K_{1,1}u_i G(\ZZ_\ell)$, and the set of left $K_{1,0}$-coset representatives of $K_{1,1}u_i K_{1,0}$.
\end{defn}

In order to carry out explicit calculations using these Hecke operators, we will need the following double coset decomposition.

\begin{lem}\label{double coset decomposition}
Let $m\geq0$ and $n\geq1$ be integers.
\begin{itemize}
\item
For the $\GSp_4$-component and $K_{m,n}^{(1)}$-level, we have the double coset decomposition
\footnotesize
\begin{align*}
 K_{m,n}^{(1)}u_1^{(1)}K_{m,n}^{(1)}
&=
\bigcup_{\substack{x,y,z\in\ZZ/\ell}}
\begin{pmatrix}
 \ell &  &  x & y \\
 & \ell & z & x \\
 & & 1 &  \\
 & & & 1
\end{pmatrix}
K_{m,n}^{(1)}
\cup
\bigcup_{x,y\in\ZZ/\ell} 
\begin{pmatrix}
 \ell & x &   & y \\
 & 1 &  &  \\
 & & \ell & -x \\
 & & & 1
\end{pmatrix}
K_{m,n}^{(1)} \\
 K_{m,n}^{(1)}u_2^{(1)}K_{m,n}^{(1)}=  K_{m,n}^{(1)}u_3^{(1)}K_{m,n}^{(1)}
&=
\bigcup_{\substack{x,y\in\ZZ/\ell \\ z\in\ZZ/{\ell^2}}}
\begin{pmatrix}
 \ell^2 & \ell x & \ell y & z \\
 & \ell &  & y \\
 & & \ell & -x \\
 & & & 1
\end{pmatrix}
K_{m,n}^{(1)} \\
 K_{m,n}^{(1)}
u_{4}^{(1)}
K_{m,n}^{(1)}
&=
\bigcup_{\substack{x,y,z \in\ZZ/\ell^2}}
\begin{pmatrix}
 \ell^2 &  & x & y \\
 & \ell^2 & z & x \\
 & & 1 &  \\
 & & & 1
\end{pmatrix}
K_{m,n}^{(1)}
\cup
\bigcup_{\substack{x,y\in\ZZ/\ell \\ w \in (\bZ/\ell)^{\times} \\ z\in\ZZ/{\ell^2}}}
\begin{pmatrix}
 \ell^2 & \ell x & wx+\ell y & z \\
 & \ell & w & y \\
 & & \ell & -x \\
 & & & 1
\end{pmatrix}
K_{m,n}^{(1)} 
\\
&\qquad \cup
\bigcup_{\substack{x,y\in\ZZ/\ell^2}}
\begin{pmatrix}
 \ell^2 &  x &  & y \\
 &  1 &  &  \\
 & & \ell^2 & -x \\
 & & & 1
\end{pmatrix}
K_{m,n}^{(1)} \\
 K_{m,n}^{(1)}
u_{5}^{(1)}
K_{m,n}^{(1)}
&=
\bigcup_{\substack{w,x\in\ZZ/\ell \\ y\in\ZZ/{\ell^2} \\ z\in \ZZ/{\ell^3}}}
\begin{pmatrix}
 \ell^3 & \ell^2 x & \ell wx+\ell y & z \\
 & \ell^2 & \ell w & y \\
 & & \ell & -x \\
 & & & 1
\end{pmatrix}
K_{m,n}^{(1)}
\cup
\bigcup_{\substack{x\in\ZZ/\ell \\ y\in\ZZ/{\ell^2} \\ z\in \ZZ/{\ell^3}}}
\begin{pmatrix}
 \ell^3 & \ell y & \ell^2 x & z \\
 & \ell &  & x \\
 & & \ell^2 & -y \\
 & & & 1
\end{pmatrix}
K_{m,n}^{(1)} \\
 K_{m,n}^{(1)}
u_{6}^{(1)}
K_{m,n}^{(1)}
&=
\bigcup_{\substack{x,y\in\ZZ/\ell^{2} \\ z\in\ZZ/{\ell^4}}}
\begin{pmatrix}
 \ell^4 & \ell^{2} x & \ell^{2} y & z \\
 & \ell^{2} &  & y \\
 & & \ell^{2} & -x \\
 & & & 1
\end{pmatrix}
K_{m,n}^{(1)}
\end{align*}

\item
\normalsize
For the $\GSp_4$-component and $B_{m,n}^{(1)}$-level, we have the double coset decomposition
\footnotesize
\begin{align*}
 B_{m,n}^{(1)}u_1^{(1)}B_{m,n}^{(1)}
&=
\bigcup_{\substack{x,y,z\in\ZZ/\ell}}
\begin{pmatrix}
 \ell &  &  x & y \\
 & \ell & z & x \\
 & & 1 &  \\
 & & & 1
\end{pmatrix}
B_{m,n}^{(1)}\\
 B_{m,n}^{(1)}u_2^{(1)}B_{m,n}^{(1)}
&=
\bigcup_{\substack{x,y\in\ZZ/\ell \\ z\in\ZZ/{\ell^2}}}
\begin{pmatrix}
 \ell^2 & \ell x & \ell y & z \\
 & \ell &  & y \\
 & & \ell & -x \\
 & & & 1
\end{pmatrix}
B_{m,n}^{(1)} \\
 B_{m,n}^{(1)}u_{5}^{(1)}B_{m,n}^{(1)}
&=
\bigcup_{\substack{w,x\in\ZZ/\ell \\ y\in\ZZ/{\ell^2} \\ z\in \ZZ/{\ell^3}}}
\begin{pmatrix}
 \ell^3 & \ell^2 x & \ell wx+\ell y & z \\
 & \ell^2 & \ell w & y \\
 & & \ell & -x \\
 & & & 1
\end{pmatrix}
B_{m,n}^{(1)}
\end{align*}

\item
\normalsize
For the $\GL_2$-component we have the double coset decomposition
\begin{align*}
\GL_2(\bZ_{\ell})
 \begin{pmatrix}
 \ell & \\
 & 1	       
 \end{pmatrix}
\GL_2(\bZ_{\ell})
&=
\bigcup_{u\in\ZZ/\ell}
 \begin{pmatrix}
 \ell & u \\
 & 1	       
 \end{pmatrix}
\GL_2(\bZ_{\ell})
\cup
 \begin{pmatrix}
 1 &  \\
 & \ell	       
 \end{pmatrix}
\GL_2(\bZ_{\ell})\\
\GL_2(\bZ_{\ell})
 \begin{pmatrix}
 \ell^2 & \\ 
 & 1	       
 \end{pmatrix}
\GL_2(\bZ_{\ell})
&=
\bigcup_{u\in\ZZ/{\ell^2}}
 \begin{pmatrix}
 \ell^2 & u \\ 
 & 1	       
 \end{pmatrix}
 \GL_2(\bZ_{\ell})
\cup
\bigcup_{u\in(\ZZ/\ell)^{\times}}
 \begin{pmatrix}
 \ell & u \\
 & \ell	       
 \end{pmatrix}
 \GL_2(\bZ_{\ell})
\cup
 \begin{pmatrix}
 1 & \\
 & \ell^2	       
 \end{pmatrix}
 \GL_2(\bZ_{\ell}).
\end{align*}
\end{itemize}
\end{lem}
\begin{proof}
 For the $\GSp_4$ cases, the decompositions follow the same line of argument as in the proof of Lemma 6.1.1 of \cite{RS-GSp4}, making use of Iwahori factorization of $K_{m,n}$ (equation (2.7) of \cite{RS-GSp4}) which is available as we assumed $n\geq1$.
 The $\GL_2$ case is a well-known result with a little bit more computation.
\end{proof}

\subsection{Formula for tame norm relations} \label{sec:tame}
Let $\sigma=\sigma'\tensor\tau'$ be an irreducible admissible representation of $G(\QQ_\ell)=(\GSp_4\times_{\GL_1}\GL_2)(\QQ_\ell)$.
Let $\Psi$ be an unramified character of $\QQ_\ell$
in the sense that 
\[
\Psi(\bZ_\ell) = 1 \,\,\, \textrm{ and } \,\,\, \Psi(\ell^{-1} \cdot \bZ_\ell) \neq 1. 
\]

\begin{defn}
A Whittaker functional of $\sigma'$ (resp. $\tau'$) with respect to $\Psi$ (resp. $\Psi^{-1}$) is a linear function
\[
  W_{\sigma'}\colon \sigma' \rightarrow \CC \quad
\text{(resp.\ }    W_{\tau'}\colon \tau' \rightarrow \CC \text{)}
\]
satisfying the Whittaker relation
\[
 W_{\sigma'}(\begin{pmatrix}1&y &\ast&\ast\\ &1&x&\ast \\ &&1&-y\\&&&1 \end{pmatrix}g) = \Psi(x+y)W_{\sigma'}(g) \quad
\text{(resp.\ }  W_{\tau'}(\begin{pmatrix}1&x\\&1\end{pmatrix}g)=\Psi^{-1}(x)W_{\tau'}(g) \text{)}.
\]
\end{defn}

\begin{thm}[\cite{Shalika}]
Let $\pi$ be an irreducible admissible representation of a quasi-split reductive group over a local field.
Fix a non-degenerate character of a maximal unipotent subgroup. (This comes from $\Psi$ and $\Psi^{-1}$ in our case of $\GSp_4$ and $\GL_2$, respectively.)
Then the vector space of Whittaker functionals of $\pi$ with respect to the non-degenerate character has dimension $\leq1$. 
(One says $\pi$ is generic if it admits a nonzero Whittaker functional.) 
\end{thm}

We assume throughout that $\sigma'$ and $\tau'$ are generic. 
Let $W_{\sigma'}$ (resp. $W_{\tau'}$) be a non-zero Whittaker functional of $\sigma'$ (resp. $\tau'$). 
Then for each $\varphi_{\sigma'}\in \sigma'$ (resp. $\varphi_{\tau'}\in \tau'$), one can define a function
\begin{alignat*}{2}
 W_{\sigma',\varphi_{\sigma'}} &\colon \GSp_4(\QQ_\ell) &&\rightarrow \CC; \,\,\, g \mapsto W_{\sigma'}(g\cdot \varphi_{\sigma'}) \\
 \text{(resp. \ } W_{\tau',\varphi_{\tau'}} &\colon \GL_2(\QQ_\ell) &&\rightarrow \CC; \,\,\, g \mapsto W_{\tau'}(g\cdot \varphi_{\tau'}) \text{)}
\end{alignat*}
which by construction satisfies the Whittaker relation.
There is a left $\GSp_4(\QQ_\ell)$ (resp. $\GL_2(\QQ_\ell)$)-action on $\{W_{\sigma', \varphi_{\sigma'}}\}$ (resp. $\{W_{\tau', \varphi_{\tau'}}\}$): 
\[
 g\cdot W_{\sigma', \varphi_{\sigma'}}(\--) = W_{\sigma', \varphi_{\sigma'}}(\-- g) \quad
\text{(resp. \ }  g\cdot W_{\tau', \varphi_{\tau'}}(\--) = W_{\tau', \varphi_{\tau'}}(\-- g) \text{)}
\]  
and the canonical map $\sigma' \rightarrow \{W_{\sigma', \varphi_{\sigma'}}\}$ (resp. $\tau' \rightarrow \{W_{\tau', \varphi_{\tau'}}\}$) is an isomorphism of $\GSp_4(\QQ_\ell)$ (resp. $\GL_2$)-representations.
The space $\{W_{\sigma', \varphi_{\sigma'}}\}$ (resp. $\{W_{\tau', \varphi_{\tau'}}\}$) is called the Whittaker model of $\sigma'$ (resp. $\tau'$).

Recall from Definition \ref{def:local_zeta_int} that we have the local zeta integral 
\[
 \sI(\varphi_{\sigma'}\tensor\varphi_{\tau'}, \phi,s) \defeq 
\int_{CN\backslash H(\QQ_\ell)} W_{\sigma',\varphi_{\sigma'}}(h_1,h_2) W_{\tau',\varphi_{\tau'}}(h_2)f^\phi(h_1;\omega,s) dh.
\]
Here $h=(h_1,h_2)\in H$, $dh$ is the Haar measure on $H(\QQ_\ell)$ normalized so that $H(\ZZ_\ell)$ has volume $1$,
$C=\{(tI_2,tI_2) \mid t \in \QQ_\ell^{\times} \}\subset H(\QQ_\ell)$, 
$N\subset H(\QQ_\ell)$ is the standard maximal unipotent subgroup,
$\phi$ is a locally constant complex function on $\QQ_\ell^2$,
$\omega = \omega_{\sigma'}\omega_{\tau'}$ is the product of central characters of $\sigma'$ and $\tau'$, and
$f^\phi(h_1;\omega,s) = f_{\phi,\psi,\chi}(h_1,s)
=\frac{|\det h_1|^s}{L(2s,\omega)} \int_{\QQ_\ell^\times}  \phi((0,x)h_1)\omega(x)\left|x\right|^{2s} d^\times x$, with $\psi = |\cdot|^{-1/2}, \chi = \omega^{-1}|\cdot|^{1/2}$.

We define the following simplified integral.
\begin{defn}[{\cite[Equation (3.8)]{Sou}}]
Let $\eta$ be an unramified character of $\QQ_\ell^\times$, used for twisting.
Define
\[
 l(\varphi_{\sigma'}\tensor\varphi_{\tau'},\eta,s) \defeq 
\int_{\QQ_\ell^\times\times \QQ_\ell^\times} 
W_{\sigma',\varphi_{\sigma'}} (\begin{pmatrix}xy &&&\\ &x&&\\ &&y&\\ &&&1\end{pmatrix})
W_{\tau', \varphi_{\tau'}} (\begin{pmatrix}x&\\&y\end{pmatrix})
\left|x\right|^{s-2}
\left|y\right|^{s}
\eta(xy)
d^\times x d^\times y
\]
\end{defn}

\begin{defn}
 Denote by $\varphi_0=\varphi_{\sigma',0}\tensor\varphi_{\tau',0}$ a spherical vector of $\sigma=\sigma'\tensor \tau'$, normalized in the sense that $W_{\sigma',\varphi_{\sigma',0}}(1)W_{\tau', \varphi_{\tau',0}}(1)=1$.
\end{defn}

Let $K^H\subset H(\ZZ_\ell)$ be a compact open subgroup fixing both  $\varphi$ and $\phi$.
We have the Iwasawa decomposition $H(\QQ_\ell)=NT H(\ZZ_\ell)$, where $T\subset H(\QQ_\ell)$ is the maximal torus of diagonal matrices. 
According to the decomposition, $dh = \delta_H(t)^{-1} dn dt dk$ (\cite[Sec.\ 4.1 Eq.\ (4)(9)]{Cartier}), where $dn$ (resp. $dt, dk$ is the Haar measure on $N$ (resp. $T, H(\ZZ_\ell)$) so that $N \cap H(\ZZ_\ell)$ (resp. $T\cap H(\ZZ_\ell), H(\ZZ_\ell)$) has volume $1$, and $\delta_H$ is the modulus function of the Borel subgroup $P=NT$ of $H(\QQ_\ell)$, i.e., $\int_P f(pp_0) dp = \delta_H(p_0) \int_P f(p) dp$ for $dp$ a left-invariant Haar measure on $P$ and $f$ any locally-constant compactly-supported function.
Note that 
\[
\delta_H(\begin{pmatrix} xy &&&\\ &x&&\\ &&y& \\ &&&1\end{pmatrix}) = |xy|\cdot |x|/|y| = |x|^2.
\]
From this we can rewrite the local zeta integral 
\begin{equation} \label{simple local zeta}
\sI(\varphi, \phi,s)
= \Vol(K^H) \sum_{\gamma \in H(\ZZ_\ell)/K^H}
f^{\gamma\cdot \phi}(\id_2;\omega,s)\cdot l(\gamma\cdot\varphi,1,s).
\end{equation}

\begin{prop}
\label{Z-formula}
$l(\varphi_0,\eta,s)=L(2s,\eta^2\omega)^{-1}L(s,\sigma\tensor \eta)$.
\end{prop}
\begin{proof}
Since $\varphi_0$ is a spherical vector, we can rewrite the integral $l(\varphi_0,\eta,s)$ as
 \begin{align*}
l(\varphi_0,\eta,s) & =\sum_{i\in\ZZ}\sum_{j\in\ZZ}
W_{\sigma',\varphi_{\sigma',0}}
(\begin{pmatrix}
\ell^{i+2j}\\
 & \ell^{i+j}\\
 &  & \ell^{j}\\
 &  &  & 1
\end{pmatrix})
W_{\tau',\varphi_{\tau',0}}
(\begin{pmatrix}
\ell^{i+j} \\
 & \ell^{j}
\end{pmatrix})
\ell^{2(i+j)-(i+2j)s} \eta(\ell)^{i+2j}.
\end{align*}
The Casselman--Shalika formula says that 
\begin{align*}
W_{\sigma',\varphi_{\sigma',0}}
(\begin{pmatrix}
\ell^{i+2j}\\
 & \ell^{i+j}\\
 &  & \ell^{j}\\
 &  &  & 1
\end{pmatrix})
&= \delta^{1/2}_{\GSp_4}(\begin{pmatrix}
\ell^{i+2j}\\
 & \ell^{i+j}\\
 &  & \ell^{j}\\
 &  &  & 1
\end{pmatrix}) \cdot \chi_{V^{j,i}}(t_{\sigma'}),
&\text{for } i\geq 0, j\geq0,
\\
W_{\tau',\varphi_{\tau',0}}
(\begin{pmatrix}
\ell^{i+j}\\
 & \ell^{j}
\end{pmatrix})
&=\delta^{1/2}_{\GL_2}(\begin{pmatrix}
\ell^{i+j}\\
 & \ell^{j}
\end{pmatrix}) \cdot \chi_{\Sym^i\tensor \det^j}(t_{\tau'}),
&\text{for } i\geq0,
\end{align*}
and $0$ for $i,j$ otherwise. 
Here $\delta_{\GSp_4}$ (resp. $\delta_{\GL_2}$) is the modulus function of the Borel subgroup of $\GSp_4$ (resp. $\GL_2$);
in particular, 
$\delta_{\GSp_4}(\begin{pmatrix}
\ell^{\lambda_1}\\
 & \ell^{\lambda_2}\\
 &  & \ell^{\lambda_3}\\
 &  &  & \ell^{\lambda_4}
\end{pmatrix}) = \ell^{-3\lambda_1+2\lambda_3+\lambda_4}$
and
$\delta_{\GL_2}(\begin{pmatrix}
\ell^{\lambda_1}\\
 & \ell^{\lambda_2}
\end{pmatrix}) = \ell^{-\lambda_1+\lambda_2}
$.
Also $\chi_{V^{j,i}}$ (resp. $\chi_{\Sym^i\tensor\det^j}$) is the character of the irreducible representation of $\GSp_4$ (resp. $\GL_2$) with highest weight $(i+j)\chi_1+j\chi_2$ (resp. $(i+j)t_1+jt_2$) defined at the beginning of Section \ref{sec:alg rep}.
And $t_{\sigma'}$ (resp. $t_{\tau'}$) is the semisimple conjugacy class in the $L$-group of $\GSp_4$ (resp. $\GL_2$) associated to $\sigma'$ (resp. $\tau'$.)
Note that in the above formula, we have implicitly used the fact that ${}^L\GSp_4$, the $L$-group of $\GSp_4$, is isomorphic to $\GSp_4(\CC)$ 
and similarly ${}^L\GL_2\isom \GL_2(\CC)$. 

Hence if we denote $\eta(\ell)\ell^{-s}$ by $X$, then
\[
 l(\varphi_0, \eta, s) = \sum_{i\geq0}\sum_{j\geq0} \chi_{V^{j,i}}(t_{\sigma'})\chi_{\Sym^i\tensor\det^j}(t_{\tau'}) X^{i+2j}.
\]
A similar calculation to \cite[p.139-p.140]{GPS}, except that we have a twist by $\eta$ and we use $\GSp_4$ instead of $\Sp_4$, then implies that
\[
 l(\varphi_0,\eta,s) = L(2s,\wedge^2(\tau'\tensor\eta) \tensor \omega_{\sigma'})^{-1} L(s,\sigma'\tensor\tau'\tensor \eta).
\]
Note that $\wedge^2(\tau'\tensor \eta)\tensor\omega_{\sigma'} = \eta^2\det(\tau')\omega_{\sigma'}=\eta^2\omega$.
\end{proof}

To simplify notation, we write
\begin{align*}
W_{\sigma',\varphi_{\sigma'}}(i,j) &\defeq W_{\sigma',\varphi_{\sigma'}}
(\begin{pmatrix}
\ell^{i+2j}\\
 & \ell^{i+j}\\
 &  & \ell^{j}\\
 &  &  & 1
\end{pmatrix}) \\
W_{\tau',\varphi_{\tau'}}(i,j) &\defeq W_{\tau',\varphi_{\tau'}}
(\begin{pmatrix}
\ell^{i+j} \\
 & \ell^{j}
\end{pmatrix})
\end{align*}
and when $\varphi_{\sigma'}$ (resp. $\varphi_{\tau'}$) is the normalized spherical vector $\varphi_{\sigma',0}$ (resp. $\varphi_{\tau',0}$), we further omit the subscript of $\varphi_{\sigma'}$ (resp. $\varphi_{\tau'}$).

Breaking up the local zeta integral into double sums, we see
\begin{align*}\label{l double sum}
l(\varphi_0,\eta,s) 
&=\sum_{i \in\ZZ}\sum_{j \in\ZZ}
W_{\sigma'}(i,j) W_{\tau'}(i,j) \ell^{2(i+j)} X^{i+2j}\\
&=\sum_{i \geq 0}\sum_{j \geq 0}
W_{\sigma'}(i,j) W_{\tau'}(i,j) \ell^{2(i+j)} X^{i+2j}.
\end{align*}
More generally, we have the following:

\begin{prop}\label{U_l actions}
We have the following identities:
\begin{align*}
    l (U_1(\ell) \varphi_0, \eta, s) 
    =\  &
    \ell^2 X^{-1}\big(\sum_{i \geq 0}\sum_{j \geq 1} + \sum_{i \geq 1}\sum_{j \geq 0}\big)  W_{\sigma'}(i,j) W_{\tau'}(i,j)
    \ell^{2(i+j)} X^{i+2j} 
    \\
    & + \ell X^{-1}\sum_{i \geq 2}\sum_{j \geq 0} W_{\sigma'}(i,j) W_{\tau'}(i-2,j+1)
    \ell^{2(i+j)} X^{i+2j}  \\
    & + \ell^3 X^{-1}\sum_{i \geq 0}\sum_{j \geq 1} W_{\sigma'}(i,j) W_{\tau'}(i+2,j-1)
    \ell^{2(i+j)} X^{i+2j}, \\
    l (U_2(\ell) \varphi_0, \eta, s) 
    =\  & 
    \ell^2 X^{-2}\sum_{i \geq 0}\sum_{j \geq 1} W_{\sigma'}(i,j) W_{\tau'}(i,j)
    \ell^{2(i+j)} X^{i+2j}, \\
    l ( \left(U_{3}(\ell)+U_2(\ell)\right)  \varphi_0, \eta, s) 
    =\  &
    (\ell-1)\ell^2 X^{-2}\sum_{i \geq 1}\sum_{j \geq 1}  W_{\sigma'}(i,j) W_{\tau'}(i,j)
    \ell^{2(i+j)} X^{i+2j} 
    \\
    & + \ell^2 X^{-2}\sum_{i \geq 2}\sum_{j \geq 1} W_{\sigma'}(i,j) W_{\tau'}(i-2,j+1)
    \ell^{2(i+j)} X^{i+2j}  \\
    & + \ell^4 X^{-2}\sum_{i \geq 0}\sum_{j \geq 1} W_{\sigma'}(i,j) W_{\tau'}(i+2,j-1)
    \ell^{2(i+j)} X^{i+2j}, \\
    l (\left(U_{4}(\ell)+U_2(\ell)\right) \varphi_0, \eta, s) 
    =\  &
    (\ell-1)\ell^2 X^{-2}\sum_{i \geq 1}\sum_{j \geq 1}  W_{\sigma'}(i,j) W_{\tau'}(i,j)
    \ell^{2(i+j)} X^{i+2j} 
    \\
    & + \ell^2 X^{-2}\sum_{i \geq 2}\sum_{j \geq 0} W_{\sigma'}(i,j) W_{\tau'}(i-2,j+1)
    \ell^{2(i+j)} X^{i+2j}  \\
    & + \ell^4 X^{-2}\sum_{i \geq 0}\sum_{j \geq 2} W_{\sigma'}(i,j) W_{\tau'}(i+2,j-1)
    \ell^{2(i+j)} X^{i+2j}, \\
    l (U_5(\ell) \varphi_0, \eta, s) 
    =\  &
    \ell^4 X^{-3}\big(\sum_{i \geq 1}\sum_{j \geq 1} + \sum_{i \geq 0}\sum_{j \geq 2}\big)  W_{\sigma'}(i,j) W_{\tau'}(i,j)
    \ell^{2(i+j)} X^{i+2j} 
    \\
    & + \ell^3 X^{-3}\sum_{i \geq 2}\sum_{j \geq 1} W_{\sigma'}(i,j) W_{\tau'}(i-2,j+1)
    \ell^{2(i+j)} X^{i+2j}  \\
    & + \ell^5 X^{-3}\sum_{i \geq 0}\sum_{j \geq 2} W_{\sigma'}(i,j) W_{\tau'}(i+2,j-1)
    \ell^{2(i+j)} X^{i+2j}, \\
    l (U_6(\ell) \varphi_0, \eta, s) 
    =\  &
    \ell^4 X^{-4}\sum_{i \geq 0}\sum_{j \geq 2}  W_{\sigma'}(i,j) W_{\tau'}(i,j)
    \ell^{2(i+j)} X^{i+2j}.
\end{align*}
\end{prop}

\begin{proof}
The Proposition follows from explicit calculations of local zeta integrals. We will only show the proof for $U_{4}(\ell)+U_2(\ell)$. The others are completely analogous.

Following the explicit double coset decomposition in Lemma \ref{double coset decomposition}, we see
\[
\left(U_{4}(\ell) +U_2(\ell)\right) \varphi_0
=\sum_{\gamma \in J_4\sqcup J_2} \gamma \cdot \varphi_{0}.
\]
Here $\gamma$ runs through the disjoint union $J_4\sqcup J_2 = J_4^1 \sqcup (J_4^2 \sqcup J_2) \sqcup J_4^3$ with
\begin{align*}
J_4^1 &:= \{
(A^{x,y,z}_1, \ell \cdot I_{2}) \colon x,y,z \in \{0, \ldots, \ell^{2}-1\}\}, 
\\
J_4^2 \sqcup J_2 &:= \{
(A^{x,y,z,w}_2, \ell \cdot I_{2}) \colon x,y,w \in \{0, \ldots, \ell-1\}, z \in \{0, \ldots, \ell^{2}-1\} \},
\\
J_4^3 &:= \{
(A^{x,y}_3, \ell \cdot I_{2}) \colon x,y \in \{0, \ldots, \ell^{2}-1\}\}, 
\end{align*}
where to simplify notation (we only use the following notation in the proof) we write
\[
A^{x,y,z}_1 := \begin{pmatrix}
 \ell^2 &  & x & y \\
 & \ell^2 & z & x \\
 & & 1 &  \\
 & & & 1
\end{pmatrix}, \,\,\,  
A^{x,y,z,w}_2
:=\begin{pmatrix}
 \ell^2 & \ell x & wx+\ell y & z \\
 & \ell & w & y \\
 & & \ell & -x \\
 & & & 1
\end{pmatrix}, \,\,\,  
A_{3}^{x,y} := 
\begin{pmatrix}
 \ell^2 &  x &  & y \\
 &  1 &  &  \\
 & & \ell^2 & -x \\
 & & & 1
\end{pmatrix}
\]
and $I_{2} = \begin{pmatrix} 1 &  \\ 
& 1 
\end{pmatrix}$. 

Because the $\GL_2$-component $\mathrm{pr}_2(\gamma)$ of any $\gamma\in J_4\sqcup J_2$ is $\ell\cdot I_2$, the zeta integral $l ( \left(U_4(\ell)+U_2(\ell)\right)\varphi_0, \eta, s)$ becomes 
\begin{equation}\label{U_4 sum}
 \sum_{\gamma \in J_4 \sqcup J_2}
 \sum_{i\in \mathbb{Z}} \sum_{j\in \ZZ}
 W_{\sigma'}(
 \begin{pmatrix}
 \ell^{i+2j} &&& \\ & \ell^{i+j} && \\ && \ell^j & \\ &&&1
 \end{pmatrix} \mathrm{pr}_1(\gamma))
 W_{\tau'}(
 \begin{pmatrix}
 \ell^{i+j+1} & \\ & \ell^{j+1}
 \end{pmatrix}) 
 \ell^{2(i+j)} X^{i+2j}.
\end{equation}
\begin{itemize}
    \item 
For $\gamma \in J_4^1$, we have
\[
\begin{pmatrix}
 \ell^{i+2j} &&& \\ & \ell^{i+j} && \\ && \ell^j & \\ &&&1
 \end{pmatrix}
 A_1^{x,y,z}
=
\begin{pmatrix}
1 && \ell^{i+j}x & \ell^{i+2j}y \\ & 1 & \ell^i z & \ell^{i+j} x \\ && 1 & \\ &&&1
\end{pmatrix}
\begin{pmatrix}
 \ell^{i+2+2j} &&& \\ & \ell^{i+2+j} && \\ && \ell^j & \\ &&&1
 \end{pmatrix}
\]
so by definition of Whittaker functions
\[
W_{\sigma'}(
 \begin{pmatrix}
 \ell^{i+2j} &&& \\ & \ell^{i+j} && \\ && \ell^j & \\ &&&1
 \end{pmatrix} A_1^{x,y,z})
 =
 \Psi (\ell^i z) W_{\sigma'}(i+2,j).
\]
Note that
\[
\sum_{0 \leq z < \ell^2} \Psi(\ell^i z)
=
\begin{cases}
\ell^2 \qquad \mathrm{if} \ i \geq 0; \\
0 \qquad \ \mathrm{if} \ i = -1 \ \mathrm{or} \ -2.
\end{cases}
\]
by the unramified condition on $\Psi$.
Thus the contribution to equation (\ref{U_4 sum}) is
\begin{equation}\label{contribution 1}
\begin{split}
&\sum_{x,y,z\in\ZZ/\ell^2} \sum_{i\in\ZZ} \sum_{j\in\ZZ}  W_{\sigma'}(
 \begin{pmatrix}
 \ell^{i+2j} &&& \\ & \ell^{i+j} && \\ && \ell^j & \\ &&&1
 \end{pmatrix} A_1^{x,y,z})
  W_{\tau'}(
 \begin{pmatrix}
 \ell^{i+j+1} & \\ & \ell^{j+1}
 \end{pmatrix}) 
 \ell^{2(i+j)} X^{i+2j} \\
&= \ell^6 \sum_{i \geq 0}\sum_{j \geq 0} W_{\sigma'}(i+2,j) W_{\tau'}(i,j+1)
    \ell^{2(i+j)} X^{i+2j} \\
&= \ell^2 X^{-2}\sum_{i \geq 2}\sum_{j \geq 0} W_{\sigma'}(i,j) W_{\tau'}(i-2,j+1)
    \ell^{2(i+j)} X^{i+2j}.
\end{split}
\end{equation}
Here note that the terms with $i < -2$ in the sum on the left vanishes because $W_{\sigma'}(i+2,j) = 0$, and similarly for terms with $j<0$.

\item
For $\gamma \in J_4^2\sqcup J_2$, we have
\[
\begin{pmatrix}
 \ell^{i+2j} &&& \\ & \ell^{i+j} && \\ && \ell^j & \\ &&&1
 \end{pmatrix}
 A_2^{x,y,z,w}
=
\begin{pmatrix}
1 & \ell^j x & \ell^{i+j} (y+\frac{wx}{l}) & \ell^{i+2j} z \\ & 1 & \ell^{i-1}w & \ell^{i+j}y \\ && 1 & -\ell^j x\\ &&&1
\end{pmatrix}
\begin{pmatrix}
 \ell^{i+2j+2} &&& \\ & \ell^{i+j+1} && \\ && \ell^{j+1} & \\ &&&1
 \end{pmatrix}
\]
so we are summing
\[
\Psi(\ell^{i-1}w+\ell^j x) W_{\sigma'}(i,j+1)W_{\tau'}(i,j+1)\ell^{2(i+j)}X^{i+2j}.
\]
Note that
\[\begin{split}
\sum_{0 \leq x,w < \ell} \Psi(\ell^{i-1}w + \ell^j x)
&=
\sum_{0\leq w <\ell} \Psi(\ell^{i-1}w)
\sum_{0 \leq x < \ell} \Psi(\ell^j x) \\
&=
\begin{cases}
\ell^2-\ell &\mathrm{if} \ i \geq 1 \ \mathrm{and} \ j \geq 0; \\
0 & \mathrm{if} \ i = 0 \ \mathrm{or}\ j = -1.\\
\end{cases}
\end{split}\]
Thus the contribution to equation (\ref{U_4 sum}) is
\begin{equation}\label{contribution 2}
\begin{split}
&\sum_{\substack{x,y,w\in\ZZ/\ell\\z\in\ZZ/\ell^2}} \sum_{i\in\ZZ} \sum_{j\in\ZZ}  W_{\sigma'}(
 \begin{pmatrix}
 \ell^{i+2j} &&& \\ & \ell^{i+j} && \\ && \ell^j & \\ &&&1
 \end{pmatrix} A_2^{x,y,z,w})
  W_{\tau'}(
 \begin{pmatrix}
 \ell^{i+j+1} & \\ & \ell^{j+1}
 \end{pmatrix}) 
 \ell^{2(i+j)} X^{i+2j} \\
&= \ell^4(\ell-1) \sum_{i \geq 1}\sum_{j \geq 0} W_{\sigma'}(i,j+1) W_{\tau'}(i,j+1)
    \ell^{2(i+j)} X^{i+2j} \\
&= \ell^2(\ell-1) X^{-2}\sum_{i \geq 1}\sum_{j \geq 1} W_{\sigma'}(i,j) W_{\tau'}(i,j)
    \ell^{2(i+j)} X^{i+2j}.
    \end{split}
\end{equation}

\item
For $\gamma \in J_4^3$, we have
\[
\begin{pmatrix}
 \ell^{i+2j} &&& \\ & \ell^{i+j} && \\ && \ell^j & \\ &&&1
 \end{pmatrix}
A_3^{x,y}
=
\begin{pmatrix}
1 & \ell^j x && \ell^{i+2j} y \\ & 1 && \\ && 1 & -\ell^j x\\ &&&1
\end{pmatrix}
\begin{pmatrix}
 \ell^{i+2j+2} &&& \\ & \ell^{i+j} && \\ && \ell^{j+2} & \\ &&&1
 \end{pmatrix}
\]
so we are summing
\[
\Psi(\ell^j x) W_{\sigma'}(i-2,j+2)W_{\tau'}(i,j+1)\ell^{2(i+j)}X^{i+2j}.
\]
Note that
\[
\sum_{0 \leq x < \ell^2} \Psi(\ell^j x)
=
\begin{cases}
\ell^2 \qquad \mathrm{if} \ j \geq 0; \\
0 \qquad \ \mathrm{if} \ j = -1 \ \mathrm{or} \ -2.
\end{cases}
\]
Thus the contribution to equation (\ref{U_4 sum}) is
\begin{equation}\label{contribution 3}
\begin{split}
&\sum_{\substack{x,y,w\in\ZZ/\ell\\z\in\ZZ/\ell^2}} \sum_{i\in\ZZ} \sum_{j\in\ZZ}  W_{\sigma'}(
 \begin{pmatrix}
 \ell^{i+2j} &&& \\ & \ell^{i+j} && \\ && \ell^j & \\ &&&1
 \end{pmatrix} A_2^{x,y,z,w})
  W_{\tau'}(
 \begin{pmatrix}
 \ell^{i+j+1} & \\ & \ell^{j+1}
 \end{pmatrix}) 
 \ell^{2(i+j)} X^{i+2j} \\
&= \ell^4 \sum_{i \geq 2}\sum_{j \geq 0} W_{\sigma'}(i-2,j+2) W_{\tau'}(i,j+1)
    \ell^{2(i+j)} X^{i+2j} \\
&= \ell^4 X^{-2}\sum_{i \geq 0}\sum_{j \geq 2} W_{\sigma'}(i,j) W_{\tau'}(i+2,j-1)
    \ell^{2(i+j)} X^{i+2j}.
\end{split}
\end{equation}

\end{itemize}

Combining equations (\ref{contribution 1}), (\ref{contribution 2}) and (\ref{contribution 3}) in the above three cases immediately yields the desired identity. 
\end{proof}

\begin{defn} \label{defn:xi_s}
 For $s\in\CC$, define
\[
\xi_s(\ell)
\defeq
U_0(\ell)
-\frac{\eta(\ell)}{\ell^{s+2}}U_1(\ell)
+\frac{2\eta^2(\ell)}{\ell^{2s+3}}U_{2}(\ell)
+\frac{\eta^2(\ell)}{\ell^{2s+3}}U_{3}(\ell)
+\frac{\eta^2(\ell)}{\ell^{2s+3}}U_{4}(\ell)
-\frac{\eta^3(\ell)}{\ell^{3s+4}}U_5(\ell)
+\frac{\eta^4(\ell)}{\ell^{4s+4}}U_6(\ell).
\]
\end{defn}

\begin{cor}\label{linear combination}
$l(\xi_s(\ell)\varphi_0,\eta,s) = 1$.
\end{cor}
\begin{proof}
A direct calculation using Proposition \ref{U_l actions} shows $l(\xi_s(\ell)
\varphi_0,\eta,s) = 1$
\end{proof}

The local zeta integral evaluated at $s=0$ gives a function
\[
 \cS(\QQ_\ell^2,\CC) \tensor \sigma \rightarrow \CC \quad \phi\tensor\varphi \mapsto \sI(\varphi,\phi,0).
\]
Since $\sI(\varphi,\phi,0)$ is a finite sum of $f^{\gamma\cdot\phi}(\id_2;\omega,0) \cdot l(\gamma\cdot \varphi,1,0)$ over coset representatives of $\gamma \in H(\ZZ_\ell)/K^H$, 
it can be checked using Proposition \ref{Siegel} that 
\[
\sI(\varphi,\begin{pmatrix} a & \\ & a\end{pmatrix}\phi, 0) = \omega^{-1}(a) \sI(\varphi,\phi,0)
\]
for any scalar matrices $\begin{pmatrix}a&\\&a\end{pmatrix}\in \GL_2(\QQ_\ell)$.
Moreover, note that if $\psi=|\cdot|^{-1/2}$ and $\chi/\psi \neq |\cdot|^{-1}$, then $\phi\mapsto f_{\hat{\phi}, \chi,\psi}$ (the Siegel section defined in Proposition \ref{Siegel}) gives an isomorphism between the maximal quotient of $\cS(\QQ_\ell^2,\CC)$ on which scalar matrices act by $\chi\psi$ with $I(\chi,\psi)$ (see \cite[Proposition 3.3(b)]{Loeff_zeta} for example).
Let $\chi = \omega^{-1}|\cdot |^{1/2}, \psi = |\cdot|^{-1/2}$ and assume that
$\chi/\psi =\omega^{-1}|\cdot|\neq |\cdot|^{-1}$.
Then the map given by the local zeta integral $\sI(\varphi,\phi,0)$ descends to a map
\[
\frakz_{\chi,\psi}\colon I(\chi,\psi) \tensor \sigma \rightarrow \CC.
\]

To simplify notation, when the characters $\chi,\psi$ are clear from the context, for $\phi \in \mathcal{S}(\mathbb{Q}_\ell^2,\mathbb{C})$, we write
\[
F_{\phi} \defeq 
f_{\hat{\phi},\chi,\psi} \in I(\chi,\psi).
\]


For the remaining of Section \ref{sec:local formula}, we assume that $\sigma=\sigma'\tensor\tau'$, a representation of $G(\QQ_\ell)$, is  
so that $\sigma'$ and $\tau'=I(\chi',\psi')$ are irreducible unramified principal series of $\GSp_4(\QQ_\ell)$ and $\GL_2(\QQ_\ell)$, respectively.
Let $\chi=\omega^{-1}|\cdot|^{1/2}$ and $\psi=|\cdot|^{-1/2}$, where $\omega$ is the central character of $\sigma$.
Assume that when we write $\chi = |\cdot|^{1/2+\kappa}\nu$ with $\nu$ an unramified character, then $\kappa\geq0$ is an integer. 

\begin{prop}\label{z U action}
Let $\mathfrak{z} \in \mathrm{Hom}_{H(\mathbb{Q}_\ell)}(I(\chi,\psi)\otimes\sigma,\mathbb{C})$. Then for any integer $t \geq 1$,
\[\begin{split}
\mathfrak{z}(F_{\phi_{t,0}},\varphi_0) &= 
\frac{1}{\ell^{t-1}(\ell+1)} \biggl(1-\frac{\ell^\kappa}{\nu(\ell)}\biggr) \mathfrak{z}(F_{\phi_{0,0}}, \varphi_0), \\
\mathfrak{z}(F_{\phi_{t,0}},\xi_0(\ell)  \varphi_0) &= 
\frac{1}{\ell^{t-1}(\ell+1)}L(0,\sigma)^{-1} \mathfrak{z}(F_{\phi_{0,0}}, \varphi_0), \\
\end{split}\]
where $\phi_{t,0}$ was defined in Definition \ref{phi_st} and $\varphi_0\in \sigma$ is the normalized spherical vector.
\end{prop}
\begin{proof}
We first prove the proposition for $\frakz = \frakz_{\chi,\psi}$ coming from the local zeta integral.
By Lemma \ref{Siegel-value},
$f_{\phi_{t,0},\psi,\chi}$ restricted to $\mathrm{GL}_2(\mathbb{Z}_\ell)$ is supported on $K_1^{\GL_2}(\ell^t,1)$,
so by equation (\ref{simple local zeta}),
\begin{align*}
\frakz_{\chi,\psi}(F_{\phi_{t,0}}\tensor \varphi_0) 
&= \Vol(K_1^H(\ell^t,1)) \cdot f_{\phi_{t,0},\psi,\chi}(\id_2) \cdot l(\varphi_0,1,0) \\
&=\begin{cases}
\frac1{\ell^{t-1}(\ell+1)}L(\psi/\chi,1)^{-2} L(0,\sigma)    & \mathrm{if} \ t > 0;\\
 L(\psi/\chi,1)^{-1}  L(0,\sigma)
& \mathrm{if} \ t =0.
\end{cases}
\end{align*}
Here in the last equality we have used Lemma \ref{Siegel-value} and Proposition \ref{Z-formula}.
As 
\[
L(\psi/\chi,1)^{-1}
= 1-(\psi/\chi)(\ell)\cdot\ell^{-1}
= 1-\frac{\ell^\kappa}{\nu(\ell)},
\]
this proves the first part of the proposition.

For the second part, by Remark \ref{indep of coset decomp} each $U_i(\ell)$ is invariant under left-translation of $K_{0,1}$, so $\xi_0(\ell) \varphi_0$ is fixed by $K_1^{H}(\ell,1)= K_{0,1}\cap H(\ZZ_\ell)$.
Then a similar argument as above shows that if $t>0$, 
\begin{align*}
\mathfrak{z}_{\chi,\psi}(F_{\phi_{t,0}}\tensor\xi_0(\ell)\varphi_0)
&= \Vol(K_1^{H}(\ell^t,1))
\cdot f_{\phi_{t,0},\psi,\chi}(\id_2) \cdot l(\xi_0(\ell)\varphi_0,1,0) \\
&= \frac1{\ell^{t-1}(\ell+1)} L(\psi/\chi,1)^{-1} ,
\end{align*}
using Corollary \ref{linear combination} for the value of $l(\xi_0(\ell)\varphi_0,1,0)$.
The desired conclusion follows immediately by comparing the formula for $\mathfrak{z}_{\chi,\psi}(F_{\phi_{t,0}},\xi_0(\ell)\varphi_0)$ and $\mathfrak{z}_{\chi,\psi}(F_{\phi_{0,0}},\varphi_0)$.

Now note that for our choice of $\chi_1,\chi_2,\psi_1, \psi_2$, the assumptions of Corollary \ref{cor:multi1} are satisfied, 
so from it we know that
$\dim \mathrm{Hom}_{H(\mathbb{Q}_\ell)}(I(\chi_1,\psi_1)\otimes\sigma,\mathbb{C})\leq 1$.
In addition, $\mathfrak{z}_{\chi,\psi}$ is nonzero by the computation of $\frakz(F_{\phi_{0,0}}\tensor\varphi_0) = L(\psi/\chi,1)^{-1}L(0,\sigma)$ above.
Hence the proposition holds for any $\mathfrak{z}$.

\end{proof}

Recall that for $\tau$ a smooth representation of $\GL_2(\QQ_\ell)$ and $\sigma$ a smooth representation of $G(\QQ_\ell)$, we defined $\frakX(\tau,\sigma^\vee) = \Hom_{\cH_\ell(G)}(\tau \tensor_{\cH_\ell(H)} \cH_\ell(G), \sigma^\vee)$ in Section \ref{subsec:Hecke}.
Also recall that for integers $s\geq t\geq0$ we defined in Definition \ref{phi_st} the locally constant function $\phi_{s,t}\in \cS(\QQ_\ell^2,\CC)$ and its stabilizer $K_1^{\GL_2}(\ell^s,\ell^t)\subset \GL_2(\ZZ_\ell)$.

\begin{lem} \label{lem:limitexists}
Let $\frakZ\in \frakX(\cS(\QQ_\ell^2,\CC),\sigma^\vee)$.
Let $\xi\in\cH_\ell(G)$ be invariant under left-translation by the principal congruence subgroup of level $\ell^T$ in $H(\ZZ_\ell)$ for some $T\geq0$.
 Then 
 \[
 \frac1{\Vol(K_1^{\GL_2}(\ell^s,\ell^t))}\frakZ(\phi_{s,t} \tensor \xi) 
 \]
 is independent of $s \geq t\geq T$.
\end{lem}
\begin{proof}
First observe that the coset representatives of $K_1^{\GL_2}(\ell^T,\ell^T)/K_1^{\GL_2}(\ell^{s},\ell^t)$ can be chosen to be inside the principle congruence subgroup of level $\ell^T$ in $\GL_2(\QQ_\ell)$.
Hence we have 
 \[\begin{split}
 \frakZ(\phi_{T, T} \tensor \xi) 
 &=  \sum_{\gamma\in K_1^{\GL_2}(\ell^T,\ell^T)/K_1^{\GL_2}(\ell^{s},\ell^t)} \frakZ(\gamma\cdot \phi_{s,t} \tensor \xi) \\
 &= \sum_{\gamma\in K_1^{\GL_2}(\ell^T,\ell^T)/K_1^{\GL_2}(\ell^{s},\ell^t)} \frakZ( \phi_{s,t} \tensor \xi\cdot \gamma) \\
 &= \frac{\Vol(K_1^{\GL_2}(\ell^{T},\ell^T))}{\Vol (K_1^{\GL_2}(\ell^{s},\ell^t))} \frakZ( \phi_{s,t} \tensor \xi).
 \end{split}\]
 Here the third equality follows from our choice of $\gamma$ inside the principle congruence subgroup of level $\ell^T$ in $\GL_2(\QQ_\ell)$, and by assumption $\xi$ is invariant under the left-translation it.
\end{proof}

\begin{defn}
We write $\frakZ(\phi_\infty\tensor\xi) \defeq \displaystyle \lim_{s,t\to \infty} \frac1{\Vol(K_1^{\GL_2}(\ell^s,\ell^t))} \frakZ(\phi_{s,t}\tensor\xi)$ for this limit value.
\end{defn}

\begin{rmk}
Let $U$ be an open compact subgrouop of $G(\bQ_\ell)$ and $\eta \in G(\bQ_\ell)$. 
The group $\eta U \eta^{-1} \cap H(\bZ_\ell)$ contains  the principal congruence subgroup of level $\ell^T$ in $H(\ZZ_\ell)$ for some $T\geq0$ 
since $\eta U \eta^{-1} \cap H(\bZ_\ell)$ is an open compact subgroup of $H(\bZ_\ell)$. 
Hence $\ch(\eta U)$ is invariant under left-translation by the principal congruence subgroup of level $\ell^T$ in $H(\ZZ_\ell)$ for some $T\geq0$. 
In particular, one can define $\frakZ(\phi_\infty\tensor\xi)$ for any element $\xi \in \cH_\ell(G)$. 
\end{rmk}

\begin{thm}\label{tame main result}
Let $\frakZ\in \frakX(\cS(\QQ_\ell^2,\CC),\sigma^\vee)$.
Suppose that $\frakZ$ lies in the image of $\frakX(I(\chi,\psi),\sigma^\vee) \rightarrow \frakX(\cS(\QQ_\ell^2,\CC),\sigma^\vee)$, where the map is induced by $\cS(\QQ_\ell^2,\CC) \rightarrow I(\chi,\psi), \phi \mapsto F_\phi$.
Then
\[
\mathfrak{Z}(\phi_{\infty}\otimes \xi_0(\ell))
=
 L(0,\sigma)^{-1} \mathfrak{Z}(\phi_{0,0}\otimes \mathrm{ch}(G(\ZZ_\ell))).
\]
\end{thm}
\begin{proof}
We have $\Hom_H(I(\chi,\psi)\tensor\left.\sigma\right|_H,\CC) = \frakX(I(\chi,\psi),\sigma^\vee)$ by Proposition \ref{frakZ&frakz}, so $\frakZ$ corresponds to a $\frakz\in\Hom_H(I(\chi,\psi)\tensor\left.\sigma\right|_H,\CC)$.
More precisely, $\frakZ(\phi\tensor \xi)(\varphi)=\frakz(F_\phi \tensor \xi\varphi)$.
Translating the second identity of Theorem \ref{z U action} from $\frakz$ to $\mathfrak{Z}$, we obtain
\[
\mathfrak{Z}(\phi_{t,0}\otimes \xi_0(\ell))(\varphi_0) = 
\frac{1}{\ell^{t-1}(\ell+1)} L(0,\sigma)^{-1} \mathfrak{Z}(\phi_{0,0}\otimes \mathrm{ch}(G(\ZZ_\ell)))(\varphi_0).
\]
We note that $U_i(\ell) \cdot \mathrm{ch}(G(\bZ_\ell)) = U_i(\ell)$ by definition. Since $\mathrm{ch}(G(\bZ_\ell)) \cdot \sigma = \CC \varphi_0$, we indeed have
\[
\mathfrak{Z}(\phi_{t,0}\otimes \xi_0(\ell)) = 
\frac{1}{\ell^{t-1}(\ell+1)} L(0,\sigma)^{-1} \mathfrak{Z}(\phi_{0,0}\otimes \mathrm{ch}(G(\ZZ_\ell))).
\]
Since $\xi_0(\ell)$ is invariant under left-translation by $K_{0,1}$, Lemma \ref{lem:limitexists} implies that 
\[
\frakZ(\phi_\infty \tensor \xi_0(\ell)) 
=  \frac1{\Vol(K_1^{\GL_2}(\ell^t,\ell^t))}\frakZ(\phi_{t,t} \tensor \xi_0(\ell))
=  \ell^{2t-2}(\ell^2-1)\frakZ(\phi_{t,t} \tensor \xi_0(\ell))
\]
for any integer $t\geq1$.

Note that $\phi_{t,0} = \mathrm{ch}(\ell^t \mathbb{Z}_\ell \times \mathbb{Z}_\ell^\times)$ and
$\phi_{t,t} = \mathrm{ch}(\ell^t \mathbb{Z}_\ell \times (1+\ell^t\mathbb{Z}_\ell))$ are related via
\[
\phi_{t,0} = \sum_{a\in(\mathbb{Z}/\ell^t)^\times} 
\begin{pmatrix}1 & \\ & a\end{pmatrix} \cdot \phi_{t,t}.
\]
Hence by $H(\mathbb{Q}_\ell)$-equivariance of $\mathfrak{Z}$,
\[
\mathfrak{Z}(\phi_{t,0}\otimes \xi_0(\ell))
=
\sum_{a\in(\mathbb{Z}/\ell^t)^\times} 
\mathfrak{Z}(\phi_{t,t}\otimes\xi_0(\ell)\cdot (\begin{pmatrix}1 & \\ &a \end{pmatrix},\begin{pmatrix}1 & \\ &a \end{pmatrix}) )
=
\ell^{t-1}(\ell-1)\mathfrak{Z}(\phi_{t,t}\otimes \xi_0(\ell) ),
\]
where in the last identity we used the fact that $\xi_0(\ell)$ is invariant under left-translation by $K_{0,1}$, and $(\begin{pmatrix}1 & \\ &a \end{pmatrix},\begin{pmatrix}1 & \\ &a \end{pmatrix})\in K_{0,1}\cap H(\ZZ_\ell)$.
The desired result follows by combining all three equations:
\begin{align*}
\frakZ(\phi_\infty \tensor \xi_0(\ell)) 
&= \ell^{2t-2}(\ell^2-1)\frakZ(\phi_{t,t} \tensor \xi_0(\ell))\\
&=\ell^{t-1}(\ell+1) \mathfrak{Z}(\phi_{t,0}\otimes \xi_0(\ell))\\
&=L(0,\sigma)^{-1} \mathfrak{Z}(\phi_{0,0}\otimes \mathrm{ch}(G(\ZZ_\ell))).
\end{align*}
\end{proof}

\subsection{Conjugation lemmas}\label{conjugation lemma subsection}
In this section we record some technical computational results which will be used repeatedly in the following sections. 

Let $\frakZ\in \frakX(\cS(\QQ_\ell^2,\CC),\sigma^\vee)$.
For integers $i,j\geq0$ and $a,b,c\in \QQ_\ell$, let
\[
\eta_{i,j}^{a,b,c}\defeq
(\begin{pmatrix}
  1 & a\ell^{-i} & b\ell^{-i} & \\ & 1 & & b\ell^{-i} \\ & & 1 & -a \ell^{-i}\\ & & & 1
 \end{pmatrix},
 \begin{pmatrix}
  1 & c\ell^{-j} \\ & 1
 \end{pmatrix}) \in G(\QQ_\ell).
\]

\begin{lem}\label{conjugation lemma}
Let $U\subset G(\QQ_\ell)$ be an open compact subgroup and $v$ an integer.  
Take an element  $h \in H(\bQ_{\ell})$   satisfying $\mathrm{pr}_{1}(h) \in \begin{pmatrix}
\ell^{v} \cdot \bZ_{\ell}^{\times} & \bQ_{\ell} \\ & 1
\end{pmatrix}$. 
 Then for any  $g \in G(\bQ_{\ell})$, we have  
\[
\mathfrak{Z}(\phi_{\infty}\otimes  \mathrm{ch}(hg U)) = \ell^{-v} \cdot \mathfrak{Z}(\phi_{\infty}\otimes  \mathrm{ch}(g U)). 
\]
In particular, if $h \in U$ and $v=0$, we have 
\[
\mathfrak{Z}(\phi_{\infty}\otimes  \mathrm{ch}(hgh^{-1} U)) = 
\mathfrak{Z}(\phi_{\infty}\otimes  \mathrm{ch}(hg U)) =
\mathfrak{Z}(\phi_{\infty}\otimes  \mathrm{ch}(g U)). 
\]
    \end{lem}
\begin{proof}
Take a sufficiently large integer $T > 0$ such that $\mathrm{ch}(g U)$ and $\mathrm{ch}(hg U)$  are invariant under left-translation by the principal congruence subgroup of level $\ell^{T}$ in $H(\bZ_{\ell})$. 
Since $\mathrm{pr}_{1}(h) \in \begin{pmatrix}
\ell^{v} \cdot \bZ_{\ell}^{\times} & \bQ_{\ell}  \\ & 1
\end{pmatrix}$, one can take integers $s$ and $t$ satisfying $\min\{s+v,s\} \geq t \geq T$ and 
\[
h^{-1} \cdot \phi_{s, t} = \phi_{s + v, t}. 
\]
Then we have 
\begin{align*}
\mathfrak{Z}(\phi_{\infty}\otimes  \mathrm{ch}(hg U)) 
&= \frac{1}{\mathrm{Vol}(K_{1}^{\GL_{2}}(\ell^{s}, \ell^{t}))}\mathfrak{Z}(\phi_{s, t}\otimes  \mathrm{ch}(hg  U)) 
\\
&= \frac{1}{\mathrm{Vol}(K_{1}^{\GL_{2}}(\ell^{s}, \ell^{t}))}\mathfrak{Z}(\phi_{s+v, t}\otimes  \mathrm{ch}(g U))
\\
&= \ell^{-v} \cdot \mathfrak{Z}(\phi_{\infty}\otimes  \mathrm{ch}(g U)), 
\end{align*}
where the second equality follows from the $H(\bQ_{\ell})$-equivariance  of $\mathfrak{Z}$ and 
the third equality follows from $\ell^{v}\cdot \mathrm{Vol}(K_{1}^{\GL_{2}}(\ell^{s+v}, \ell^{t})) = \mathrm{Vol}(K_{1}^{\GL_{2}}(\ell^{s}, \ell^{t}))$. 
\end{proof}

\begin{lem}\label{conjugation lemma2}
Let $i, j, m, n\geq0$ be integers, and $U\subset G(\QQ_\ell)$ be either $K_{m,n}$ or $B_{m,n}$.
\begin{itemize}
\item[(1)] For any $\alpha, \beta \in \bZ_{\ell}^{\times}$ with $\alpha\beta\in 1+\ell^m\ZZ_\ell$ and $a,b,c \in \bQ_{\ell}$, 
we have 
\begin{align*}
\mathfrak{Z}(\phi_{\infty}\otimes  \mathrm{ch}(\eta^{a,b,c}_{i,j} U))
    &=
    \mathfrak{Z}(\phi_{\infty}\otimes  \mathrm{ch}(\eta^{\beta a, \alpha b, \alpha \beta^{-1} c}_{i,j} U)). 
\end{align*}
\item[(2)] For any $a_{1}, a_{2}, b_{1}, b_{2}, c_{1}, c_{2} \in \bQ_{\ell}$, we have 
\begin{align*}
\eta^{a_{1},b_{1},c_{1}}_{i,j}\eta^{a_{2}, b_{2}, c_{2}}_{i,j} = \eta^{a_{1}+a_{2},b_{1}+b_{2},c_{1}+c_{2}}_{i,j}.
\end{align*}
In particular, if $a_1 \congr a_{2} \pmod {\ell^i}, b_1 \congr  b_{2} \pmod {\ell^i}$, and $c_1\congr c_{2} \pmod {\ell^j}$, then we have 
\begin{align*}
\ch(\eta^{a_{1},b_{1},c_{1}}_{i,j}U) &= \ch(\eta^{a_{2},b_{2},c_{2}}_{i,j}U).
\end{align*}
\end{itemize}
\end{lem}
\begin{proof}
The element  
$\eta^{a, b, c}_{i,j}$ is conjugate to $\eta^{\beta a, \alpha b, \alpha \beta^{-1} c}_{i,j}$  via
$(\begin{pmatrix} 
\alpha \beta & \\ 
& 1 \\ 
\end{pmatrix},
\begin{pmatrix} 
\alpha& \\ 
& \beta\\ 
\end{pmatrix})$, 
which fixes $\phi_{s, t}$ for any $s, t\geq0$ and belongs to $U$. 
Hence claim (1) follows from Lemma \ref{conjugation lemma}. 
Claim (2) is by direct computation. 
\end{proof}

\subsection{Divisibility of $m_\eta$}
\label{sec:div}
\begin{defn}
Define 
\begin{align*}
\xi_0^{K_{1,0}}(\ell)
\defeq
\ch(K_{1,0}) 
-\ell^{-2}U_1^{K_{1,0}}(\ell)
&+2\ell^{-3}U_{2}^{K_{1,0}}(\ell)
+\ell^{-3}U_{3}^{K_{1,0}}(\ell)
\\
&+\ell^{-3}U_{4}^{K_{1,0}}(\ell)
-\ell^{-4}U_5^{K_{1,0}}(\ell)
+\ell^{-4}U_6^{K_{1,0}}(\ell)
,
\end{align*}
In particular, we have $\sum_{g\in K_{1,0}\backslash G(\ZZ_\ell)} g\cdot \xi_0^{K_{1,0}}(\ell) = \xi_0(\ell)$ by Remark \ref{indep of coset decomp}, where $\xi_0(\ell)$ was defined in Definition \ref{defn:xi_s}. 
\end{defn}

For $\frakZ\in \frakX(\cS(\QQ_\ell^2,\CC), \sigma^\vee)$, put 
\[
\mathfrak{Z}(\eta) \defeq \mathfrak{Z}(\phi_{\infty} \otimes \ch(\eta K_{1,0})) 
\]
for notational simplicity. 
Then we can write 
\[\frakZ(\phi_\infty \tensor \xi_0^{K_{1,0}}(\ell)) = \sum_\eta m_\eta\frakZ(\eta)\]
for some $m_\eta\in\QQ$. 
Here $\eta$ runs over the set $J_1 \cup \cdots \cup J_6$. 
In this section, we record a divisibility result of the multiplicities $m_\eta$ that will be useful later when proving the constructed Euler system is integral.

Recall from Definition \ref{defn:J_i} that $J_i$ is the set of left $K_{1,1}$-coset representatives of $K_{1,1}u_iK_{1,1}$ as given by Lemma \ref{double coset decomposition}, and 
by Remark \ref{indep of coset decomp}, $J_i$ is also the set of left $K_{1,0}$-coset representatives of $K_{1,1}u_i K_{1,0}$.
In the next six lemmas, we simplify the sum $\sum_{\eta \in J_{i}}\mathfrak{Z}(\eta)$ for each of $i=1,2,...,6$, and combine the results to conclude divisibility of $m_\eta$.
The computations are rather involved and repetitive, so we put them in the appendix.

\begin{lem}\label{lem:J_{1}}
\begin{alignat*}{3}
\sum_{\eta \in J_{1}}\mathfrak{Z}(\eta) = 
   (\ell + 1) \cdot \mathfrak{Z}(\eta^{0,0,0}_{0,0})   
     +(\ell^{2}-1) \cdot \mathfrak{Z}(\eta^{1,0,0}_{1,0})  
+  \ell(\ell+1) \cdot \mathfrak{Z}(\eta^{0, 0, 1}_{0,1} ) 
+  \ell (\ell^{2}-1) \cdot \mathfrak{Z}(\eta^{0, 1, 1}_{1,1}).
\end{alignat*} 
\end{lem}
\begin{proof}
 See Lemma \ref{lemma:J_{1}}.
\end{proof}


\begin{lem}\label{lem:J_{2}}
\begin{align*}
\sum_{\eta \in J_{2}}\mathfrak{Z}(\eta) 
=  \mathfrak{Z}(\eta^{0,0,0}_{0,0}) 
+ (\ell^{2}-1) \cdot \mathfrak{Z}(\eta^{1, 0, 0}_{1,0}).
\end{align*}
\end{lem}
\begin{proof}
 See Lemma \ref{lemma:J_{2}}.
\end{proof}


\begin{lem}\label{lem:J_{3}}
\begin{alignat*}{3}
\sum_{\eta \in J_{3}}\mathfrak{Z}(\eta) = 
\ell(\ell +1 )\cdot \mathfrak{Z}(\eta^{0,0,1}_{0,1}) 
 + \ell^{2}(\ell^{2}-1) \cdot \mathfrak{Z}(\eta^{1, 0, 1}_{1,1})  
+ \ell(\ell^{2}-1) \cdot \mathfrak{Z}(\eta^{0, 1, 1}_{1,1})
\end{alignat*}
\end{lem} 
\begin{proof}
 See Lemma \ref{lemma:J_{3}}.
\end{proof}


\begin{lem}\label{lem:J_{4}}
\begin{alignat*}{3}
\sum_{\eta \in J_{4}}\mathfrak{Z}(\eta) =  \ell(\ell +1) \cdot \mathfrak{Z}(\eta_{0,1}^{0, 0, 1}) 
+ \ell(\ell^{2} -1) \cdot \mathfrak{Z}(\eta_{1,1}^{0, 1,1})    
+ \ell^{2}(\ell^{2}-1) \cdot \mathfrak{Z}(\eta_{2,1}^{-\ell, 1, 1}). 
\end{alignat*}
\end{lem}
\begin{proof}
 See Lemma \ref{lemma:J_{4}}.
\end{proof}


\begin{lem}\label{lem:J_{5}}
\begin{align*}
\sum_{\eta \in J_{5}}\mathfrak{Z}(\eta  )  = 
(\ell+1) \cdot 
\mathfrak{Z}(\eta^{0,0,0}_{0,0}) 
&+ (\ell+1)(\ell^{2}-1) \cdot \mathfrak{Z}(
\eta^{1, 0, 0}_{1, 0}   ) 
+ \ell^{2}(\ell^{2}-1) \cdot \mathfrak{Z}(   
\eta^{1, 0, 0}_{2,0}   ) 
\\
&+ \ell(\ell+1) \cdot \mathfrak{Z}(   
\eta^{0, 0, 1}_{0, 1}   ) 
+ \ell(\ell^{2}-1) \cdot \mathfrak{Z}(  \eta^{0,1,1}_{1,1}  ) 
+ \ell^{2}(\ell^{2}-1) \cdot\mathfrak{Z}(  \eta^{1, 0, 1}_{1,1}  ) 
\\
&+ \ell^2(\ell^2-1) \cdot \frakZ(\eta_{2,1}^{-\ell,1,1})
+ \ell^2(\ell^2-1) \cdot \mathfrak{Z}(\phi_{\infty} \otimes  \mathrm{ch}(\eta^{0 ,1,1}_{2,1} K_{0,0})).
\end{align*}

\end{lem}
\begin{proof}
 See Lemma \ref{lemma:J_{5}}.
\end{proof}


\begin{lem}\label{lem:J_{6}}
\begin{align*}
\sum_{\eta \in J_{6}}\mathfrak{Z}(\eta) 
=  \mathfrak{Z}(\eta^{0,0,0}_{0,0}) 
+ (\ell^{2}-1) \cdot \mathfrak{Z}(\eta^{1, 0, 0}_{1,0}) 
+ \ell^{2}(\ell^{2}-1) \cdot \mathfrak{Z}(\eta^{1, 0, 0}_{2,0}).  
\end{align*}
\end{lem}
\begin{proof}
 See Lemma \ref{lemma:J_{6}}.
\end{proof}

\begin{defn}
Define 
\begin{align*}
\xi^{\rm Simple} \defeq
\ell(\ell-1)^{3}(\ell+1)^{2} \cdot \bigg( \ch(K_{1,0}) 
&- \ch(\eta^{1,0,0}_{1,0}K_{1,0})  
- \ch(\eta^{0,0,1}_{0,1}K_{1,0})  
+ \ell \cdot \ch(\eta^{1,0,1}_{1,1}K_{1,0}) 
\\
& -(\ell-1) \cdot \ch(\eta^{0,1,1}_{1,1}K_{1,0})  
+ 
\ell \cdot \ch(\eta^{-\ell,1,1}_{2,1}K_{1,0}) \bigg) \\
&+ \ell^3(\ell -1)^2(\ell+1)^2 \cdot 
\ch(\eta^{0 , 1, 1}_{2,1}K_{0,0}). 
\end{align*}

\end{defn}

Combining the above six lemmas, we obtain the following:

\begin{thm} \label{m_eta}
We have the following identity
\begin{alignat*}{3}
 \mathfrak{Z}(\phi_{\infty} \otimes \xi_0^{K_{1,0}}(\ell)) 
=
\frac1{\ell^4(\ell^2-1)} \cdot \mathfrak{Z}(\phi_{\infty} \otimes \xi^{\rm Simple}).
\end{alignat*}
\end{thm}
\begin{proof}
This theorem follows from 
Lemmas \ref{lem:J_{1}} through \ref{lem:J_{6}} and the definition that
\begin{align*}
\xi_0^{K_{1,0}}(\ell) = 
\ch(K_{1,0}) &- \ell^{-2} \sum_{\eta \in J_{1}} \ch(\eta K_{1,0}) 
+ 2\ell^{-3} \sum_{\eta \in J_{2}} \ch(\eta K_{1,0}) 
\\
&+ \ell^{-3} \sum_{\eta \in J_{3} \cup J_{4}} \ch(\eta K_{1,0}) 
-  \ell^{-4}\sum_{\eta \in J_{5}} \ch(\eta K_{1,0}) 
+ \ell^{-4}\sum_{\eta \in J_{6}} \ch(\eta K_{1,0}).
\end{align*}
\end{proof}

\begin{rmk} \label{rmk:tameintegrality}
Let $p \neq \ell$ be a prime. Then for any open compact subgroup $V$ of $H(\bQ_\ell)$, we have 
\[
{\rm Vol}(V) \in \frac{1}{(\ell-1)^{3}(\ell+1)^{2}} \bZ_p.
\]
Furthermore, since $G(\ZZ_\ell)$ contains diagonal matrices with integral coefficients, 
\[
{\rm Vol}(g G(\ZZ_\ell)g^{-1} \cap H(\bZ_\ell)) \in  \frac{1}{(\ell-1)^{2}(\ell+1)^{2}} \bZ_p 
\]
for any $g \in G(\bQ_\ell)$. 
Hence from the definition of $\xi^\mathrm{Simple} = \sum_\eta m_\eta \ch(\eta K_{1,0}) + \sum_{\eta'} m_{\eta'} \ch(\eta' K_{0,0})$, we see that $m_\eta \cdot \Vol(\eta K_{1,0}\eta^{-1}\cap H(\AA_f))\in \ZZ_p$ and $m_{\eta'}\cdot \Vol(\eta K_{0,0}\eta^{-1}\cap H(\AA_f))\in \ZZ_p$ for all $\eta,\eta'$.
This is the condition in Proposition \ref{prop:formalintegrality}, which will be used later to prove integrality of Euler system classes.
\end{rmk}


\begin{thm}\label{tame main result2}
Let $\frakZ\in \frakX(\cS(\QQ_\ell^2,\CC),\sigma^\vee)$.
Suppose that $\frakZ$ lies in the image of $\frakX(I(\chi,\psi),\sigma^\vee) \rightarrow \frakX(\cS(\QQ_\ell^2,\CC),\sigma^\vee)$.
Then
\[
\sum_{g \in G(\ZZ_\ell)/K_{1,0}} g \cdot \mathfrak{Z}(\phi_{1,1} \otimes \xi_0^{K_{1,0}}(\ell))
= \frac{L(0,\sigma)^{-1}}{\ell^2-1} \cdot  \mathfrak{Z}(\phi_{0,0}\otimes \mathrm{ch}(G(\ZZ_\ell))).
\]
\end{thm}
\begin{proof}
Since $\xi_0^{K_{1,0}}(\ell)$ is invariant under left-translation by $K_{1,1}$, Lemma \ref{lem:limitexists} implies that 
\begin{align*}
\sum_{g \in G(\ZZ_\ell)/K_{1,0}} g \cdot \mathfrak{Z}(\phi_{1,1} \otimes \xi_0^{K_{1,0}}(\ell))
&=\frac1{\ell^2-1}
\sum_{g \in G(\ZZ_\ell)/K_{1,0}} g \cdot \frakZ(\phi_\infty \tensor \xi_0^{K_{1,0}}(\ell)) \\
&=\frac1{\ell^2-1} \frakZ(\phi_\infty \tensor \xi_0(\ell)).
\end{align*}
The theorem then follows from Theorem \ref{tame main result}. 
\end{proof}


\subsection{Formula for wild norm relations}
\label{sec:wild}
Let $\frakZ\in \frakX(\cS(\QQ_\ell^2,\CC),\sigma^\vee)$.
\begin{prop} \label{prop:U_1 Z}  
Let $i,j,m,n$ be non-negative integers, and $n\geq1$.
When $i>0$, $j > 0$ and $2i+j\geq m$, we have 
\[
\frakZ
\left(\phi_{\infty} \otimes \ch ( \eta_{i+1,j+1}^{0,1,1}
B_{m,n}) \right) = 
 \frac{U^{B_{m,n}}_{1}(\ell)'}{\ell^{3}} \cdot \frakZ(\phi_{\infty} \otimes \ch( \eta_{i,j}^{0,1,1} B_{m,n})) 
\]
\end{prop}
\begin{proof}
Since $n\geq1$, we have 
 \begin{align*}
  U^{B_{m,n}}_{1}(\ell)' \cdot \frakZ(\phi_{\infty} \otimes \ch( \eta_{i,j}^{0,1,1} B_{m,n})) 
&= \frakZ(\phi_{\infty} \otimes \ch( \eta_{i,j}^{0,1,1} B_{m,n}) \cdot U_{1}^{B_{m,n}}(\ell) ) 
\\ 
&= \sum_{x,y,z,u \in \ZZ/\ell } \frakZ
\left(\phi_{\infty} \otimes \ch ( \eta_{i,j}^{0,1,1} 
(\begin{pmatrix}
 \ell &  &  x & y \\
 & \ell & z  & x \\ 
 & & 1 &  \\
 & & & 1
\end{pmatrix}, 
\begin{pmatrix}
 \ell & u \\
 & 1 
\end{pmatrix}) B_{m,n}) \right). 
 \end{align*}
Furthermore, we have 
\begin{align*}
\eta_{i,j}^{0,1,1} 
(\begin{pmatrix}
 \ell &  &  x & y \\
 & \ell & z  & x \\ 
 & & 1 &  \\
 & & & 1
\end{pmatrix}, 
\begin{pmatrix}
 \ell & u \\
 & 1 
\end{pmatrix}) 
&= 
(\begin{pmatrix}
 \ell &  &  x & y \\
 & \ell & z  & x \\ 
 & & 1 &  \\
 & & & 1
\end{pmatrix}, 
\begin{pmatrix}
 \ell & u \\
 & 1 
\end{pmatrix})\eta_{i+1,j+1}^{0,1,1} 
\\
&= (\begin{pmatrix}
 \ell & y
 \\
  & 1
\end{pmatrix}, 
\begin{pmatrix}
 \ell & z 
 \\
 & 1 
\end{pmatrix})\eta_{i+1,j+1}^{0,1+\ell^{i}x,1+\ell^{j}(u-z)}, 
\end{align*}
and hence Lemma \ref{conjugation lemma} implies that 
\[
 U^{B_{m,n}}_{1}(\ell)' \cdot \frakZ(\phi_{\infty} \otimes \ch( \eta_{i,j}^{0,1,1} B_{m,n})) 
= 
\ell \cdot \sum_{x, u \in \ZZ/\ell } \frakZ
\left(\phi_{\infty} \otimes \ch ( \eta_{i+1,j+1}^{0,1+\ell^{i}x,1+\ell^{j}u}
B_{m,n}) \right). 
\]
 
Since $i>0$ and $j>0$, the elements $1+\ell^{i}x$ and $1+\ell^{i}u$ are units.
Lemma \ref{conjugation lemma2}(1) along with $2i+j\geq m$ shows that 
\[
\sum_{x, u \in \ZZ/\ell } \frakZ
\left(\phi_{\infty} \otimes \ch ( \eta_{i+1,j+1}^{0,1+\ell^{i}x,1+\ell^{j}u}
B_{m,n}) \right) 
= 
\ell^{2} \cdot \frakZ
\left(\phi_{\infty} \otimes \ch ( \eta_{i+1,j+1}^{0,1,1}
B_{m,n}) \right). 
\]

\end{proof}

\begin{prop} \label{prop:U_2 Z}
Let $i,j,m,n$ be non-negative integers and $n\geq1$.
When $i\geq 2j\geq n$ and $i\geq m$,
we have  
\[
 \frakZ \left( \phi_\infty \tensor \ch(\eta_{i+1,j}^{0,1,1}B_{m,n}) \right)
=
\frac{U_2^{B_{m,n}}(\ell)'}{\ell^2} \cdot \frakZ \left(\phi_\infty \otimes \ch(\eta_{i,j}^{0,1,1}B_{m,n}) \right). 
\]



\end{prop}

\begin{proof}
Since $n \geq 1$, we have 
 \begin{align*}
  U^{B_{m,n}}_{2}(\ell)' \cdot \frakZ(\phi_{\infty} \otimes \ch( \eta_{i,j}^{0,1,1} B_{m,n})) 
&= \frakZ(\phi_{\infty} \otimes \ch( \eta_{i,j}^{0,1,1} B_{m,n}) \cdot U_{2}^{B_{m,n}}(\ell) ) 
\\ 
&= \sum_{\substack{x,y \in \ZZ/\ell \\ z,u \in \bZ/\ell^{2}}} \frakZ
\left(\phi_{\infty} \otimes \ch ( \eta_{i,j}^{0,1,1} 
(\begin{pmatrix}
 \ell^{2} & \ell x &  \ell y & z \\
 & \ell &   & y \\ 
 & & \ell & -x \\
 & & & 1
\end{pmatrix}, 
\begin{pmatrix}
 \ell &  \\
 & \ell 
\end{pmatrix}) B_{m,n}) \right). 
 \end{align*}
Since 
\begin{align*}
\eta_{i,j}^{0,1,1} 
(\begin{pmatrix}
 \ell^{2} & \ell x &  \ell y & z \\
 & \ell &   & y \\ 
 & & \ell & -x \\
 & & & 1
\end{pmatrix}, 
\begin{pmatrix}
 \ell &  \\
 & \ell 
\end{pmatrix}) 
&= 
(\begin{pmatrix}
 \ell^{2} & \ell x &  \ell y & z-2x\ell^{-i} \\
 & \ell &   & y \\ 
 & & \ell & -x \\
 & & & 1
\end{pmatrix}, 
\begin{pmatrix}
 \ell &  \\
 & \ell
\end{pmatrix}) 
\eta_{i+1,j}^{0,1,1} 
\\
&= (\begin{pmatrix}
 \ell^{2} & z-2x\ell^{-i}
 \\
  & 1
\end{pmatrix}, 
\begin{pmatrix}
 \ell &  
 \\
 & \ell 
\end{pmatrix})\eta_{i+1,j}^{\ell^{i}x,1+\ell^{i}y,1}, 
\end{align*}
Lemma \ref{conjugation lemma} implies that 
\begin{align*}
 U^{B_{m,n}}_{2}(\ell)' \cdot \frakZ(\phi_{\infty} \otimes \ch( \eta_{i,j}^{0,1,1} B_{m,n})) 
=  
 \sum_{x, y \in \ZZ/\ell } \frakZ
\left(\phi_{\infty} \otimes \ch ( \eta_{i+1,j}^{\ell^{i}x,1+\ell^{i}y,1}
B_{m,n}) \right). 
\end{align*}

We have the assumption that $i \geq 2j \geq n$ and $i\geq m$. 
Write $a=\ell^i x$ and $b=1+\ell^iy$ to simplify notation.
Since we assume $i \geq n$, we have 
\[
 h \defeq (\begin{pmatrix}b&\\&1\end{pmatrix},\begin{pmatrix}b& \\ -a& 1 \end{pmatrix})\in  H(\bQ_\ell) \cap B_{m,n}.
\]
Since we also assume $i \geq 2j$, $i-j\geq n$ and $i\geq m$,  we have 
\begin{align*}
 h^{-1}\eta_{i+1,j}^{a,b,1}h  B_{m,n}
&= \eta_{i+1,j}^{0,1,1} 
 (\id_4,\begin{pmatrix} b^{-1}-ab^{-1}\ell^{-j} & (b^{-1}-1)\ell^{-j}-ab^{-1}\ell^{-2j} \\ ab^{-1} & ab^{-1}\ell^{-j}+1 \end{pmatrix}
\begin{pmatrix} b&\\ -a &1\end{pmatrix})B_{m,n} \\
&= \eta_{i+1,j}^{0,1,1} B_{m,n}. 
\end{align*}
Hence Lemma \ref{conjugation lemma} shows that 
\begin{align*}
 \sum_{x, y \in \ZZ/\ell } \frakZ
\left(\phi_{\infty} \otimes \ch ( \eta_{i+1,j}^{\ell^{i}x,1+\ell^{i}y,1}
B_{m,n}) \right) 
= \ell^2 \cdot \frakZ
\left(\phi_{\infty} \otimes \ch ( \eta_{i+1,j}^{0,1,1}
B_{m,n}) \right). 
\end{align*}

\end{proof}

\begin{lem}\label{lemma:hecke operator product}
Let $n \geq 1$ be an integer. 
For $i \in \{1,2\}$, let 
\[
u_i \defeq  (\begin{pmatrix}
  \ell^{a_i+b_i} & & &\\
  & \ell^{b_i} & & \\
  & & \ell^{a_i} & \\
  & & & 1
 \end{pmatrix} ,
 \begin{pmatrix}
  \ell^{d_i} &\\
  & \ell^{c_i}
 \end{pmatrix}) \in G(\bQ_\ell)
 \]
  be a diagonal matrix satisfying  $b_{i} \geq a_{i} \geq 0$ and $d_{i} \geq c_{i} \geq 0$. Then we have 
\[
\ch(B_{m,n} u_1u_2 B_{m,n}) = \mathrm{Vol}(B_{m,n})^{-1} \cdot 
\ch(B_{m,n} u_1 B_{m,n})\ch(B_{m,n}u_2 B_{m,n}). 
\]
\end{lem}

\begin{proof}
Write 
\[
B_{m,n} u_1 B_{m,n} = \bigsqcup_{i\in I}\gamma_{1,i} u_1 B_{m,n}
\]
and 
\[
B_{m,n} u_1 B_{m,n} = \bigsqcup_{j \in J} B_{m,n}u_2 \gamma_{2,j}' = \bigsqcup_{j \in J} u_2 \gamma_{2,j}B_{m,n}.
\]
Then 
\begin{align*}
    \mathrm{Vol}(B_{m,n})^{-1} \cdot 
\ch(B_{m,n} u_1 B_{m,n})\ch(B_{m,n}u_2 B_{m,n}) &= \sum_{i,j} 
\ch(\gamma_{1,i} u_1 B_{m,n}u_2 \gamma_{2,j})
\\
&= \sum_{i} 
\ch(\gamma_{1,i} u_1 B_{m,n}u_2 B_{m,n}). 
\end{align*}
Let us show that $\bigcup_i \gamma_{1,i} u_1 B_{m,n}u_2 B_{m,n} = B_{m,n}u_1u_2 B_{m,n}$. 
For any lower triangular matrix $g \in B_{m,n}$, we have 
$gu_2 B_{m,n}= u_2 B_{m,n}$ by the choice of $u_2$. 
Hence we have 
\[
B_{m,n}u_2B_{m,n} = B u_2 B_{m,n}, 
\]
where $B$ denotes the subgroup of upper-triangular matrices. 
Furthermore, $B_{m,n}u_1 B = B_{m,n}u_1$, and thus 
\[
\bigcup_i \gamma_{1,i} u_1 B_{m,n}u_2 B_{m,n} = B_{m,n}u_1 B_{m,n} u_2 B_{m,n} 
=  B_{m,n}u_1 B u_2 B_{m,n}  = B_{m,n}u_1u_2 B_{m,n}. 
\]
Next, let us show that the union $\bigcup_i \gamma_{1,i} u_1 B_{m,n}u_2 B_{m,n}$ is disjoint. 
It suffices to show the union 
\[
\bigcup_{i,j} \gamma_{1,i} u_1 \gamma_{2,j} u_2 B_{m,n}
\]
is disjoint. This follows from the facts that $\bigcup_{i,j} \gamma_{1,i} u_1 \gamma_{2,j} u_2 B_{m,n} = B_{m,n}u_1u_2 B_{m,n}$ and 
\[
\# B_{m,n}u_1u_2 B_{m,n}/B_{m,n} = \#I+ \#J. 
\]
Therefore, we see that 
\[
\sum_{i} 
\ch(\gamma_{1,i} u_1 B_{m,n}u_2 B_{m,n}) =  
\ch(B_{m,n}u_1 u_2 B_{m,n}). 
\]
\end{proof}

\begin{cor} \label{U_5 Z}
Let $m\geq n \geq 1$ be integers. 
Then 
\[
 \frakZ \left( \phi_\infty \tensor \ch(\eta_{2m+2,m+1}^{0,1,1} B_{m,n}) \right)=
\frac{U_5^{B_{m,n}}(\ell)'}{\ell^5} \cdot \frakZ \left(\phi_\infty \tensor \ch(\eta_{2m,m}^{0,1,1} B_{m,n}) \right).  
\]
\end{cor}
\begin{proof}
 This follows directly combining  Propositions~\ref{prop:U_1 Z} and \ref{prop:U_2 Z} and Lemma \ref{lemma:hecke operator product}.
When $m \geq n \geq 1$, we have 
\begin{align*}
\frakZ \left( \phi_\infty \otimes \ch(\eta_{2m+2, m+1}^{0,1,1} B_{m,n}) \right) 
&\stackrel{\ref{prop:U_1 Z}}{=} \frac{U_1^{B_{m,n}}(\ell)'}{\ell^3}\cdot \frakZ \left( \phi_\infty \tensor \ch(\eta_{2m+1, m}^{0,1,1} B_{m,n}) \right) 
\\
&\stackrel{\ref{prop:U_2 Z}}{=} \frac{U_1^{B_{m,n}}(\ell)' U_2^{B_{m,n}}(\ell)'}{\ell^5} \cdot \frakZ \left( \phi_\infty \tensor \ch(\eta_{2m,m}^{0,1,1} B_{m,n}) \right) 
\\
&\stackrel{\ref{lemma:hecke operator product}}{=} 
\frac{U_5^{B_{m,n}}(\ell)'}{\ell^5} \cdot \frakZ \left(\phi_\infty \tensor \ch(\eta_{2m,m}^{0,1,1}  B_{m,n}) \right).
\end{align*}
 \end{proof}

\subsection{Heuristic of the choice of $U(p)'$}

The heuristic why $U_5(p)'$ is the correct Hecke operator involved in wild norm relation is explained more generally in \cite{LZ}.
Let $V$ be a $p$-adic representation of $\Gal_\QQ$.
\begin{defn}[{\cite[Definition 7.1]{LZ}}]
 A \emph{Panchishkin sub-representation} of $V$ at $p$ is a subspace $V^+\subset V$ such that
\begin{itemize}
 \item $V^+$ is stable under $\Gal_{\QQ_p}$, 
 \item $V^+$ has Hodge--Tate weights $\geq1$, 
 \item $V/V^+$ has Hodge--Tate weights $\leq0$. 
\end{itemize}
Here our convention is that cyclotomic characters have Hodge--Tate weight $+1$.
\end{defn}

\begin{conj}[{\cite[Conjecture 7.4]{LZ}}]
 Let $c\in\Gal_\QQ$ be complex conjugation. 
 Let \[r(V)\defeq\max\{0,\dim V^{c=-1}-\# \text{non-negative Hodge--Tate weights of $V$}\}.\]
 Suppose that $r(V) = 1, r(V^\ast(1))=0$, and $V$ has a Panchishkin sub-representation.
 Then there is an Euler system $c_m \in H^1(\QQ(\mu_m), V)$ for $V$.
\end{conj}

Let $\Pi=\Pi_1\tensor \Pi_2$ be a cuspidal automorphic representation of $G(\AA_f)$, where $\Pi_1$ (resp. $\Pi_2$ is an automorphic representation of $\GSp_4(\AA_f)$ (resp. $\GL_2(\AA_f$).
Assume that $\Pi=\Pi_f\tensor\Pi_\infty$ is so that $\Pi_\infty$ a unitary discrete series of weight $(k_1,k_2,k)$, with $k_1\geq k_2\geq3$, $k\geq2$ and $k\leq k_1-1$. 
(The last assumption comes from the branching law (Corollary \ref{cor: branching}), which we will need to construct Euler system elements.) 
Then we have an associated Galois representation $V=W_{\Pi_f}^\ast$, which is $8$-dimensional.
Write $V=V_1\tensor V_2$, where $V_1$ comes from the $\GSp_4$-component and $V_2$ comes from the the $\GL_2$-component.
Then $V_1^{c=-1}$ is $2$-dimensional and $V_2^{c=-1}$ is $1$-dimensional, and so $V^{c=-1}$ is $4$-dimensional.
Also $V_1$ has Hodge--Tate weights $0,k_2-2,k_1-1,k_1+k_2-3$, and $V_2$ has Hodge--Tate weights $0,k-1$.
Here the Hodge--Tate weights are listed in increasing order.

According to the conjecture, for some $r\in\ZZ$, we need a $3$-dimensional Panchishkin sub-representation $V^+(r)$ in the $8$-dimensional $V(r)$ to construct an Euler system.
Hence the Hodge--Tate weights of $V^+$ is the largest three of the Hodge--Tate weights of $V$, namely, $(k-1)+(k_1+k_2-3)$, $(k-1)+(k_1-1)$ and $k_1+k_2-3$, according to our assumptions on $k_1,k_2,k$. 
Hence $V_1$ must have a $1$-dimensional sub-representation $V_1^1$ and a $2$-dimensional sub-representation $V_1^2$ of $V_1$ coming from a Panchishikin sub-representation of a Tate twist of $V_1$; $V_2$ must have a $1$-dimensional Panchishkin sub-representation $V_2^1$ of $V_2$ with Hodge--Tate weight $k_2$.
And 
\[V^+ = V_1^1\tensor V_2 + V_1^2 \tensor V_2^1.\] 
Recall that if $\Pi_1$ is Siegel- (resp. Klingen-) ordinary at $p$, then $\left.V_1\right|_{\Gal_{\QQ_p}}$ has a $2$-dimensional (resp. $1$-dimensional) sub-representation coming from the Panchishikin sub-representation of a Tate twist of $V_1$.
As a result we need to assume that $\Pi$ is Borel ordinary at $p$, i.e. both Siegel and Klingen ordinary, for the $V^+$ to exist.
As $U_5(p)'$ detects Borel ordinarity, this is the heuristic why it should show up in the formula for wild norm relations.




\section{Euler system elements and their norm relations}
\label{sec:ES}
Fix a prime $p$, and a finite set of primes $\Sigma$ not containing $p$, and an open compact subgroup $K_\Sigma\subset G(\QQ_\Sigma)\defeq \prod_{\ell\in\Sigma}G(\QQ_\ell)$.
Choose a $\phi_\Sigma\in \cS(\QQ_\Sigma^2,\ZZ)$.
Then choose an open compact subgroup $K^H_\Sigma\subset H(\QQ_\Sigma)\subset K_\Sigma$ acting trivially on $\phi_\Sigma$.

Recall that we defined the symbol map in Section \ref{sec:symbl_inf}
\[
\Symbl^{[a,b,c,r]}\colon \cS(\AA_f^2,\QQ) \tensor_{\cH(H(\AA_f))} \cH(G(\AA_f))  \rightarrow  H^5_\mot(Y_G, \mathscr{W}^{a,b,c,\ast}_{\QQ}(3-a-r))[-a-r], 
\]
and the integral version in Section \ref{sec:symbl_int}
\[
{}_e\Symbl^{[a,b,c,r]} \colon {}_{ep}\cS((\AA_f^2,\ZZ_p) \tensor_{\cH(H(\AA_f))} \cH_{\ZZ_p}(G(\AA_f^{(p)}\times \ZZ_p))  \rightarrow  H^5_\et(Y_G, \mathscr{W}^{a,b,c,\ast}_{\QQ_p}(3-a-r))[-a-r].
\]
These maps are related by
\[
 {}_e\Symbl^{[a,b,c,r]}(\phi\tensor \xi)
= r_\et \circ \Symbl^{[a,b,c,r]} \left(
(e^2-e^{-(b+c-2r)}
(\begin{pmatrix} e &\\ &e  \end{pmatrix}, \begin{pmatrix} 1 & \\&1 \end{pmatrix})^{-1}) \phi \tensor \xi\right)
\]

\subsection{Definition of Euler system elements}
\begin{defn} \label{local data}
Let $\Sigma$ be a finite set of primes.
For each $\ell\in \Sigma$, fix a choice of $\phi_\ell\in \cS(\QQ_\ell^2,\ZZ)$ and an open compact subgroup $K_\ell\subset G(\QQ_\ell)$.

For any square-free integer $M\geq1$ coprime to $\Sigma \cup\{p\}$, and integers $m, n\geq1$, define
\[
 z_{\mot,Mp^m,p^n}^{[a,b,c,r]} \defeq \Symbl^{[a,b,c,r]}(\phi_{Mp^m,p^n}\tensor \xi_{Mp^m,p^n}).
\] 
Here
\begin{itemize}
 \item $\phi_{Mp^m,p^n} = \Tensor_{\ell}\phi_\ell\in \cS(\AA_f^2, \ZZ)$, where  $\phi_\ell$ is our fixed choice for $\ell\in \Sigma$, and for $\ell\notin\Sigma$,
\[
 \phi_\ell = 
\begin{cases}
 \phi_{0,0}:=\ch(\ZZ_\ell\times \ZZ_\ell) & \ell\nmid Mp, 
 \\
 \phi_{1,1}=\ch(\ell\ZZ_\ell\times (1+\ell\ZZ_\ell)) & \ell\mid M, 
 \\
 \phi_{4m+n,4m+n}=\ch(p^{4m+n}\ZZ_p \times (1+p^{4m+n}\ZZ_p))& \ell=p.
\end{cases}
\]
 \item $\xi_{Mp^m, p^n} :=  \Tensor_{\ell} \xi_\ell\in \cH(G(\AA_f))$, where
\[
 \xi_\ell := 
\begin{cases} 
\frac{1}{\Vol(K_\ell\cap H(\ZZ_\ell))} \ch(K_\ell) & \ell \in \Sigma, 
\\
\ch(G(\ZZ_\ell)) & \ell\nmid Mp, \, \ell\notin \Sigma, 
\\
\frac1{\Vol(K_1^{\GL_2}(\ell,\ell))} \xi_0^{K_{1,0}}(\ell)=(\ell^2-1) \xi_0^{K_{1,0}}(\ell) & \ell\mid M,
\\
\frac{p^{5m}}{\Vol(K_1^{\GL_2}(p^{4m+n}, p^{4m+n}))}\ch(\eta_{2m,m}^{0,1,1} B_{m,n}) = (p^2-1)p^{13m+2n-2}\ch(\eta_{2m,m}^{0,1,1} B_{m,n})& \ell=p.
\end{cases}
\]
\end{itemize}
Also let $K_{Mp^m,p^n}\defeq  \Tensor_{\ell} K_\ell$, where  $K_\ell$ is our fixed choice for $\ell\in \Sigma$, and for $\ell\notin\Sigma$,
\[
 K_\ell :=
\begin{cases}
 G(\ZZ_\ell) & \ell\nmid Mp,
 \\
 K_{1,0} & \ell\mid M,
 \\
 B_{m,n}& \ell=p.
\end{cases}
\]
Since $\phi_{Mp^m,p^n}\tensor \xi_{Mp^m,p^n}$ is invariant under $K_{Mp^m,p^n}$, so is $z_{\mot,Mp^m,p^n}^{[a,b,c,r]}$.
Hence $z_{\mot,Mp^m,p^n}^{[a,b,c,r]}$ is an element in $H^5_\mot(Y_G(K_{Mp^m,p^n}), \mathscr{W}^{a,b,c,\ast}_{\QQ}(3-a-r))$.
Note that we get rid of the twist $[-a-r]$.
\end{defn}
\begin{rmk}
The appearance of $4m+n$ in the definition of $\phi_p$ and $\xi_p$ can in fact be replaced by any sufficiently large integer $t$ so that
\begin{itemize}
    \item $\xi_p$ is invariant under left-translation of the principal congruence subgroup of level $p^t$ in $H(\ZZ_p)$; equivalently, the principal congruence subgroup of level $p^t$ in $H(\ZZ_p)$ is contained in $\eta_{2m,m}^{0,1,1}B_{m,n}(\eta_{2m,m}^{0,1,1})^{-1} \cap H(\ZZ_p)$.
    This condition is used to apply Lemma \ref{lem:limitexists} to rewrite formula involving $\phi_\infty$ in terms of $\phi_{t,t}$.
    \item $\frac{p^{5m}}{\Vol K_1^{\GL_2}(p^t,p^t)} \cdot \Vol(\eta_{0,0}^{0,1,1} s_m B_{m,n} s_m^{-1} (\eta_{0,0}^{0,1,1})^{-1}) \in \ZZ_p$, where $s_m$ is to be defined below.
    This condition is used to apply Proposition \ref{prop:formalintegrality} to check integrality in Theorem \ref{thm:integrality} below. 
\end{itemize}
\end{rmk}

We need an integral version of $z_{\mot,Mp^m, p^n}^{[a,b,q,r]}$.
Note that the naive idea of simply replacing $\Symbl^{[a,b,c,r]}$ by ${}_e\Symbl^{[a,b,c,r]}$ to define ${}_ez_{\et,Mp^m, p^n}^{[a,b,q,r]} = {}_e\Symbl^{[a,b,c,r]}(\phi_{Mp^m,p^n}\tensor \xi_{Mp^m,p^n})$ does not work. 
 This is because 
 \begin{enumerate}
 \item when $M$ is such that there exists $\ell\mid \gcd(M,e)$, then $\phi_\ell\neq \phi_{0,0}$, so our chosen $\phi_{Mp^m,p^n}$ does not lie in ${}_{ep}\cS(\AA_f^2, \ZZ_p)$ and 
 \item the support of $\xi_p=\ch(\eta_mB_{m,n})$ is not contained in $G(\ZZ_p)$, so our chosen $\xi_{Mp^m,p^n}$ does not lie in $\cH_{\ZZ_p}(G(\AA_f^{(p)}\times\ZZ_p))$.
 \end{enumerate}
 Hence it does not make sense to consider ${}_e\Symbl^{[a,b,c,r]}(\phi_{Mp^m,p^n}\tensor \xi_{Mp^m,p^n})$.
The remedy is to use the following Lemma \ref{lem:s_m} to reinterpret $\Symbl^{[a,b,c,r]}(\phi_{Mp^m,p^n}\tensor \xi_{Mp^m,p^n})$.

Let $s_{m} := \prod_\ell s_\ell\in H(\AA_f) \subset G(\AA_f)$, where 
\begin{align*}
s_\ell :=\begin{cases}
(\id_2, \id_2) \in H(\QQ_\ell) & \ell\neq p, 
\\
(\begin{pmatrix}
p^{3m} &\\ & 1 \end{pmatrix}, 
\begin{pmatrix}
p^{2m} & \\ & p^{m} 
\end{pmatrix})\in H(\QQ_p) & \ell=p.
\end{cases}
\end{align*}
We abuse notation to let $s_{m}$ also denote the induced map on $Y_G$ of right translation by $s_{m}$
\[
 Y_G(s_{m} K s_{m}^{-1}) \xrightarrow[\sim] {\cdot s_{m}}Y_G(K).
\]
for arbitrary level $K\subset G(\AA_f)$.
We also have a natural morphism of (resp. integral) sheaves on $Y_G(s_{m} K s_{m}^{-1})$
\[
s_{m,\#} \colon \sW_{\QQ}^{a,b,c,\ast}\xrightarrow{\sim} 
s_{m}^\ast \sW_{\QQ}^{a,b,c,\ast}
\quad \text{(resp. }
s_{m,\#} \colon \sW_{\ZZ_p}^{a,b,c,\ast}\xrightarrow{\sim} 
s_{m}^\ast \sW_{\ZZ_p}^{a,b,c,\ast}
\text{)}
\]
given by the action of $s_{m}^{-1}$ on the representation $W^{a,b,c,\ast}_{\ZZ}$ of $G$.
We again denote the map on cohomology induced by $s_{m}$ and $s_{m,\#}$ by $s_{m,\ast}$,
\begin{alignat*}{2}
&s_{m,\ast} \colon 
H^5_\mot(Y_G(s_{m} K s_{m}^{-1}), \sW_{\QQ}^{a,b,c,\ast}) &&\rightarrow
H^5_\mot(Y_G(K), \sW_{\QQ}^{a,b,c,\ast})
\\
\text{(resp. }
&s_{m,\ast} \colon 
H^5_\et(Y_G(s_{m} K s_{m}^{-1}), \sW_{\ZZ_p}^{a,b,c,\ast}) &&\rightarrow
H^5_\et(Y_G(K), \sW_{\ZZ_p}^{a,b,c,\ast})
\text{)}. 
\end{alignat*}
\begin{lem} \label{lem:s_m}
 \[
 \Symbl^{[a,b,c,r]}(\phi_{Mp^m,p^n}\tensor \xi_{Mp^m,p^n}) = s_{m,\ast}\left( \Symbl^{[a,b,c,r]}(s_{m} \cdot \phi_{Mp^m,p^n}\tensor s_{m} \cdot \xi_{Mp^m, p^n}\cdot s_{m}^{-1})\right).
 \]
\end{lem}
\begin{proof}
 This is by the definition of $s_{m,\ast}$, the $G$-equivariance of $\Symbl^{[a,b,c,r]}$, and the $H$-equivariance between $\cS(\AA_f^2, \QQ)$ and $\cH(G(\AA_f))$:
\begin{align*}
 &s_{m,\ast} \left( \Symbl^{[a,b,c,r]}(s_{m} \cdot \phi_{Mp^m,p^n}\tensor s_{m} \cdot \xi_{Mp^m, p^n} \cdot s_{m}^{-1})\right) \\
 &=\Symbl^{[a,b,c,r]}(s_{m} \cdot \phi_{Mp^m,p^n}\tensor s_{m}^{-1} s_{m} \cdot \xi_{Mp^m, p^n} \cdot s_{m}^{-1} )\\
 &=\Symbl^{[a,b,c,r]}(s_{m}^{-1} s_{m} \cdot \phi_{Mp^m,p^n}\tensor  \xi_{Mp^m, p^n}\cdot s_{m}^{-1} s_{m} ) \\
 &=\Symbl^{[a,b,c,r]}(\phi_{Mp^m,p^n}\tensor  \xi_{Mp^m, p^n})
\end{align*}
\end{proof}

We choose and fix an integer $e>1$ which is prime to $\Sigma \cup \{2,3,p\}$.
\begin{lem}\hfill 
\label{lem:phi_integral}
 \begin{enumerate}
     \item $s_{m}\cdot \phi_{Mp^m,p^n}\in {}_e \cS((\AA_f^{(p)}\times \ZZ_p)^2,\ZZ_p)$ if $M$ is  prime to $e$.
     \item $s_{m}\cdot \xi_{Mp^m,p^n}\cdot s_{m}^{-1}\in \cH_{\ZZ_p}(G(\AA_f^{(p)}\times \ZZ_p))$ if $n\geq3m$,
 \end{enumerate}
\end{lem}
\begin{proof} 
\item[1)] 
We need to check that  $s_\ell \cdot \phi_\ell =\ch (\ZZ_\ell^2)$ for $\ell \mid e$, and that $s_p \cdot \phi_p$ has support contained in $\ZZ_p^2$.
     \begin{itemize}
     \item For $\ell\mid e$, we also have $\ell\nmid Mp$ by the choice of $e$. Then $s_\ell \cdot \phi_\ell = (\id_2, \id_2)\cdot \ch(\ZZ_\ell^2) =\ch(\ZZ_\ell^2)$.
     \item $s_p\cdot \phi_p=\ch(p^{m+n}\ZZ_p \times (1+p^{4m+n}\ZZ_p))$ has support contained in $\ZZ_p^2$.
     \end{itemize}
\item[2)] For each $\ell$, the function $\xi_\ell$ is a $\ZZ_p$-linear combination of characteristic functions, so it takes value in $\ZZ_p$. 
   It remains to check that $s_m\cdot \xi_p \cdot s_m^{-1}$ has support contained in $G(\ZZ_p)$.
   Indeed 
   \begin{align*}
   s_m\cdot \xi_p \cdot s_m^{-1} 
  &=s_m\cdot \ch(\eta_{2m,m}^{0,1,1}B_{m,n}) \cdot s_m^{-1} \\
  &=\ch(s_m\eta_{2m,m}^{0,1,1}B_{m,n}s_m^{-1})\\
  &=\ch(\eta_{0,0}^{0,1,1}s_mB_{m,n}s_m^{-1}). 
   \end{align*}
   Now $\eta_{0,0}^{0,1,1}\in G(\ZZ_p)$,  and $s_mB_{m,n}s_m^{-1}\subset G(\ZZ_p)$ because we assumed $n\geq 3m$.
\end{proof}

We are now ready to define ${}_ez_{\et,Mp^m,p^n}^{[a,b,q,r]}$.

\begin{defn} \label{defn:etclass}
For  a square-free integer $M \geq1$ coprime to $\Sigma$ and $ep$, and integers $m,n\geq1$ define 
\[
{}_ez_{\et,Mp^m,p^n}^{[a,b,q,r]} \defeq
\begin{cases}
 s_{m,\ast}\left({}_e\Symbl^{[a,b,q,r]}(s_{m}\cdot \phi_{Mp^m,p^n}\tensor s_{m}\cdot\xi_{Mp^m,p^n} \cdot s_{m}^{-1}) \right) & n\geq 3m, 
 \\
 \left(\pr^{K_{Mp^m, p^{n+3m}}}_{K_{Mp^m,p^n}}\right)_\ast \left({}_ez_{\et,Mp^m,p^{n+3m}}^{[a,b,q,r]}\right) & n< 3m.
\end{cases}
\]
The definition makes sense by Lemma~\ref{lem:phi_integral}, and ${}_e z_{\et,Mp^m,p^n}^{[a,b,c,r]}$ is an element in $H^5_\et(Y_G(K_{Mp^m,p^n}), \mathscr{W}^{a,b,c,\ast}_{\QQ_p}(3-a-r))$.
Again we get rid of the twist $[-a-r]$.
\end{defn}

\begin{prop}
\label{prop:integral class}
We have 
\[
 {}_ez_{\et,Mp^m, p^n}^{[a,b,q,r]} = (e^2-e^{-(b+c-2r)}
 (\begin{pmatrix}
 e &&& \\ &e&& \\ &&e& \\ &&&e
 \end{pmatrix},
 \begin{pmatrix}
 e &\\ &e
 \end{pmatrix})^{-1}) \cdot r_\et (z_{\mot,Mp^m, p^n}^{[a,b,q,r]})  
\]
in $H^5_\et(Y_G(K_{Mp^m,p^n}), \sW_{\QQ_p}^{a,b,c,\ast}(3-a-r))$.
\end{prop}
\begin{proof}
To simplify notation, we let
\begin{align*}
{}_e h' &\defeq (e^2-e^{-(b+c-2r)}
(\begin{pmatrix} e &\\ &e  \end{pmatrix}, \begin{pmatrix} 1 & \\&1 \end{pmatrix})^{-1}) \in \ZZ[(\GL_2\times \GL_2)(\AA_f)] \\
{}_e h &\defeq (e^2-e^{-(b+c-2r)}
(\begin{pmatrix} e &\\ &e \end{pmatrix}, \begin{pmatrix} e& \\&e \end{pmatrix})^{-1}) \in \ZZ[H(\AA_f)]
\end{align*}
Because $\GL_2\times \GL_2$ acts through the first $\GL_2$-factor, we have 
\[
{}_e\Symbl^{[a,b,c,r]} (\phi\tensor\xi)  
= r_\et \circ \Symbl^{[a,b,c,r]}({}_e h'\cdot \phi \tensor \xi)\\
= r_\et \circ \Symbl^{[a,b,c,r]}({}_e h\cdot \phi \tensor \xi).
\]
Assume first that $n\geq 3m$.
Then 
\begin{align*}
{}_ez_{\et,Mp^m,p^n}^{[a,b,q,r]} 
&= s_{m,\ast}\left({}_e\Symbl^{[a,b,q,r]}(s_{m}\cdot \phi_{Mp^m,p^n}\tensor s_{m}\cdot\xi_{Mp^m,p^n} \cdot s_{m}^{-1}) \right)\\
&=s_{m,\ast} \circ  r_\et \circ  \Symbl^{[a,b,q,r]}({}_eh s_{m}\cdot \phi_{Mp^m,p^n}\tensor s_{m}\cdot \xi_{Mp^m,p^n} \cdot s_{m}^{-1}) \\
&\stackrel{\ref{lem:s_m}}{=}r_\et \circ  \Symbl^{[a,b,q,r]}(s_{m}^{-1}{}_eh s_{m}\cdot \phi_{Mp^m,p^n}\tensor \xi_{Mp^m,p^n}) \\
&=r_\et \circ \Symbl^{[a,b,q,r]}({}_eh \cdot \phi_{Mp^m,p^n}\tensor \xi_{Mp^m,p^n} ) \\
&=r_\et \circ \Symbl^{[a,b,q,r]}({}_eh \cdot \phi_{Mp^m,p^n}\tensor \xi_{Mp^m,p^n} )\\
&=r_\et \circ \Symbl^{[a,b,q,r]}(\phi_{Mp^m,p^n}\tensor \xi_{Mp^m,p^n}\cdot {}_eh )\\
&=r_\et \circ \Symbl^{[a,b,q,r]}(\phi_{Mp^m,p^n}\tensor {}_eh \cdot \xi_{Mp^m,p^n}  )\\
&= {}_eh \cdot r_\et \circ \Symbl^{[a,b,q,r]}(\phi_{Mp^m,p^n}\tensor  \xi_{Mp^m,p^n}  ) \\
&= {}_eh \cdot r_\et ( z_{\mot, Mp^m, p^n}^{[a,b,q,r]})
\end{align*}
where the third to last equality uses the fact that ${}_eh$ is in the center of $\ZZ[G(\AA_f)]$, and the second to last equality uses the $G$-equivariance of $\Symbl^{[a,b,q,r]}$.

The case of $n<3m$ follows by further noting that the motivic classes are compatible with levels:
\[
\left(\pr^{K_{Mp^m, p^{n+3m}}}_{K_{Mp^m,p^n}}\right)_\ast \left(z_{\mot,Mp^m,p^{n+3m}}^{[a,b,q,r]}\right)
=z_{\mot,Mp^m,p^{n}}^{[a,b,q,r]}.
\]
This is because it can be easily checked that $\phi_{Mp^m,p^n}\tensor \xi_{Mp^m,p^n}$ appearing in the definition of $z_{\mot,Mp^m,p^{n}}^{[a,b,q,r]}$ are compatible with levels, i.e.,
\[ 
\sum_{g\in K_{Mp^m,p^n}/K_{Mp^m,p^{n+1}}} \phi_{Mp^m,p^{n+1}}\tensor g\cdot \xi_{Mp^m,p^{n+1}} = \phi_{Mp^m,p^n}\tensor \xi_{Mp^m,p^n}.
\]
\end{proof}

\begin{thm}[Integrality] \label{thm:integrality}
Assume that $5m\geq n$.
Then 
\[
{}_ez_{\et,Mp^m,p^n}^{[a,b,c,r]} \in \mathrm{im} \left( H^5_\et(Y_G(K_{Mp^m,p^n}), \sW_{\ZZ_p}^{a,b,c,\ast}(3-a-r)) \to 
H^5_\et(Y_G(K_{Mp^m,p^n}), \sW_{\QQ_p}^{a,b,c,\ast}(3-a-r))\right).
\]
\end{thm}
\begin{proof}
We use Proposition \ref{prop:formalintegrality} to show the integrality of ${}_ez_{\et,Mp^m,p^n}^{[a,b,c,r]}$.
It suffices to check the assumption of Proposition \ref{prop:formalintegrality} for each prime $\ell$; namely if we write $\xi_\ell=\sum_\eta m_\eta\ch(\eta K_\ell)$, then $m_\eta\cdot \Vol(\eta K_\ell \eta^{-1} \cap H(\ZZ_\ell)) \in \ZZ_p$ for all $\eta$.

For $\ell\nmid Mp$, this is clear as $\xi_\ell=\frac1{\Vol(K_\ell\cap H(\ZZ_\ell))} \ch(K_\ell)$.

For $\ell\mid M$, this follows from Theorem \ref{m_eta} and Remark \ref{rmk:tameintegrality}.

For $\ell=p$, we have seen in the proof of Lemma \ref{lem:phi_integral} that $s_m\cdot \xi_p \cdot s_m^{-1}=\ch(\eta_{0,0}^{0,1,1}s_m B_{m,n}s_m^{-1})$.
Observe that $\eta_{0,0}^{0,1,1}s_m B_{m,n}s_m^{-1}(\eta_{0,0}^{0,1,1})^{-1} \cap H(\ZZ_p)$ contains
\[
\{
h=(h_1,h_2)\in H(\ZZ_p)\colon 
h_1\equiv \begin{pmatrix}
\ast & \ast \\&1 
\end{pmatrix} \bmod p^{3m+n},
h_2\equiv 
\begin{pmatrix}
\ast & \ast \\&\ast 
\end{pmatrix} \bmod p^{m+n},
\mu(h) \equiv 1 \bmod p^m
 \} 
\]
which has index in $H(\ZZ_p)$ being $(p+1)^2(p-1)p^{8m+3n-3}$.
Hence 
\[
(p^2-1)p^{13m+2n-2}\cdot \Vol \left(\eta_{0,0}^{0,1,1}s_m B_{m,n}s_m^{-1}(\eta_{0,0}^{0,1,1})^{-1} \cap H(\ZZ_p)\right) \in  \frac{p^{5m-n+1}}{p+1}\ZZ_p\subset \ZZ_p,
\]
if $5m\geq n$.
\end{proof}

\subsection{Midway results on wild norm relations: two methods}
\begin{prop} \label{mot_wild}
Let $m\geq1$ be an integer.
Then
\[
 \left( \pr^{K_{Mp^{m+1},p^n}}_{K_{Mp^m,p^n}} \right)_\ast  
 \left( z_{\mot,Mp^{m+1},p^n}^{[a,b,q,r]}\right)
=
 \frac{U_5^{B_{m,n}}(p)'}{p^{3(a+r)}} 
\cdot z_{\mot,Mp^{m},p^n}^{[a,b,q,r]}
\]
\end{prop}
\begin{proof} 
By construction (Definition \ref{local data}), $\phi_{Mp^{m+1},p^n}\tensor \xi_{Mp^{m+1},p^n}$ and $\phi_{Mp^m,p^n}\tensor \xi_{Mp^m,p^n}$ only differ at $p$, so we only need to check the equality at $p$.
Also because we defined $z_{\mot,Mp^{m+1},p^n}^{[a,b,q,r]}$ omitting the twist $[-a-r]=\lVert \mu(\--) \rVert^{-a-r}$, what we need to show is 
\[
\frakZ\left(\phi_\infty \tensor p^{5(m+1)} \ch(\eta_{2m+2,m+1}^{0,1,1}B_{m,n})\right) =  \frac{U_5^{B_{m,n}}(p)' \cdot p^{3(a+r)}}{p^{3(a+r)}} \frakZ\left(\phi_\infty \tensor p^{5m}\ch(\eta_{2m,m}^{0,1,1}B_{m,n})\right).
\]
where $\frakZ$ is the $p$-component of $\Symbl^{[a,b,c,r]}$.
This then is what we proved in Corollary \ref{U_5 Z}. 
\end{proof}

There is an alternative way to prove the wild norm relations as in Proposition \ref{mot_wild}.
Instead of using local representation theory, there is a geometric method to establish the wild norm relations. 
The method has the advantage that the wild norm relations hold in \'etale cohomology with integral coefficients, while Proposition \ref{mot_wild} is a result in cohomology with rational coefficients, and hence forgets about torsion.
The geometric method is formalized by Loeffler \cite{Loeff_sph}. 
We recall his notations and show how to modify the results to apply to our case.

Let $T_G$ be a maximal torus of $G$ and $B_G\supset T_G$ a Borel subgroup of $G$ so that $T_H = T_G \cap H$ and $B_H = B_G \cap H$ is a maximal torus and a Borel subgroup of $G$, respectively.
Let $Q_H^0$ be a mirabolic in $H$, and $Q_G=L_G\cdot N_G$ a parabolic with its Levi decomposition in $G$.

In our case, we take $T_G$ to be the diagonal matrices, $B_G$ to be the upper triangular matrices, $Q_H^0 = \{h \in H \mid \mathrm{pr}_1(h) = \begin{pmatrix}\ast& \ast \\  &1\end{pmatrix}\}$ and $Q_G=B_G$.
Consider the left action of $G$ on the flag variety $\cF = G/\bar{Q}_G$, where $\bar{Q}_G$ is the opposite of $Q_G$, i.e., the lower triangular matrices in our case.

We verify the assumptions in \cite[Section 4.3]{Loeff_sph}.
\begin{lem}
Let $u = (\eta_{0,0}^{0,1,1})^{-1} \in G(\ZZ_p)$.
\begin{itemize}
\item[(A)] The $Q_H^0$-orbit of $u$ is open in $\cF$. 
\item[(B)] We have $u^{-1}Q_H^0 u \cap \bar{Q}_G \subset \bar{Q}_G^0$, where $\bar{Q}_G^0=\bar{N}_G\cdot L_G^0$ for some normal reductive subgroup $L_G^0\subset L_G$.
\end{itemize}
\end{lem}
\begin{proof}
Assumption (A) is equivalent to saying $u^{-1}Q_H^0u\cdot \bar{Q}_G$ is open in $G$. 
Then both assumptions are verified by a direct computation:
For $(\begin{pmatrix}a&b\\&1\end{pmatrix},\begin{pmatrix}a'&b'\\c'&d'\end{pmatrix}) \in Q_H^0$, 
we have
\begin{align*}
    u^{-1}Q_H^0u = \{(\begin{pmatrix} a &c' & d'-a & b-c' \\ & a' & b'& 1-a' \\ & c' & d' & -c' \\ &&&1 \end{pmatrix}, \begin{pmatrix}a'+c' & -a'-c'+b'+d' \\ c' & -c'+d'\end{pmatrix})\}.
\end{align*}
Hence we may pick $L_G^0=\{(\begin{pmatrix}a&&&\\&1&&\\&&a&\\&&&1\end{pmatrix}, \begin{pmatrix} 1&\\&a\end{pmatrix})\}$.
\end{proof}

Let $\eta\in X_\bullet(T_G)$ be a cocharacter which factors through the center $Z(L_G)$ and strictly dominant. 
Set $\tau=\eta(p)$.
For  an integer $m\geq0$, define open compact subgroups of $G(\ZZ_p)$:
\begin{align*}
    U_m &\defeq \{g\in G(\ZZ_p) \colon \tau^{-m}g\tau^m\in G(\ZZ_p), g\bmod p^m \in \bar{Q}_G^0 \} \\
    U_m' &\defeq \{g\in G(\ZZ_p) \colon \tau^{-(m+1)}g\tau^{m+1}\in G(\ZZ_p), g\bmod p^m \in \bar{Q}_G^0 \}  = U_m \cap \tau U_m \tau^{-1}\\
    V_m &\defeq \tau^{-m}U_m \tau^m,
\end{align*}
so $U_m\supset U_m' \supset U_{m+1}$.
Let $\cT$ be the Hecke operator defined by the double coset $U_m \tau^{-1} U_m$.

Let $M_H\rightarrow M_G$ be a map of Cartesian cohomology functors (defined in \cite[Section 2]{Loeff_sph}).
In our situation we take $M_H(\--) = H^1_\et(Y_H(\--), \iota^\ast\sW_{\ZZ_p}^{a,b,c,\ast})$ and $M_G(\--) = H^5_\et(Y_G(\--), \sW_{\ZZ_p}^{a,b,c,\ast}(2))$.
Given a cohomology functor $M$, one can define the Iwasawa completion: for any (not necessarily open) compact subgroups $K$
\[
M_{\Iw}(K) = \varprojlim_{U\supset K} M(U),
\]
where $U$ runs through open compact subgroups containing $K$, and the transition maps are given by pushforwards.
If $z_H$ is an element in $M_{H,\Iw}(Q_H^0)$.
we can define $z_{G,m}$ to be the image of $z_H$ under the map 
\[
M_{H,\Iw}(Q_H^0) \xrightarrow {\pr^\ast} M_{H,\Iw}(Q_H^0\cap u^{-1}U_m u) \xrightarrow{u_\ast} M_{G,\Iw}(U_m).
\]
Let $\xi_m = (\tau^m)_\ast \cdot z_{G,m}$.

\begin{thm}[{\cite[Proposition 4.5.2]{Loeff_sph}}] \label{Lthm}
 $(\pr^{V_{m+1}}_{V_m})_\ast \xi_{m+1} = \cT \cdot \xi_m$.
\end{thm}
The proof of this theorem is a formal argument in cohomology, as long as the following key lemma on the interplay between the level subgroups $U_m$ and $U_m'$ holds.
\begin{lem}[{\cite[Lemma 4.4.1]{Loeff_sph}}] \label{Llemma}
Suppose $m\geq1$.
\begin{itemize}
    \item[(i)] We have $u^{-1} Q_H^0 u\cap U_m' = u^{-1}Q_H^0u \cap U_{m+1}$. 
    \item[(ii)] We have $[u^{-1}Q_H^0u\cap U_m \colon u^{-1}Q_H^0u \cap U_m'] = [U_m\colon U_m']$. 
\end{itemize}
\end{lem}

We cannot directly apply Theorem \ref{Lthm} and Lemma \ref{Llemma} because the level subgroups we use are different from the $U_m$ defined here.
In our situation, we take $\tau = u_5$ and hence $\cT = U_5(p)'$.
We also take $z_H$ to be 
\[
\bran^{[a,b,c,r]}\circ \pr_1^\ast \circ {}_e\Eis^{b+c-2r}_\et(\phi_\infty) = \left(\frac1{\Vol(K_{Mp^t,p^n})}\bran^{[a,b,c,r]}\circ \pr_1^\ast \circ {}_e\Eis^{b+c-2r}_\et(\phi_{Mp^t,p^n})\right)_t.
\]
Then ${}_ez_{\et, Mp^m,p^n}^{[a,b,q,r]}$ (Definition \ref{defn:etclass}) is constructed from $z_H$ similar to how $z_{G,m}$ is constructed from $z_H$, except that $U_m$ is replaced by $s_mB_{m,n+3m}s_m^{-1}$ (and a Tate twist by $1-a-r$.) 

Define $\cU_m = s_mB_{m,n+3m}s_m^{-1}, \cU'_m = \cU_m \cap \tau \cU_m \tau^{-1}$, and $\cV_m = \tau^{-m} \cU_m \tau^m = B_{m,n+3m}$, so $\cU_m\supset \cU_m' \supset \cU_{m+1}$.
We have an analogue of Lemma \ref{Llemma}.
\begin{lem}
Suppose $m\geq1$.
\begin{itemize}
    \item[(i)] We have $u^{-1} Q_H^0 u\cap \cU_m' = u^{-1}Q_H^0u \cap \cU_{m+1}$. 
    \item[(ii)] We have $[u^{-1}Q_H^0u\cap \cU_m \colon u^{-1}Q_H^0u \cap \cU_m'] = [\cU_m\colon \cU_m']$. 
\end{itemize}
\end{lem}
\begin{proof}
The proof of (i) is very similar to that of Lemma \ref{Llemma}.
By a direct computation, 
\[ 
q\in \cU_m' \Longleftrightarrow q\in \begin{pmatrix} \ZZ_p & p^{m+1} & p^{2m+2} & p^{3m+3} \\ p^{n+2m} & \ZZ_p & p^{m+1} & p^{2m+2} \\ p^{n+m} & p^{n+2m} & \ZZ_p & p^{m+1} \\ p^n & p^{n+m} & p^{n+2m} & 1+p^{n+3m}\ZZ_p \end{pmatrix}, \begin{pmatrix} \ZZ_p & p^m \\ p^n & \ZZ_p \end{pmatrix} )
\]
and 
\[ 
q\in \cU_{m+1} \Longleftrightarrow q\in \begin{pmatrix} \ZZ_p & p^{m+1} & p^{2m+2} & p^{3m+3} \\ p^{n+2m+2} & \ZZ_p & p^{m+1} & p^{2m+2} \\ p^{n+m+1} & p^{n+2m+2} & \ZZ_p & p^{m+1} \\ p^n & p^{n+m+1} & p^{n+2m+2} & 1+p^{n+3m+3}\ZZ_p \end{pmatrix}, \begin{pmatrix} \ZZ_p & p^m \\ p^n & \ZZ_p \end{pmatrix} ).
\]
It can then be computed that their intersections with $u^{-1} Q_H^0 u$ are the same.
Here we implicitly used a stronger version of assumption (B) in the computation.

To prove (ii), we need to show that there is a set of representatives for $\cU_m'\backslash\cU_m$ contained  $u^{-1}Q_H^0u$.
We have an isomorphism $\cU_m'\backslash\cU_m = N_{m+1}\backslash N_m$, where $N_m =\tau^r N_G(\ZZ_p)\tau^{-m}$.
Assumption (A) says that the orbit of identity under $u^{-1}Q_H^0 u$ is open as a $\ZZ_p$-subscheme of $\cF = G/\bar{Q}_G$, and hence contains the kernel of reduction modulo $p$.
As a result, $u^{-1}Q_H^0 u \cdot \bar{Q}_G \supset N_{m}\bar{Q}_G$ for any integer $m\geq0$.
Indeed, if we write $\bar{Q}_G(p^r)\subset \bar{Q}_G(\ZZ_p)$ for the kernel of reduction modulo $p^r$, then we still have  $u^{-1}Q_H^0 u \cdot \bar{Q}_G(p^r) \supset N_{m}\bar{Q}_G(p^r)$ for any integer $m\geq0$.
Let $x\in N_{m}$. Then there exists some $q\in Q_H^0$ such that $u^{-1}q u \in N_{m+1}x\bar{Q}_G(p^{n+3m})$.
So $u^{-1}q u\in \cU_m$ and maps to $x$ under the isomorphism $\cU_m'\backslash\cU_m = N_{m+1}\backslash N_m$.
\end{proof}

Having this Lemma, the proof of Theorem \ref{Lthm} then gives us
\begin{thm}
Let $m\geq1$ be an integer.
Then
 \[
  \left( \pr^{K_{Mp^{m+1},p^n}}_{K_{Mp^m,p^n}} \right)_\ast  
 \left( {}_ez_{\et,Mp^{m+1},p^n}^{[a,b,q,r]}\right)
= \frac{U_5^{B_{m,n}}(p)'}{p^{3(a+r)}} 
\cdot {}_ez_{\et,Mp^{m},p^n}^{[a,b,q,r]}. 
 \]
\end{thm}

\subsection{Projection to $\Pi$-component and norm relations}
Let $\Pi=\Pi_f\tensor\Pi_\infty$ be a cuspidal automorphic representation of $G(\AA_f)$ with $\Pi_\infty$ a unitary discrete series.
Assume that $\Pi$ is non-endoscopic.
Let $K\subset G(\AA_f)$ be the level of $\Pi$.

Let $k_1\geq k_2\geq3$ and $k\geq2$ be integers.
Let $a=k_2-3$, $b=k_1-k_2$, $c=k-2$, so $a,b,c\geq0$.
Let $\Sigma(k_1,k_2,k)$ denote the set of isomorphism classes of representations $\Pi_f$ of $G(\AA_f)$ which are the finite part of a cuspidal automorphic representation $\Pi=\Pi_f\tensor\Pi_\infty$ where $\Pi_\infty$ is a unitary discrete series of weight $(k_1,k_2,k)$.

\begin{thm}[Taylor, Weissauer \cite{Wei}]
 There is a $G(\AA_f)\times \Gal_\QQ$-equivariant decomposition
\[
 H^4_{\et,!}(Y_{G,\Qbar},\sW_{\QQ_p}^{a,b,c}) \tensor_{\QQ_p} \Qpbar \isom \Directsum_{\Pi_f\in\Sigma(k_1,k_2,k)} \Pi_f\left[-\frac{2a+b+c}2\right]\tensor W_{\Pi_f},
\]
where $W_{\Pi_f}$ is a finite dimensional $p$-adic representation of $\Gal_\QQ$.
Here $\Pi_f[r]\defeq \Pi_f\tensor \lVert \mu(\cdot) \rVert^r$.

If $\Pi$ is non-endoscopic, then the semi-simplification of $W_{\Pi_f}$ is isomorphic to $\rho_{\Pi,p}$,
the unique (up to isomorphism) semisimple Galois representation so that $\det(1-\rho_{\Pi,p}(\Frob_\ell^{-1})\cdot \ell^{-s}) = L(\Pi_\ell,s-\frac{2a+b+c+4}2)^{-1}$.

Let $\Pi$ be non-endoscopic of level $K$, and write $\frakm_{\Pi_f}$ for the associated maximal ideal of the spherical Hecke algebra away from $K$. 
(As the notation suggests, $\frakm_{\Pi_f}$ only depends on the finite part $\Pi_f$ of $\Pi$.)
Then the localization of $H^i_\et(Y_{G}(K)_{\Qbar},\sW_{\QQ_p}^{a,b,c})$ at $\frakm_{\Pi}$ is zero unless $i=4$, in which case the localization is equal to that of $H^i_{\et,!}$.
\end{thm}

We rewrite the above decomposition as
\[
H^4_{\et,!}(Y_{G,\Qbar},\sW_{\QQ_p}^{a,b,c}) \tensor_{\QQ_p} \Qpbar \isom \Directsum_{\Pi_f\in\Sigma(k_1,k_2,k)} \Pi_f^\ast\left[-\frac{2a+b+c}2\right]\tensor W_{\Pi_f^\ast}.
\]
The Poincar\'e duality says there is a perfect pairing 
\[
 H^4_{\et,!}(Y_{G,\Qbar},\sW_{\QQ_p}^{a,b,c}) \times H^4_{\et,!}(Y_{G,\Qbar},\sW_{\QQ_p}^{a,b,c,\ast}) \rightarrow \QQ_p(-4).
\]
Combining this with the above decomposition, and recall that $\sW^{a,b,c,\ast} = \sW^{a,b,c}[2a+b+c](2a+b+c)$, we obtain 
\[
W_{\Pi_f^\ast} \isom W_{\Pi_f}^\ast(-2a-b-c-4).
\]
Hence we can again rewrite the decomposition as
\[
H^4_{\et,!}(Y_{G,\Qbar},\sW_{\QQ_p}^{a,b,c,\ast}(4)) \tensor_{\QQ_p} \Qpbar \isom \Directsum_{\Pi_f\in\Sigma(k_1,k_2,k)} \Pi_f^\ast\left[\frac{2a+b+c}2\right]\tensor W_{\Pi_f}^\ast.
\]
There is a Hochschild--Serre spectral sequence
\[
E_2^{r,s} = H^r(\QQ, H^s_\et(Y_G(K)_{\Qbar},\sW^{a,b,c,\ast}_{\QQ_p}(n))) \Rightarrow H^{r+s}_\et(Y_G(K),\sW^{a,b,c,\ast}_{\QQ_p}(n)),
\]
for any integer $n$, compatible with the Hecke action.
Localizing at $\frakm_{\Pi_f}$, we obtain
\[
H^5_\et(Y_G(K), \sW^{a,b,c,\ast}_{\QQ_p}(n))_{\frakm_{\Pi_f}} \tensor_{\QQ_p} \Qpbar = \Directsum_{\Pi_f'\sim \Pi_f} \Pi'^\ast_f\left[\frac{2a+b+c}2\right]^K \tensor H^1(\QQ,W_{\Pi'_f}^\ast(n-4)).
\]
We write 
\[
\pr_{\Pi^\ast} \colon H^5_{\et}(Y_{G}(K),\sW_{\QQ_p}^{a,b,c,\ast}(3-a-r))\rightarrow \Pi_f^\ast\left[\frac{2a+b+c}2\right]^K \tensor H^1(\QQ,W_{\Pi_f}^\ast(-1-a-r))
\]
for the map given by localization at $\frakm_{\Pi_f}$ and then projection onto the isotypical component.

From now on, we assume that $\Pi$ is unramified and ordinary at $p$, i.e,
$\Pi_p$ is unramified and there is a $p$-adic eigenvalue $\alpha$ of $U_5^{B_{0,1}}(p)$
on $\Pi_p\left[-\frac{2a+b+c}{2}\right]^{B_{0,1}}$. 
Let $v_\alpha\in \Pi_f\left[ -\frac{2a+b+c}{2} \right]^{K^pB_{0,1}}$ be an eigenvector for  $\alpha$. 
Then $v_\alpha$ induces 
\[
\pr_\alpha\colon\Pi_f^\ast \left[  \frac{2a+b+c}{2} \right]^{K^p B_{0,1}}\rightarrow \Qpbar.
\]

Let $\Sigma\not\ni p$ be the set of primes $\ell$ for which $K_\ell \neq G(\ZZ_\ell)$.
For any square-free integer $M\geq1$ coprime to $\Sigma \cup \{p\}$ and any integer $m\geq0$, we consider the open compact subgroups $K_{Mp^m,p}$ as defined in Definition \ref{local data}.
We have isomorphisms
\begin{align*}
Y_G(K_{Mp^m,p}) 
&\cong Y_G(K^pB_{0,1})\times Y_{\GL_1}(1+Mp^m\hat\ZZ)\\
&\cong Y_G(K^pB_{0,1})\times \Spec \QQ[X]/\Phi_{Mp^m}(X), 
\end{align*}
where $\Phi_{Mp^m}(X)$ is the cyclotomic polynomial of degree $Mp^m$.
Hence 
\begin{align*}
H^i_\et(Y_{G}(K_{Mp^m,p}), \sW_{\QQ_p}^{a,b,c})
&\isom 
H^i_\et(Y_{G, \QQ(\mu_{Mp^m})}(K^pB_{0,1}),\sW_{\QQ_p}^{a,b,c}), 
\\
H^i_\et(Y_{G,\Qbar}(K_{Mp^m,p}), \sW_{\QQ_p}^{a,b,c})
&\isom 
\QQ_p[\Gal(\QQ(\mu_{Mp^m})/\QQ)] \tensor H^i_\et(Y_{G,\Qbar}(K^pB_{0,1}),\sW_{\QQ_p}^{a,b,c}). 
\end{align*}
Then composing $\pr_\alpha$ and $\pr_{\Pi^\ast}$ for the level $K_{Mp^m,p}$ gives
\begin{center}
\begin{tikzcd}
 H^5_{\et}(Y_{G,\QQ(\mu_{Mp^m})}(K_{Mp^m,p}),\sW_{\QQ_p}^{a,b,c,\ast}(3-a-r)) \ar[d,"\pr_\alpha\circ\pr_{\Pi^\ast}"] \\
  \QQ_p[\Gal(\QQ(\mu_{Mp^m})/\QQ)]\tensor_{\QQ_p} H^1(\QQ,W_{\Pi_f}^\ast(-1-a-r). 
\end{tikzcd}
\end{center}

\begin{defn}[Euler system classes] 
\label{def:Gal class} 
Fix an integer $e>1$ which is prime to $\Sigma \cup \{2,3,p\}$.
Assume that $c\leq a+b$, and let $r$ be an integer such that $\max(0,-a+c) \leq r \leq \min(b,c)$.
For any square-free integer $M\geq1$ coprime to $\Sigma \cup \{p\}$, and any integer $m\geq0$, define
\[
 {}_ez_{Mp^m}^{[\Pi,r]} \defeq 
\begin{cases}
 (\frac{p^{3(a+r)}\sigma_p^3}{\alpha})^m \cdot \pr_\alpha \circ \pr_{\Pi^\ast} ({}_ez^{[a,b,c,r]}_{\et,Mp^m, p}) & m\geq1, 
 \\ 
 \cores^{\QQ(\mu_{Mp})}_{\QQ(\mu_M)}({}_ez_{Mp}^{[\Pi,r]}) & m=0,
\end{cases}
\]
where $\sigma_p$ is the image of $p^{-1}$ under the Artin reciprocity map $\QQ_p^\times \inj \AA_f^\times \rightarrow \Gal(\QQ(\mu_{Mp^m})/\QQ)$.
By Shapiro's Lemma, this is an element in
$H^1(\QQ(\mu_{Mp^m}),W_{\Pi_f}^\ast(-1-a-r))$.

In fact, note that $\pr_\alpha$ induces 
\[H^4_\et(Y_G(K^pB_{0,1})_{\Qbar},\sW_{\QQ_p}^{a,b,c,\ast})\tensor_{\QQ_p}\Qpbar \rightarrow W_{\Pi_f}^\ast\] 
and we denote by $T_{\Pi_f}^\ast$ the image of $H^4_\et(Y_G(K^pB_{0,1})_{\Qbar}, \sW_{\ZZ_p}^{a,b,c,\ast})\tensor_{\ZZ_p}\Zpbar$ under this map.
Then $T_{\Pi_f}^\ast$ is a Galois-stable lattice in $W_{\Pi_f}^\ast$, and by Theorem \ref{thm:integrality}, ${}_ez^{[\Pi,r]}_{Mp^m}$ lies in $H^1(\QQ(\mu_{Mp^m}),T_{\Pi_f}^\ast(-1-a-r))$. 
\end{defn}

\begin{prop}[Wild norm relation] 
\label{wild}
Let $m\geq0$ be an integer. Then 
 \[
  \cores^{\QQ(\mu_{\ell Mp^{m+1}})}_{\QQ(\mu_{Mp^m})}({}_ez_{Mp^{m+1}}^{[\Pi,r]})
=  {}_e z_{Mp^m}^{[\Pi,r]}.
 \]
\end{prop}
\begin{proof}
The $m=0$ case is by definition.
For $m\geq1$, this follows from Proposition \ref{mot_wild}, noting that under $\pr_\alpha$, $U_5^{B_{m,1}}(p)'$ acts as $\alpha$, and the identification of $Y_G(K_{Mp^m,p}) \isom Y_G(K^pB_{0,1})\times \Spec \QQ(\mu_{Mp^m})$ intertwines $U_5^{B_{m,1}}(p)'$ with $U_5^{B_{m,1}}(p)'\sigma_p^{-3}$ (see \cite[Sec 5.4]{LSZ}).
\end{proof}

\begin{prop}[Tame norm relation] 
\label{tame}
Assume that $b+c-2r\neq0$ (As $r$ can be freely chosen as long as $\max(0,-a+c)\leq r\leq \min(b,c)$, we cannot choose such $r$ so that $b+c-2r\neq 0$ only when $a=b=c=0$.)
Let $m\geq0$ be an integer. Then 
 \[
  \cores^{\QQ(\mu_{\ell Mp^m})}_{\QQ(\mu_{Mp^m})}({}_ez_{\ell Mp^m}^{[\Pi,r]})
= P_\ell(\ell^{-2-a-r}\sigma_\ell^{-1}) \cdot {}_e z_{Mp^m}^{[\Pi,r]}, 
 \]
where $P_\ell(X) = \det(1-X \Frob_{\ell}^{-1} \mid W_{\Pi_f})$, and $\sigma_\ell$ is the arithmetic Frobenius at $\ell$ in $\Gal(\QQ(\mu_{Mp^m})/\QQ)$.
\end{prop}
\begin{proof}
Using Shapiro's Lemma, we have the following commutative diagram
\[
 \begin{tikzcd}
  H^1(\QQ, \QQ_p[\Gal(\QQ(\mu_{\ell Mp^m})/\QQ)])\tensor W_{\Pi_f}^\ast(-1-a-r)) \ar[r,"\sim"] \ar[d,"\pr^{\Gal(\QQ(\mu_{\ell Mp^m })/\QQ)}_{\Gal(\QQ(\mu_{Mp^m})/\QQ)}"] 
& H^1(\QQ(\mu_{\ell Mp^m}),W_{\Pi_f}^\ast(-1-a-r)) 
\ar[d,"\cores^{\QQ(\mu_{\ell Mp^m })}_{\QQ(\mu_{Mp^m})}"] 
\\
H^1(\QQ, \QQ_p[\Gal(\QQ(\mu_{Mp^m})/\QQ)])\tensor W_{\Pi_f}^\ast(-1-a-r)) \ar[r,"\sim"] 
& H^1(\QQ(\mu_{Mp^m}),W_{\Pi_f}^\ast(-1-a-r)) 
 \end{tikzcd}
\]
The upper (resp. lower) isomorphism respects the $\Gal(\QQ(\mu_{\ell Mp^m})/\QQ)$ (resp. $\Gal(\QQ(\mu_{Mp^m})/\QQ)$)-action on both sides. 
On the left hand side, the action is induced from the $\Gal(\QQ(\mu_{Mp^m})/\QQ)$-action on $\QQ_p[\Gal(\QQ(\mu_{Mp^m})/\QQ)]\tensor W_{\Pi_f}^\ast$ given by $g\cdot (g'\tensor w) = gg'\tensor g^{-1}w$.
On the right hand side, the action is the usual one: given $g\in \Gal(\QQ(\mu_{Mp^m})/\QQ)$, there is an automorphism on $H^1(\QQ(\mu_{Mp^m}),W_{\Pi_f}^\ast(-1-a-r))$ coming from $\Gal_{\QQ(\mu_{Mp^m})}\rightarrow \Gal_{\QQ(\mu_{Mp^m})}, h \mapsto ghg^{-1}$ and $W_{\Pi_f}^\ast(-1-a-r)\rightarrow W_{\Pi_f}^\ast(-1-a-r), w\mapsto g^{-1}w$.
So we need to prove
\[
\pr^{\Gal(\QQ(\mu_{\ell Mp^m })/\QQ)}_{\Gal(\QQ(\mu_{Mp^m})/\QQ)}({}_ez_{\ell Mp^m}^{[\Pi,r]})
= P_\ell(\ell^{-2-a-r}\sigma_\ell^{-1}) {}_e z_{Mp^m}^{[\Pi,r]}.
\]

By construction (Definition \ref{local data}, \ref{defn:etclass} and \ref{def:Gal class}), the two sides of the equality only differ at $\ell$, so using the relation between integral classes and rational classes (Proposition \ref{prop:integral class}) we only need to show 
\begin{align} \label{eq:tame norm relation}
\sum_{g\in K_{0,0}/K_{1,0}} g\cdot \frakZ\left(\phi_{1,1}\tensor (\ell^2-1)\cdot \xi_0^{K_{1,0}}(\ell)\right) 
= P_\ell(\ell^{-2-a-r}\sigma_\ell^{-1}) \cdot 
\frakZ\left(\phi_{0,0}\tensor \ch(G(\ZZ_\ell)) \right) 
\end{align}
where 
\[
\frakZ\colon \cS(\QQ_\ell^2,\QQ) \tensor \cH_\ell(G) \rightarrow \Pi^\ast_\ell\left[\frac{b+c-2r}{2}\right]^{G(\QQ_\ell)} \tensor H^1(\QQ,\QQ_p[\Gal(\QQ(\mu_{Mp^m})/\QQ)]\tensor W_{\Pi_f}^\ast(-1-a-r))
\]
is the $\ell$-component of $\Symbl^{[a,b,c,r]}$ composed with $\pr_{\Pi_\ell^\ast}$.

Note that by construction $\Symbl^{[a,b,c,r]}$ factors through $\Eis_\mot^{b+c-2r}$.
By our assumption, $b+c-2r>0$.
By Theorem \ref{Eis&PPk} the image of $\Eis_\mot^{b+c-2r}$ is isomorphic via $\partial_{b+c-2r}$ to $\Directsum_\nu I(|\cdot|^{b+c-2r+1/2}\nu, |\cdot|^{-1/2})$.
Also $\Pi_\ell$ is an unramified principal series of $G(\QQ_\ell)$.
Hence $\frakZ$ satsifies the assumptions of Theorem \ref{tame main result2}.

When $M=1$ and $m=0$, the $\Gal(\QQ(\mu_{Mp^m}/\QQ)$-action is trivial, so 
\[
P_\ell(\ell^{-2-a-r}\sigma_\ell^{-1})
=P_\ell(\ell^{-2-a-r})
=L(\frac{-b-c+2r}{2}, \Pi_\ell)^{-1}
=L(0,\Pi_\ell\left[\frac{-b-c+2r}{2}\right])^{-1}.
\]
Hence (\ref{eq:tame norm relation}) follows from Theorem \ref{tame main result2}.
For general $M$ and $m$, we apply the $M=1, m=0$ case replacing $\Pi$ by $\Pi\tensor \chi^{-1}$, where 
\[\chi\colon (\ZZ/Mp^m)^\times \isom \Gal(\QQ(\mu_{Mp^m})/\QQ) \rightarrow \CC^\times
\]
is a Dirichlet character.
Since $W^\ast_{\Pi_f\tensor\chi^{-1}} = W^\ast_{\Pi_f}\tensor \chi$, and 
\[
L(s,\Pi_\ell\tensor \chi_\ell^{-1}) 
= P_\ell(\ell^{-s-\frac{2a+b+c+4}2}\chi^{-1}(\ell))
= P_\ell(\ell^{-s-\frac{2a+b+c+4}2}\chi(\sigma_\ell^{-1})),
\]
we conclude that (\ref{eq:tame norm relation}) is also true in this case.
\end{proof}

\begin{rmk}
We remark on the case of $b+c-2r=0$ which we deliberately omitted because we do not yet know how to prove it and because when considering Euler systems in $p$-adic families, missing one single weight ($a=b=c=0 \Leftrightarrow k_1=k_2=3,k=2$) might not matter too much.

We can still prove the tame norm relation as long as we can apply Theorem \ref{tame main result2} to $\frakZ$, and then argue as in the case $b+c-2r>0$ above.
By Theorem \ref{Eis&PP0}, $\cO^\times(Y_{\GL_2})$ surjects via $\partial_0$ onto $I^0(|\cdot|^{1/2},|\cdot|^{-1/2}) \directsum \Directsum_{\nu\neq1} I(|\cdot|^{1/2}\nu, |\cdot|^{-1/2})$ with the kernel being a sum of $1$-dimensional (and hence non-generic) representations of $\GL_2(\AA_f)$.
So to apply Theorem \ref{tame main result2} we only need that $\Symbl^{[a,b,c,r]}$ factors through the induced representations $I(|\cdot|^{1/2}\nu, |\cdot|^{-1/2})$; hence it suffices to show that 
\[
\frakX(\tau,\Pi^\vee)=\Hom_H(\tau\tensor\Pi,\CC) = 0
\]
for any non-generic representation $\tau$ of $\GL_2(\AA_f)$.
If it were not true, then locally at all $\ell$, we must have 
\[
\dim \Hom_{H(\QQ_\ell)}(\tau_\ell\tensor\Pi_\ell,\CC)\geq1.
\]
By Theorem D of \cite{Loeff_zeta}, this implies that $s=0$ is an exceptional pole of $L(\Pi_\ell,s)$ for all $\ell$.
One might be able to say that this only happens for very special $\Pi$ which are endoscopic, but we have not figured this out.
\end{rmk}

\appendix
\section{Proofs of Lemmas \ref{lem:J_{1}} to \ref{lem:J_{6}}}

In the appendix we present the proofs of Lemmas \ref{lem:J_{1}} through \ref{lem:J_{6}}.
Let $i,j$ be integers and $a,b,c\in \ZZ_\ell$.
Recall the notation
 \[
 \eta^{a,b,c}_{i,j} =
 (\begin{pmatrix}
1 & a\ell^{-i} & b\ell^{-i} & \\ & 1 && b\ell^{-i} \\ && 1 & -a \ell^{-i} \\ &&& 1
\end{pmatrix},
\begin{pmatrix} 1 & c\ell^{-j} \\ & 1\end{pmatrix}) \in G(\QQ_\ell).
\]
In addition, we will also write
\[
\theta^{a,b,c}_{i,j} = (\begin{pmatrix}
1 & a\ell^{-i} & b\ell^{-i} & \\ & 1 && b\ell^{-i} \\ && 1 & -a \ell^{-i} \\ &&& 1
\end{pmatrix},
\begin{pmatrix} \ell & c\ell^{1-j} \\ & \ell^{-1}\end{pmatrix}) \in G(\QQ_\ell)
\]
and
\[
\iota^{a,b,c}_{i,j} = (\begin{pmatrix}
1 & a\ell^{-i} & b\ell^{-i} & \\ & 1 && b\ell^{-i} \\ && 1 & -a \ell^{-i} \\ &&& 1
\end{pmatrix},
\begin{pmatrix} \ell^{-1} & c\ell^{-1-j} \\ & \ell \end{pmatrix}) \in G(\QQ_\ell).
\]

\subsection{More conjugation lemmas}
From now on let $m, n\geq0$ be integers.
We recall Lemma \ref{conjugation lemma} in the following lemma, then generalize Lemma \ref{conjugation lemma2}, and later prove more technical lemmas which will be used in the computation.

\begin{lem}\label{conjugation lemma appendix}
Let $v$ be an integer.  
Take an element  $h \in H(\bQ_{\ell})$ satisfying $\mathrm{pr}_{1}(h) \in \begin{pmatrix}
\ell^{v} \cdot \bZ_{\ell}^{\times} & \bQ_{\ell} \\ 0& 1
\end{pmatrix}$. 
 Then for any  $g \in G(\bQ_{\ell})$, we have  
\[
\mathfrak{Z}(\phi_{\infty}\otimes  \mathrm{ch}(hg K_{m,n})) = \ell^{-v} \cdot \mathfrak{Z}(\phi_{\infty}\otimes  \mathrm{ch}(g K_{m,n})). 
\]
In particular, if $h \in K_{m,n}$ and $v=0$, we have 
\[
\mathfrak{Z}(\phi_{\infty}\otimes  \mathrm{ch}(hgh^{-1} K_{m,n})) = 
\mathfrak{Z}(\phi_{\infty}\otimes  \mathrm{ch}(hg K_{m,n})) =
\mathfrak{Z}(\phi_{\infty}\otimes  \mathrm{ch}(g K_{m,n})). 
\]
\end{lem}

\begin{cor}\label{conjugation corollary}\ 
\begin{itemize}
\item[(1)] For any $\alpha, \beta \in \bZ_{\ell}^{\times}$ with $\alpha\beta \in 1+ \ell^m \bZ_\ell$ and $a,b,c \in \bQ_{\ell}$, 
we have 
\begin{align*}
\mathfrak{Z}(\phi_{\infty}\otimes  \mathrm{ch}(\eta^{a,b,c}_{i,j} K_{m,n}))
    &=
    \mathfrak{Z}(\phi_{\infty}\otimes  \mathrm{ch}(\eta^{\beta a, \alpha b, \alpha \beta^{-1} c}_{i,j} K_{m,n})), 
    \\
        \mathfrak{Z}(\phi_{\infty}\otimes  \mathrm{ch}(\theta^{a,b,c}_{i,j} K_{m,n}))
    &=
    \mathfrak{Z}(\phi_{\infty}\otimes  \mathrm{ch}(\theta^{\beta a, \alpha b, \alpha \beta^{-1} c}_{i,j} K_{m,n})),  
    \\
    \mathfrak{Z}(\phi_{\infty}\otimes  \mathrm{ch}(\iota^{a,b,c}_{i,j} K_{m,n}))
    &=
    \mathfrak{Z}(\phi_{\infty}\otimes  \mathrm{ch}(\iota^{\beta a, \alpha b, \alpha \beta^{-1} c}_{i,j} K_{m,n})). 
\end{align*}
\item[(2)] For any $a_{1}, a_{2}, b_{1}, b_{2}, c_{1}, c_{2} \in \bZ_{\ell}$, we have 
\begin{align*}
\eta^{a_{1},b_{1},c_{1}}_{i,j}\eta^{a_{2}, b_{2}, c_{2}}_{i,j} = \eta^{a_{1}+a_{2},b_{1}+b_{2},c_{1}+c_{2}}_{i,j}, 
\\
\theta^{a_{1},b_{1},c_{1}}_{i,j}\eta^{a_{2},b_{2},c_{2}}_{i,j} = \theta^{a_{1}+a_{2},b_{1}+b_{2},c_{1}+c_{2}}_{i,j}, 
\\
\iota^{a_{1},b_{1},c_{1}}_{i,j}\eta^{a_{2},b_{2},c_{2}}_{i, j} = \iota^{a_{1}+a_{2},b_{1}+b_{2},c_{1}+c_{2}}_{i,j}. 
\end{align*}
In particular, if $a_{1}\equiv a_{2} \pmod {\ell^i}, b_{1}\equiv b_{2} \pmod {\ell^i}$ and $c_{1}\equiv c_{2} \pmod {\ell^j}$, then we have 
\begin{align*}
\ch(\eta^{a_{1},b_{1},c_{1}}_{i,j}K_{m,n}) &= \ch(\eta^{a_{2}, b_{2}, c_{2}}_{i,j}K_{m,n}), 
\\
\ch(\theta^{a_{1},b_{1},c_{1}}_{i,j}K_{m,n}) &= \ch(\theta^{a_{2}, b_{2}, c_{2}}_{i,j}K_{m,n}). 
\\
\ch(\iota^{a_{1},b_{1},c_{1}}_{i,j}K_{m,n}) &= \ch(\iota^{a_{2}, b_{2}, c_{2}}_{i,j}K_{m,n}). 
\end{align*}
\end{itemize}
\end{cor}

\begin{proof}
The element  
$\eta^{a, b, c}_{i,j}$ (resp. $\theta^{a, b, c}_{i,j}$, $\iota^{a, b, c}_{i,j}$) is conjugate to $\eta^{\beta a, \alpha b, \alpha \beta^{-1} c}_{i,j}$ (resp. $\theta^{\beta a, \alpha b, \alpha \beta^{-1} c}_{i,j}$, $\iota^{\beta a, \alpha b, \alpha \beta^{-1} c}_{i,j}$) via
$(\begin{pmatrix} 
\alpha \beta & \\ 
& 1 \\ 
\end{pmatrix},
\begin{pmatrix} 
\alpha& \\ 
& \beta\\ 
\end{pmatrix})$, which fixes $\phi_{s,t}$ for any $s,t\geq 0$ and belongs to $K_{m,n}$. 
Hence claim (1) follows from Lemma \ref{conjugation lemma appendix}. 
Claim (2) is by direct computation.
\end{proof}


\begin{cor}\label{cor:rel-eta}\ 
\begin{itemize}
\item[(1)] 
For any  $\alpha, \beta, \gamma, \delta \in \bZ_{\ell}$ with $\alpha \delta - \beta \gamma = 1$, we have 
\[
    \mathfrak{Z}(\phi_{\infty}\otimes  \mathrm{ch}(\eta^{a,b,0}_{i,0} K_{m,n}))
    =
    \mathfrak{Z}(\phi_{\infty}\otimes  \mathrm{ch}(\eta^{\delta a - \gamma b, -\beta a + \alpha b ,0}_{i,0} K_{m,n})) 
\]
Hence if $a,b\in \ZZ_\ell$ so that $(a,b)\neq (0,0)$, then 
\[
    \mathfrak{Z}(\phi_{\infty}\otimes  \mathrm{ch}(\eta^{a,b,0}_{i,0} K_{m,n}))
    =
    \mathfrak{Z}(\phi_{\infty}\otimes  \mathrm{ch}(\eta^{1,0,0}_{i-v,0} K_{m,n})), 
\]
where $v \defeq \min\{ \mathrm{ord}_{\ell}(a), \mathrm{ord}_{\ell}(b)\} < \infty$. 
\item[(2)] For any $\beta \in \bZ_{\ell}$, we have 
\[
    \mathfrak{Z}(\phi_{\infty}\otimes  \mathrm{ch}(\eta^{a,b,c}_{i,j} K_{m,n}))
    =
    \mathfrak{Z}(\phi_{\infty}\otimes  \mathrm{ch}(\eta^{ a, -\beta a +  b , c}_{i,j} K_{m,n})). 
\]
Hence if $\ord_\ell(a) \leq \ord_\ell(b)$, then
\[
    \mathfrak{Z}(\phi_{\infty}\otimes \mathrm{ch}(\eta^{a,b,c}_{i,j}K_{m,n}) = \mathfrak{Z}(\phi_{\infty}\otimes \mathrm{ch}(\eta^{a,0,c}_{i,j}K_{m,n}))
\]
\end{itemize}
\end{cor}
\begin{proof}\
\begin{itemize}
\item[(1)]
Apply Lemma \ref{conjugation lemma appendix} to $h = (\begin{pmatrix} 1 &  \\ & 1\end{pmatrix}, \begin{pmatrix}  \alpha & \beta \\ \gamma & \delta \end{pmatrix}) \in H(\QQ_\ell)\cap K_{m,n}$. 

\item[(2)] 
Apply Lemma \ref{conjugation lemma appendix} to $h = (\begin{pmatrix} 1 &  \\ & 1\end{pmatrix}, \begin{pmatrix}  1 & \beta \\  & 1 \end{pmatrix}) \in H(\QQ_\ell) \cap K_{m,n}$.

\end{itemize}
\end{proof}


\begin{cor}\label{cor:rel-theta}
Let $a,b,c \in \bZ_{\ell}$. 
If $i \leq 2$ and $j \leq 2$, we have 
    \[
    \mathfrak{Z}(\phi_{\infty}\otimes  \mathrm{ch}(\theta^{a,b,c}_{i,j} K_{m,n}))
    =
    \mathfrak{Z}(\phi_{\infty}\otimes  \mathrm{ch}(\theta^{a,b - ac \ell^{2-j},0}_{i, 0} K_{m,n})). 
    \]
    In particular, 
    \[
        \mathfrak{Z}(\phi_{\infty}\otimes  \mathrm{ch}(\theta^{a,b,c}_{1,1} K_{m,n}))
    =
    \mathfrak{Z}(\phi_{\infty}\otimes  \mathrm{ch}(\theta^{a,b, 0}_{1, 0} K_{m,n})). 
    \]
\end{cor}
\begin{proof}
Put 
\[
h \defeq  (\begin{pmatrix} 1 &  \\ & 1\end{pmatrix}, \begin{pmatrix}  1&  (\ell^{2}-1)^{-1}c \ell^{2-j} \\ & 1 \end{pmatrix})\in H(\QQ_\ell) \cap K_{m,n}. 
\]
Since $j \leq 2$ and $c \in \bZ_{\ell}$, we have $h \in H(\bZ_{\ell})$. 
Note that 
$(\ell^{2}-1)^{-1}ac \ell^{2-j} \equiv - ac \ell^{2-j} \pmod{\ell^{i}}$ since we assume $i \leq 2$. 
Since
\[
h \cdot \theta^{a,b,c}_{i,j} \cdot h^{-1} = \theta_{i, 0}^{a, b- (\ell^{2}-1)^{-1}ac \ell^{2-j}, 0}, 
\]
this corollary follows from Lemmas \ref{conjugation lemma appendix}  and \ref{conjugation corollary}. 
\end{proof}

\begin{rmk}\label{remark:theta}
When $i \leq 2$ and $j \leq 2$, Corollary \ref{cor:rel-theta} shows that 
\[
\sum_{b = 0}^{\ell^{i}-1}  \mathfrak{Z}(\phi_{\infty}\otimes  \mathrm{ch}(\theta^{a,b,c}_{i,j} K_{m,n}))
= \sum_{b = 0}^{\ell^{i}-1}  \mathfrak{Z}(\phi_{\infty}\otimes  \mathrm{ch}(\theta^{a,b,0}_{i,0} K_{m,n}))
\]
and
\[
\sum_{b = 0}^{\ell-1}  \mathfrak{Z}(\phi_{\infty}\otimes  \mathrm{ch}(\theta^{a,\ell b,c}_{i,1} K_{m,n}))
= \sum_{b = 0}^{\ell-1}  \mathfrak{Z}(\phi_{\infty}\otimes  \mathrm{ch}(\theta^{a,\ell b,0}_{i,0} K_{m,n})). 
\]
\end{rmk}


\begin{cor}\label{cor:rel-iota}\ 
For $j \leq 0$ and $c \in \bZ_{\ell}$, we have 
\begin{align*}
\ch(\iota^{a,b,c}_{i,j}K_{m,n}) = \ch(\iota^{a,b, 0}_{i, 0}K_{m,n}). 
\end{align*}
\end{cor}
\begin{proof}
This follows immediately from Corollary \ref{conjugation corollary} (2) and the fact that $c\in \ell^j\ZZ_\ell$.
\end{proof}


\begin{cor}\label{cor:rel-theta-iota}
    \[
    \mathfrak{Z}(\phi_{\infty}\otimes  \mathrm{ch}(\theta^{a,b,0}_{i,0} K_{m,n}))
    =
    \mathfrak{Z}(\phi_{\infty}\otimes  \mathrm{ch}(\iota^{b,-a,0}_{i,0} K_{m,n})). 
    \]
\end{cor}
\begin{proof}
Apply  Lemma \ref{conjugation lemma appendix} to  $w \defeq (\begin{pmatrix} 1 &  \\ & 1\end{pmatrix}, \begin{pmatrix}  & 1 \\ -1& \end{pmatrix}) \in H(\bQ_{\ell})\cap K_{m,n}$. 
\end{proof}


\begin{cor}\label{cor:rel-eta-iota}
Let $a,b\in \ZZ_\ell$ and $c \in \bZ_{\ell}^{\times}$ be a unit.  
For any $x \in \bZ_{\ell}^{\times}$ with $xc \equiv 1 \pmod{\ell}$, we have 
    \begin{align*}
    \mathfrak{Z}(\phi_{\infty} \otimes  \mathrm{ch}(\eta^{a,b,c}_{i,1} K_{m,n})) =
    \mathfrak{Z}(\phi_{\infty} \otimes  \mathrm{ch}(\iota^{a + bx \ell , b ,0}_{i,0} K_{m,n})). 
   \end{align*}
\end{cor}

\begin{proof}
Put $h \defeq (\begin{pmatrix} 1 &  \\ & 1\end{pmatrix}, \begin{pmatrix} 1 &  \\ - x\ell  & 1\end{pmatrix}) \in H(\bQ_{\ell})\cap K_{m,n}$. 
Then we have 
\[
h\eta^{a,b,c}_{i,1}h^{-1} = \iota^{a+b x \ell, b,0}_{i,0}(I_{4}, 
\begin{pmatrix} 
\ell & c  \\ -x & \ell^{-1}(1 - xc) 
\end{pmatrix}
\begin{pmatrix} 
1 & 0 \\ x \ell  & 1
\end{pmatrix} ). 
\]
Since $xc \equiv 1 \pmod{\ell}$, we see that 
\[
h\eta^{a,b,c}_{i,1}h^{-1} K_{m,n} = \iota^{a+bx\ell, b,0}_{i,0} K_{m,n}, 
\]
and Lemma \ref{conjugation lemma appendix} implies that 
\[
\mathfrak{Z}(\phi_{\infty}\otimes  \mathrm{ch}(\eta^{a,b,c}_{i,1} K_{m,n})) = \mathfrak{Z}(\phi_{\infty}\otimes  \mathrm{ch}(\iota_{i,0}^{a+bx \ell, b, 0} K_{m,n})). 
\]
\end{proof}



\begin{cor}\label{cor:eta-conjugation}
Let $a,b\in \ZZ_\ell$ and $c \in \bZ_{\ell}^{\times}$ be a unit.  
\begin{itemize}
\item[(1)] For any $\alpha \in \bZ_{\ell}^{\times}$, we have 
    \begin{align*}
 \mathfrak{Z}(\phi_{\infty} \otimes  \mathrm{ch}(\eta^{a,b,c}_{1,1} K_{m,n})) = \mathfrak{Z}(\phi_{\infty} \otimes  \mathrm{ch}(\eta^{\alpha a, \alpha^{-1}b, 1}_{1,1} K_{m,n})). 
       \end{align*}
 \item[(2)] For any $\alpha \in \bZ_{\ell}^{\times}$, we have 
     \begin{align*}
 \sum_{a=0}^{\ell-1}\mathfrak{Z}(\phi_{\infty}\otimes  \mathrm{ch}(\eta^{\ell a,b,c}_{2,1} K_{m,n})) =  \sum_{a=0}^{\ell-1} \mathfrak{Z}(\phi_{\infty} \otimes  \mathrm{ch}(\eta^{\ell a, \alpha b, 1}_{2,1} K_{m,n})). 
       \end{align*}
      \end{itemize}
\end{cor}

\begin{proof}
Use Corollaries \ref{conjugation corollary} and \ref{cor:rel-eta-iota}. 
\begin{itemize}
\item[(1)] 
   \begin{align*}
 \mathfrak{Z}(\phi_{\infty} \otimes  \mathrm{ch}(\eta^{a,b,c}_{1,1} K_{m,n})) &= 
 \mathfrak{Z}(\phi_{\infty} \otimes  \mathrm{ch}(\iota^{a, b, 0}_{1, 0} K_{m,n})) 
 \\
 &= 
 \mathfrak{Z}(\phi_{\infty} \otimes  \mathrm{ch}(\iota^{\alpha a, \alpha^{-1} b, 0}_{1, 0} K_{m,n})) 
 \\
 &= \mathfrak{Z}(\phi_{\infty} \otimes  \mathrm{ch}(\eta^{\alpha a, \alpha^{-1}b, 1}_{1,1} K_{m,n})). 
  \end{align*}
\item[(2)]
\begin{align*}
\sum_{a=0}^{\ell-1}\mathfrak{Z}(\phi_{\infty} \otimes  \mathrm{ch}(\eta^{\ell a,b,c}_{2,1} K_{m,n})) 
&=  \sum_{a=0}^{\ell-1} \mathfrak{Z}(\phi_{\infty} \otimes  \mathrm{ch}(\iota^{\ell (a + bc^{-1}), b, 0}_{2,0} K_{m,n}))
\\
&= \sum_{a=0}^{\ell-1} \mathfrak{Z}(\phi_{\infty} \otimes  \mathrm{ch}(\iota^{\ell a , b, 0}_{2,0} K_{m,n}))
\\
&= \sum_{a=0}^{\ell-1} \mathfrak{Z}(\phi_{\infty} \otimes  \mathrm{ch}(\iota^{\ell \alpha^{-1} a , \alpha b, 0}_{2,0} K_{m,n}))
\\
&= \sum_{a=0}^{\ell-1} \mathfrak{Z}(\phi_{\infty} \otimes  \mathrm{ch}(\iota^{\ell  a , \alpha b, 0}_{2,0} K_{m,n}))
\\
&= \sum_{a=0}^{\ell-1} \mathfrak{Z}(\phi_{\infty} \otimes  \mathrm{ch}(\eta^{\ell  (a-\alpha b) , \alpha b, 1}_{2,1} K_{m,n}))
\\
&= \sum_{a=0}^{\ell-1} \mathfrak{Z}(\phi_{\infty} \otimes  \mathrm{ch}(\eta^{\ell a , \alpha b, 1}_{2,1} K_{m,n})). 
\end{align*}
\end{itemize}
\end{proof}


\begin{rmk}\label{remark:theta and iota}

When $a \in \ZZ_\ell^\times, b\in \ZZ_\ell$,
Corollaries \ref{cor:rel-eta}(2), \ref{cor:rel-eta-iota}, and \ref{cor:eta-conjugation} show that  
\begin{align*}
  \mathfrak{Z}(\phi_{\infty}\otimes  \mathrm{ch}(\iota^{a, b ,0}_{1,0} K_{m,n}))
   &\stackrel{\ref{cor:rel-eta-iota}}{=} \mathfrak{Z}(\phi_{\infty}\otimes  \mathrm{ch}(\eta^{a-\ell b,b,1}_{1,1} K_{m,n}))
   \\
   &\stackrel{\ref{cor:rel-eta}}{=} \mathfrak{Z}(\phi_{\infty}\otimes  \mathrm{ch}(\eta^{a-\ell b,0,1}_{1,1} K_{m,n}))
   \\
   &\stackrel{\ref{cor:eta-conjugation}}{=} \mathfrak{Z}(\phi_{\infty}\otimes  \mathrm{ch}(\eta^{1,0,1}_{1,1} K_{m,n}))
   \\
   &\stackrel{\ref{cor:rel-eta-iota}}{=} 
   \mathfrak{Z}(\phi_{\infty}\otimes  \mathrm{ch}(\iota^{1,0,0}_{1,0} K_{m,n})). 
\end{align*}
By Corollary \ref{cor:rel-theta-iota}, we also have 
\[
 \mathfrak{Z}(\phi_{\infty}\otimes  \mathrm{ch}(\theta^{b, a ,0}_{1,0} K_{m,n}))
 = \mathfrak{Z}(\phi_{\infty}\otimes  \mathrm{ch}(\theta^{0,1,0}_{1,0} K_{m,n})). 
\]

\end{rmk}


\begin{lem}\label{lemma:auxiliary}
We have 
\[
\sum_{a=1}^{\ell-1}\mathfrak{Z}(\phi_{\infty} \otimes  \mathrm{ch}(\iota^{\ell a,1,0}_{2,0} K_{1,0})) 
= \mathfrak{Z}(\phi_{\infty} \otimes  \mathrm{ch}(\iota^{\ell ,1,0}_{2,0} K_{0,0})) 
= \mathfrak{Z}(\phi_{\infty} \otimes  \mathrm{ch}(\eta^{0 ,1,1}_{2,1} K_{0,0})). 
\]
\end{lem}
\begin{proof}
Put $h_a \defeq (\begin{pmatrix}
a &  \\ & 1
\end{pmatrix}, \begin{pmatrix}
1 &  \\ & a
\end{pmatrix})$.
Then $h_a \iota^{\ell,1,0}_{2,0} h_a^{-1}=\iota^{\ell a,1,0}_{2,0}$. 
Since $K_{0,0} = \bigsqcup_{a=1}^{\ell-1}h_a^{-1} K_{1,0}$, Lemma \ref{conjugation lemma appendix} shows that 
\begin{align*}
\mathfrak{Z}(\phi_{\infty} \otimes  \mathrm{ch}(\iota^{\ell ,1,0}_{2,0} K_{0,0})) &= \sum_{a=0}^{\ell-1}\mathfrak{Z}(\phi_{\infty} \otimes  \mathrm{ch}(\iota^{\ell,1,0}_{2,0} h_a^{-1} K_{1,0})) 
\\
&= \sum_{a=0}^{\ell-1}\mathfrak{Z}(\phi_{\infty} \otimes  \mathrm{ch}( h_a^{-1}\iota^{\ell a,1,0}_{2,0} K_{1,0})) 
\\
&= \sum_{a=0}^{\ell-1}\mathfrak{Z}(\phi_{\infty} \otimes  \mathrm{ch}( \iota^{\ell a,1,0}_{2,0} K_{1,0})). 
\end{align*}
The second equality follows from Corollary \ref{cor:rel-eta-iota}.
\end{proof}


\subsection{Computations}

Recall the notation
\[
\mathfrak{Z}(g) \defeq \mathfrak{Z}(\phi_\infty \otimes \ch(g K_{1,0})),
\]
for any element $g \in G(\bQ_\ell)$.

\subsubsection{$J_{1}$}

The finite set $J_1$ is the disjoint union of the following four subsets:
\begin{align*}
J_1^1 &= \{(A^{x, y, z}_1, B^u_1) \colon  x,y,z,u \in \{0, \ldots, \ell -1 \}\},
\\
J_1^2 &= \{
(A^{x, y, z}_1, B_2) \colon x,y,z, \in \{0, \ldots, \ell -1 \},
\\
J_1^3 &= \{
(A^{x,y}_2, B^u_1) \colon  x, y, u \in \{0, \ldots, \ell -1 \},
\\
J_1^4 &= \{
(A^{x, y}_2, B_2) \colon x, y \in \{0, \ldots, \ell -1 \},
\end{align*}
where to simplify notation (we only use the following notation in this subsubsection) we write
\[
A^{x, y, z}_1 = \begin{pmatrix}
\ell &  & x & y \\ 
& \ell & z & x \\ 
&& 1 &  \\
 &&&1
\end{pmatrix}, \qquad
A^{x,y}_2
=\begin{pmatrix}
\ell & x & & y \\ 
& 1 &&  \\ 
&& \ell & -x \\ 
&&&1
\end{pmatrix},
\]
and
\[
B^u_1 = 
\begin{pmatrix} \ell & u \\ 
& 1 
\end{pmatrix}, 
\qquad
B_2 = 
\begin{pmatrix} 1 &  \\
 & \ell \end{pmatrix}.
\]

\begin{lem}\label{lemma:J_{1}}\ 
\begin{itemize}

\item[(1)]
 \begin{align*} 
\sum_{\eta \in J_{1}^{1}}\mathfrak{Z}(\eta) =
\ell \cdot \mathfrak{Z}(\eta^{0,0,0}_{0,0})  
+ \ell(\ell-1) \cdot \mathfrak{Z}(\eta^{1, 0, 0}_{1, 0}) 
+ \ell(\ell-1) \cdot \mathfrak{Z}(\eta^{0, 0, 1}_{0,1}) 
+ \ell (\ell-1)^{2} \cdot \mathfrak{Z}(\eta^{0, 1, 1}_{1,1}).  
\end{align*}

\item[(2)] 
\begin{align*}
\sum_{\eta \in J_{1}^{2}}\mathfrak{Z}(\eta) =  
\mathfrak{Z}(\eta^{0,0,1}_{0,1}) 
+  (\ell-1) \cdot \mathfrak{Z}(\eta^{0, 1, 1}_{1,1}).
\end{align*}

\item[(3)] 
\begin{align*}
\sum_{\eta \in J_{1}^{3}}\mathfrak{Z}(\eta) = 
\sum_{\eta \in J_{1}^{2}}\mathfrak{Z}(\eta). 
\end{align*}

\item[(4)] 
\begin{align*}
\sum_{\eta \in J_{1}^{4}}\mathfrak{Z}(\eta) = \mathfrak{Z}(\eta^{0,0,0}_{0,0}) + 
(\ell-1 ) \cdot \mathfrak{Z}(\eta^{1,0,0}_{1,0}). 
\end{align*}
\end{itemize}

In particular, since $J_{1} = J_{1}^{1} \sqcup J_{1}^{2} \sqcup J_{1}^{3} \sqcup J_{1}^{4}$, 
\begin{align*}
\sum_{\eta \in J_{1}}\mathfrak{Z}(\eta) = 
   (\ell + 1) \cdot \mathfrak{Z}(\eta^{0,0,0}_{0,0})   
     +(\ell^{2}-1) \cdot \mathfrak{Z}(\eta^{1,0,0}_{1,0})  
+  \ell(\ell+1) \cdot \mathfrak{Z}(\eta^{0, 0, 1}_{0,1} ) 
+  \ell (\ell^{2}-1) \cdot \mathfrak{Z}(\eta^{0, 1, 1}_{1,1}).  
\end{align*} 

\end{lem}
\begin{proof} \
\begin{itemize} 
\item[(1)] 
For representatives in $J^1_1$, we have
\[
(A^{x,y,z}_1, B^u_1)
=
(\begin{pmatrix} \ell & y 
\\ & 1 
\end{pmatrix},
\begin{pmatrix} 
\ell & z 
\\ & 1 \end{pmatrix})
\cdot
\eta^{0, x, u-z}_{1,1}. 
\]
Hence Lemma \ref{conjugation lemma appendix} implies that 
\begin{align*}
\sum_{\eta \in J_{1}^{1}}\mathfrak{Z}(\eta) 
= \sum_{x,z,u = 0}^{\ell-1}\mathfrak{Z}(\eta^{0, x, u-z}_{1,1}). 
\end{align*}
Corollaries \ref{cor:rel-eta} and \ref{cor:eta-conjugation}(1) shows that 
\begin{align*}
 \sum_{x,z,u = 0}^{\ell-1}\mathfrak{Z}(\eta^{0, x, u-z}_{1,1}) 
 &=  \ell \cdot \sum_{x = 0}^{\ell-1} \mathfrak{Z}(\eta^{0, x, 0}_{1,0}) 
 +   \ell(\ell-1) \cdot \sum_{x = 0}^{\ell-1}\mathfrak{Z}(\eta^{0, x, 1}_{1,1}) 
 \\
&= \ell \cdot \mathfrak{Z}(\eta^{0,0,0}_{0,0})  
+ \ell(\ell-1) \cdot \mathfrak{Z}(\eta^{1, 0, 0}_{1, 0}) 
+ \ell(\ell-1) \cdot \mathfrak{Z}(\eta^{0, 0, 1}_{0,1}) 
+ \ell (\ell-1)^{2} \cdot \mathfrak{Z}(\eta^{0, 1, 1}_{1,1}). 
\end{align*}

\item[(2)] 
For representatives in $J^2_1$, we have
\[
(A^{x,y,z}_1, B_2)
=
(\begin{pmatrix} \ell & y 
\\ & 1 
\end{pmatrix},
\begin{pmatrix} 
\ell & z 
\\ & 1 \end{pmatrix})
\cdot
\iota^{0, x, -z}_{1,-1}. 
\]
Lemma \ref{conjugation lemma appendix} and Corollary \ref{cor:rel-iota} imply that 
\begin{align*}
\sum_{\eta \in J_{1}^{2}}\mathfrak{Z}(\eta) 
= \sum_{x, z = 0}^{\ell-1}\mathfrak{Z}(\iota^{0, x, -z}_{1,-1}) 
= \ell \cdot \sum_{x = 0}^{\ell-1}\mathfrak{Z}(\iota^{0, x, 0}_{1,0}), 
\end{align*}
and Corollary \ref{conjugation corollary} and Corollary \ref{cor:rel-eta-iota} show that 
\begin{align*}
 \sum_{x = 0}^{\ell-1}\mathfrak{Z}(\iota^{0, x, 0}_{1,0})
&= 
\mathfrak{Z}(\iota^{0,0,0}_{0,0}) 
+  (\ell-1) \cdot \mathfrak{Z}(\iota^{0, 1, 0}_{1,0}) \\
&=
\mathfrak{Z}(\eta^{0,0,1}_{0,1}) 
+  (\ell-1) \cdot \mathfrak{Z}(\eta^{-\ell, 1, 1}_{1,1}) \\
&=
\mathfrak{Z}(\eta^{0,0,1}_{0,1}) 
+  (\ell-1) \cdot \mathfrak{Z}(\eta^{0, 1, 1}_{1,1}) 
. 
\end{align*}

\item[(3)] 
For representatives in $J^3_1$, we have
\[
(A^{x,y}_2, B_1^{u})
=
(\begin{pmatrix} \ell & y 
\\ & 1 
\end{pmatrix},
\begin{pmatrix} 
1 &  
\\ & \ell \end{pmatrix})
\cdot
\theta^{x,0,u}_{1,1}. 
\]
Lemma \ref{conjugation lemma appendix} implies that  
\begin{align*}
\sum_{\eta \in J_{1}^{3}}\mathfrak{Z}(\eta) 
= \sum_{x, u = 0}^{\ell-1}\mathfrak{Z}(\theta^{x,0,u}_{1,1}) 
\end{align*}
Furthermore, Corollaries \ref{cor:rel-theta} and \ref{cor:rel-theta-iota} imply that 
\[
\sum_{x, u = 0}^{\ell-1}\mathfrak{Z}(\theta^{x,0,u}_{1,1}) = \ell \cdot \sum_{x=0}^{\ell -1} \mathfrak{Z}(\theta^{x,0,0}_{1,0}) = 
\ell \cdot \sum_{x=0}^{\ell -1} \mathfrak{Z}(\iota^{0,x,0}_{1,0}) = \ell \cdot \sum_{\eta \in J_{1}^{2}}\mathfrak{Z}(\eta). 
\]

\item[(4)]
For representatives in $J^4_1$, we have
\[
(A^{x,y}_2, B_2)
=
(\begin{pmatrix} \ell & y 
\\ & 1 
\end{pmatrix},
\begin{pmatrix} 
1 &  
\\ & \ell \end{pmatrix})
\cdot
\eta^{x,0,0}_{1,0}.
\]
Hence Lemma \ref{conjugation lemma appendix} and Corollary \ref{conjugation corollary} imply that 
\begin{align*}
\sum_{\eta \in J_{1}^{3}}\mathfrak{Z}(\eta) 
= \sum_{x = 0}^{\ell-1}\mathfrak{Z}(\eta^{x,0,0}_{1,0})   
=  \mathfrak{Z}(\eta^{0,0,0}_{0,0}) + 
(\ell-1) \cdot \mathfrak{Z}(\eta^{1,0,0}_{1,0}). 
\end{align*}
\end{itemize}
\end{proof}


\subsubsection{$J_{2}$}
We have 
\begin{align*}
J_2 = \{(A^{x, y, z}_1, \ell \cdot I_{2}) \colon  x,y \in \{0, \ldots, \ell -1 \}, z \in \{0, \ldots, \ell^{2}-1\}\},
\end{align*}
where to simplify notation (we only use the following notation in this subsubsection) we write
\[
A^{x, y, z}_1 = 
\begin{pmatrix}
 \ell^2 & \ell x & \ell y & z \\
 & \ell &  & y \\
 & & \ell & -x \\
 & & & 1
\end{pmatrix}
\,\,\, \textrm{ and } \,\,\, 
I_{2} = 
\begin{pmatrix} 1 &  \\ 
& 1 
\end{pmatrix}. 
\]

\begin{lem}\label{lemma:J_{2}}
\begin{align*}
\sum_{\eta \in J_{2}}\mathfrak{Z}(\eta) 
=  \mathfrak{Z}(\eta^{0,0,0}_{0,0}) 
+ (\ell^{2}-1) \cdot \mathfrak{Z}(\eta^{1, 0, 0}_{1,0}). 
\end{align*}
\end{lem}
\begin{proof}
For representatives in $J_{2}$, we have
\[
(A^{x,y,z}_1, \ell \cdot I_{2})
=
(\begin{pmatrix} \ell^{2} & z 
\\ & 1 
\end{pmatrix},
\begin{pmatrix} 
\ell &  
\\ & \ell \end{pmatrix})
\cdot
\eta^{x,y,0}_{1,0}. 
\]
Lemma \ref{conjugation lemma appendix} implies that  
\begin{align*}
\sum_{\eta \in J_{2}}\mathfrak{Z}(\eta) 
=  \sum_{x, y = 0}^{\ell-1}\mathfrak{Z}(\eta^{x,y,0}_{1,0}),   
\end{align*}
and Corollary  \ref{cor:rel-eta} 
shows that 
\begin{align*}
\sum_{x, y = 0}^{\ell-1}\mathfrak{Z}(\eta^{x,y,0}_{1,0})
= \mathfrak{Z}(\eta^{0,0,0}_{0,0}) 
+ (\ell^{2}-1) \cdot \mathfrak{Z}(\eta^{1, 0, 0}_{1,0}). 
\end{align*}
\end{proof}

\subsubsection{$J_{3}$}
The finite set $J_3$ is the disjoint union of the following three subsets:
\begin{align*}
J_3^{1} &= \{(A^{x, y, z}_1,  B_{1}^{u}) \colon  x,y \in \{0, \ldots, \ell -1 \}, z, u \in \{0, \ldots, \ell^{2}-1\}\}, 
\\
J_3^{2} &= \{(A^{x, y, z}_1, B_{2}^{u}) \colon  x,y \in \{0, \ldots, \ell -1 \}, z \in \{0, \ldots, \ell^{2}-1\}, u \in \{1, \ldots, \ell-1\}\}, 
\\
J_3^{3} &= \{(A^{x, y, z}_1, B_{3}) \colon  x,y \in \{0, \ldots, \ell -1 \}, z \in \{0, \ldots, \ell^{2}-1\}\}, 
\end{align*}
where to simplify notation (we only use the following notation in this subsubsection) we write
\[
A^{x, y, z}_1 = 
\begin{pmatrix}
 \ell^2 & \ell x & \ell y & z \\
 & \ell &  & y \\
 & & \ell & -x \\
 & & & 1
\end{pmatrix}
\]
and
\[
B^u_1 = 
\begin{pmatrix} \ell^{2} & u \\ 
& 1 
\end{pmatrix}, 
\qquad
B_2^{u} = 
\begin{pmatrix} \ell & u \\
 & \ell \end{pmatrix}, 
 \qquad
B_3 = 
\begin{pmatrix} 1 &  \\
 & \ell^{2} \end{pmatrix}. 
\]
Put 
\[
S_3 \defeq \mathfrak{Z}(\eta^{0,0,1}_{0,1}) 
+ \ell(\ell-1) \cdot \mathfrak{Z}(\eta^{1, 0, 1}_{1,1}) 
+ (\ell-1) \cdot \mathfrak{Z}(\eta^{0, 1, 1}_{1,1}). 
\]
Note that we have 
\[
\sum_{x,y = 0}^{\ell-1}\mathfrak{Z}(\theta^{x, y, 0}_{1,0}) \stackrel{\ref{cor:rel-theta-iota}}{=}
\sum_{x,y = 0}^{\ell-1}\mathfrak{Z}(\iota^{x, y, 0}_{1,0}) 
\stackrel{\ref{cor:rel-eta-iota}}{=}
\sum_{x,y = 0}^{\ell-1}\mathfrak{Z}(\eta^{x, y, 1}_{1,1}) 
\stackrel{\ref{cor:rel-eta}(2),\ref{cor:eta-conjugation}}{=} S_3
\]

\begin{lem}\label{lemma:J_{3}}\ 
\begin{itemize}

\item[(1)]
\begin{align*}
\sum_{\eta \in J_{3}^{1}}\mathfrak{Z}(\eta) 
= \ell^2 S_3. 
\end{align*}

\item[(2)] 
\begin{align*}
\sum_{\eta \in J_{3}^{2}}\mathfrak{Z}(\eta) 
= (\ell -1 )S_3. 
\end{align*}

\item[(3)]  
\begin{align*}
 \sum_{\eta \in J_{3}^{3}}\mathfrak{Z}(\eta) 
=  S_3.  
\end{align*}
\end{itemize}
In particular, since $J_{3} = J_{3}^{1} \sqcup J_{3}^{2} \sqcup J_{3}^{3}$, 
\begin{align*}
\sum_{\eta \in J_{3}}\mathfrak{Z}(\eta) = 
\ell(\ell +1 )\cdot \mathfrak{Z}(\eta^{0,0,1}_{0,1}) 
 + \ell^{2}(\ell^{2}-1) \cdot \mathfrak{Z}(\eta^{1, 0, 1}_{1,1})  
+ \ell(\ell^{2}-1) \cdot \mathfrak{Z}(\eta^{0, 1, 1}_{1,1}).
\end{align*}

\end{lem}

\begin{proof}
Put 
\[
h_{z} \defeq(\begin{pmatrix} \ell^{2} & z 
\\ & 1 
\end{pmatrix},
\begin{pmatrix} 
\ell &  
\\ & \ell \end{pmatrix}). 
\]

\begin{itemize}
\item[(1)] 
For representatives in $J^1_3$, we have
\[
(A^{x,y,z}_1, B^u_1)
=
h_{z}
\cdot
\theta^{x, y, u}_{1,2}. 
\]
Lemma \ref{conjugation lemma appendix} implies that 
\begin{align*}
\sum_{\eta \in J_{3}^{1}}\mathfrak{Z}(\eta) 
=  \sum_{u = 0}^{\ell^{2}-1}\sum_{x,y = 0}^{\ell-1}\mathfrak{Z}(\theta^{x, y, u}_{1,2}), 
\end{align*}
and Corollary \ref{cor:rel-theta}  (see Remark \ref{remark:theta}) implies that 
\begin{align*}
\sum_{u = 0}^{\ell^{2}-1}\sum_{x,y = 0}^{\ell-1}\mathfrak{Z}(\theta^{x, y, u}_{1,2}) 
= \ell^{2} \cdot \sum_{x,y = 0}^{\ell-1}\mathfrak{Z}(\theta^{x, y,0}_{1,0}) 
=  \ell^2 S_3. 
\end{align*}

\item[(2)] 
For representatives in $J^2_3$, we have
\[
(A^{x,y,z}_1, B^u_2)
=
h_{z}
\cdot
\eta^{x, y, u}_{1,1}. 
\]
Lemma \ref{conjugation lemma appendix} implies that 
\begin{align*}
\sum_{\eta \in J_{3}^{2}}\mathfrak{Z}(\eta) 
=  \sum_{u = 1}^{\ell-1}\sum_{x,y = 0}^{\ell-1}\mathfrak{Z}(\eta^{x, y, u}_{1,1}).  
\end{align*}
Since $\ell \nmid u$, Corollary \ref{cor:eta-conjugation}(1) implies that   
\begin{align*}
 \sum_{u = 1}^{\ell-1}\sum_{x,y = 0}^{\ell-1}\mathfrak{Z}(\eta^{x, y, u}_{1,1}) 
&= (\ell-1) \cdot \sum_{x,y = 0}^{\ell-1}\mathfrak{Z}(\eta^{x, y, 1}_{1,1}) = (\ell-1)S_3. 
\end{align*}

\item[(3)] For representatives in $J^3_3$, we have
\[
(A^{x,y,z}_1, B_3)
=
h_{z}
\cdot
\iota^{x, y, 0}_{1,0}. 
\]
Lemma \ref{conjugation lemma appendix} implies that 
\begin{align*}
\sum_{\eta \in J_{3}^{3}}\mathfrak{Z}(\eta) 
= \sum_{x,y = 0}^{\ell-1}\mathfrak{Z}(\iota^{x, y, 0}_{1,0}) = S_3.    
\end{align*}

\end{itemize}
\end{proof}



\subsubsection{$J_{4}$}
The finite set $J_4$ is the disjoint union of the following three subsets:
\begin{align*}
J_4^1 &= \{
(A^{x,y,z}_1, \ell \cdot I_{2}) \colon x,y,z \in \{0, \ldots, \ell^{2}-1\}\}, 
\\
J_4^2 &= \{
(A^{x,y,z,w}_2, \ell \cdot I_{2}) \colon x,y \in \{0, \ldots, \ell-1\}, z \in \{0, \ldots, \ell^{2}-1\}, w \in \{1, \ldots, \ell-1\}\},
\\
J_4^3 &= \{
(A^{x,y}_3, \ell \cdot I_{2}) \colon x,y \in \{0, \ldots, \ell^{2}-1\}\}, 
\end{align*}
where to simplify notation (we only use the following notation in this subsubsection) we write
\[
A^{x,y,z}_1 = \begin{pmatrix}
 \ell^2 &  & x & y \\
 & \ell^2 & z & x \\
 & & 1 &  \\
 & & & 1
\end{pmatrix}, \,\,\,  
A^{x,y,z,w}_2
=\begin{pmatrix}
 \ell^2 & \ell x & wx+\ell y & z \\
 & \ell & w & y \\
 & & \ell & -x \\
 & & & 1
\end{pmatrix}, \,\,\,  
A_{3}^{x,y} \defeq 
\begin{pmatrix}
 \ell^2 &  x &  & y \\
 &  1 &  &  \\
 & & \ell^2 & -x \\
 & & & 1
\end{pmatrix}
\]
and $I_{2} = \begin{pmatrix} 1 &  \\ 
& 1 
\end{pmatrix}$. 
Put 
\[
S_4 \defeq \mathfrak{Z}(\eta^{0, 0, 1}_{0, 1}) 
+ (\ell -1) \cdot \mathfrak{Z}(\eta^{0, 1, 1}_{1, 1}) 
+ \ell(\ell-1) \cdot \mathfrak{Z}(\eta^{-\ell, 1,  1}_{2, 1}). 
\]
Note that we have 
\begin{align*}
\sum_{x = 0}^{\ell^{2}-1}\mathfrak{Z}(\iota^{0, x, 0}_{2,0}) = S_4 
\end{align*}
by Corollaries \ref{conjugation corollary} and \ref{cor:rel-eta-iota}.

\begin{lem}\label{lemma:J_{4}}\ 
\begin{itemize}
\item[(1)]
\begin{align*}
\sum_{\eta \in J_{4}^{1}}\mathfrak{Z}(\eta) 
= \ell^{2} S_4. 
\end{align*}

\item[(2)] 
\begin{align*}
\sum_{\eta \in J_{4}^{2}}\mathfrak{Z}(\eta) 
= 
(\ell -1) S_4.
\end{align*}

\item[(3)] 
\begin{align*}
\sum_{\eta \in J_{4}^{3}}\mathfrak{Z}(\eta) 
= S_4. 
\end{align*}
\end{itemize}
In particular, since $J_4=J_4^1 \sqcup J_4^2 \sqcup J_4^3$,
\begin{align*}
\sum_{\eta \in J_{4}}\mathfrak{Z}(\eta) =  \ell(\ell +1) \cdot \mathfrak{Z}(\eta_{0,1}^{0, 0, 1}) 
+ \ell(\ell^{2} -1) \cdot \mathfrak{Z}(\eta_{1,1}^{0, 1,1})    
+ \ell^{2}(\ell^{2}-1) \cdot  \mathfrak{Z}(\eta^{-\ell, 1,  1}_{2, 1}). 
\end{align*}
\end{lem}

\begin{proof} \
\begin{itemize}
\item[(1)] 
Put 
\[
h_{z, y} \defeq (\begin{pmatrix} \ell^{2} & y 
\\ & 1 
\end{pmatrix},
\begin{pmatrix} 
\ell^{2} &  z
\\ & 1 \end{pmatrix}). 
\]
For representatives in $J^1_4$, we have
\[
(A^{x,y,z}_1, \ell \cdot I_{2})
=
h_{z, y}
\cdot
\iota^{0, x, -z}_{2,0}. 
\]
Lemma \ref{conjugation lemma appendix} and Corollary \ref{cor:rel-iota} show that 
\begin{align*}
\sum_{\eta \in J_{4}^{1}}\mathfrak{Z}(\eta) 
= \sum_{x, z = 0}^{\ell^{2}-1}\mathfrak{Z}(\iota^{0, x, -z}_{2,0}) 
= \ell^2 \cdot \sum_{x = 0}^{\ell^{2}-1}\mathfrak{Z}(\iota^{0, x, 0}_{2,0}) = \ell^2 S_4.  
\end{align*}

\item[(2)] 
Put 
\[
h_{z, w} \defeq (\begin{pmatrix} \ell^{2} & z 
\\ & 1 
\end{pmatrix},
\begin{pmatrix} 
\ell &  w
\\ & \ell \end{pmatrix}). 
\]
For representatives in $J^2_4$, we have
\[
(A^{x,y,z,w}_2, \ell \cdot I_{2})
=
h_{z, w}
\cdot
\eta_{2,1}^{\ell x, wx + \ell y, -w}. 
\]
Lemma \ref{conjugation lemma appendix} implies that 
\begin{align*}
\sum_{\eta \in J_{4}^{2}}\mathfrak{Z}(\eta) 
= \sum_{x,y=0}^{\ell-1}\sum_{w=1}^{\ell-1}\mathfrak{Z}(\eta_{2,1}^{\ell x, wx + \ell y, -w}), 
\end{align*}
and Corollary \ref{cor:rel-eta-iota} shows that  
\begin{align*}
\mathfrak{Z}(\eta_{2,1}^{\ell x, wx + \ell y, -w}) 
= \mathfrak{Z}(\iota_{2,0}^{0, wx + \ell y, 0}). 
\end{align*}
Hence we have 
\begin{align*}
\sum_{\eta \in J_{4}^{2}}\mathfrak{Z}(\eta) 
= 
(\ell-1) \cdot \sum_{y=0}^{\ell^{2}-1}\mathfrak{Z}(\iota_{2,0}^{0, y, 0}) = (\ell-1)S_4. 
\end{align*}

\item[(3)] 
Put 
\[
h_{y} \defeq (\begin{pmatrix} \ell^{2} & y 
\\ & 1 
\end{pmatrix},
\begin{pmatrix} 
1 &  
\\ & \ell^{2} 
\end{pmatrix}). 
\]
For representatives in $J^3_4$, we have
\[
(A^{x,y}_3, \ell \cdot I_{2})
=
h_{y}
\cdot
\theta_{2,0}^{x,0,0}. 
\]
Lemma \ref{conjugation lemma appendix} and Corollary \ref{cor:rel-theta-iota} imply that 
\begin{align*}
\sum_{\eta \in J_{4}^{3}}\mathfrak{Z}(\eta) 
= \sum_{x=0}^{\ell^{2}-1}\mathfrak{Z}(\theta_{2,0}^{x,0,0})
= \sum_{x=0}^{\ell^{2}-1}\mathfrak{Z}(\iota_{2,0}^{0,x,0}) = S_4. 
\end{align*}
\end{itemize}

\end{proof}



\subsubsection{$J_{5}$}

The finite $J_5$ is the disjoint union of the following four subsets:
\begin{align*}
J_5^1 &= \{
(A^{x,y,z,w}_1, B^u_1) \colon x,w,u \in \{0, \ldots \ell -1\}, y \in \{0, \ldots, \ell^{2}-1\}, z \in \{0, \ldots, \ell^{3}-1\}\}, 
\\
J_5^2 &= \{
(A^{x,y,z,w}_1, B_2) \colon x,w \in \{0, \ldots \ell -1\}, y \in \{0, \ldots, \ell^{2}-1\}, z \in \{0, \ldots, \ell^{3}-1\}\},
\\
J_5^3 &= \{
(A^{x,y,z}_2, B^u_1) \colon x,u \in \{ 0, \ldots, \ell-1\}, y \in \{0, \ldots, \ell^2-1\}, z \in \{0, \ldots, \ell^3-1\}\},
\\
J_5^4 &= \{
(A^{x,y,z}_2, B_2) \colon x \in \{ 0, \ldots, \ell-1\}, y \in \{0, \ldots, \ell^2-1\}, z \in \{0, \ldots, \ell^3-1\}\},
\end{align*}
where to simplify notation (we only use the following notation in this subsubsection) we write
\[
A^{x,y,z,w}_1 = \begin{pmatrix}
\ell^3 & \ell^2 x & \ell  y & z \\ 
& \ell^2 & \ell w & y-wx \\ 
&& \ell & -x \\ 
&&&1
\end{pmatrix}, \qquad
A^{x,y,z}_2
=\begin{pmatrix}
\ell^3 & \ell y & \ell^2 x & z \\ 
& \ell && x \\ 
&& \ell^2 & -y \\ 
&&&1
\end{pmatrix},
\]
and
\[
B^u_1 = 
\begin{pmatrix} 
\ell^2 & \ell u \\ 
& \ell 
\end{pmatrix}, \qquad
B_2 = 
\begin{pmatrix} \ell &  \\ 
& \ell^2 
\end{pmatrix}.
\]
Put 
\[
S_5\defeq \sum_{x=0}^{\ell-1}\sum_{y=0}^{\ell^{2}-1}
\mathfrak{Z}(\iota^{\ell x, y, 0}_{2,0}).
\]
Corollary \ref{conjugation corollary} and Remark \ref{remark:theta and iota} show that 
\begin{align*}
S_5 = 
\mathfrak{Z}(\iota^{0,0, 0}_{0,0}) &+  (\ell-1) \cdot \mathfrak{Z}(\iota^{0, 1, 0}_{1,0}) 
+ \ell(\ell-1) \cdot \mathfrak{Z}(\iota^{1, 0, 0}_{1,0}) + \ell(\ell-1) \cdot \sum_{x=0}^{\ell-1} \mathfrak{Z}(\iota^{\ell x, 1, 0}_{2,0}),
\end{align*}
which by Corollary \ref{cor:rel-eta-iota} and Lemma \ref{lemma:auxiliary} is equal to 
\begin{align*}
\mathfrak{Z}(\eta^{0,0, 1}_{0,1}) &+  (\ell-1) \cdot \mathfrak{Z}(\eta^{0, 1, 1}_{1,1}) 
+ \ell(\ell-1) \cdot \mathfrak{Z}(\eta^{1, 0, 1}_{1,1}) 
+ \ell(\ell-1) \cdot \frakZ(\eta_{2,1}^{-\ell,1,1})
+ \ell(\ell-1) \cdot \mathfrak{Z}(\phi_{\infty} \otimes  \mathrm{ch}(\eta^{0 ,1,1}_{2,1} K_{0,0})), 
\end{align*}
where the second term happens when $\ell\mid y$ and $(x,y)\neq(0,0)$ and the third term happens when $\ell\nmid y$.

\begin{lem}\label{lemma:J_{5}}\ 
\begin{itemize}

\item[(1)] 
\begin{align*}
\sum_{\eta \in J_{5}^{1}}\mathfrak{Z}(\eta)  = \ell \cdot \mathfrak{Z}(\eta^{0,0, 0}_{0,0}) + \ell(\ell^{2}-1) \cdot \mathfrak{Z}(\eta^{1,0, 0}_{1,0}) 
+ \ell^{3}(\ell-1) \cdot \mathfrak{Z}(\eta^{1,0, 0}_{2,0}) + \ell(\ell-1) S_5
\end{align*}

\item[(2)]  
\begin{align*}
\sum_{\eta \in J_{5}^{2}}\mathfrak{Z}(\eta)  
&= \ell S_5.
\end{align*}

\item[(3)]
\begin{align*}
\sum_{\eta \in J_{5}^{3}}\mathfrak{Z}(\eta)  
= \sum_{\eta \in J_{5}^{2}}\mathfrak{Z}(\eta)
= \ell S_5. 
\end{align*}

\item[(4)]  
\begin{align*}
\sum_{\eta \in J_{5}^{4}}\mathfrak{Z}(\eta)  
= 
\mathfrak{Z}(\eta^{0,0,0}_{0,0} ) 
+ (\ell^{2}-1) \cdot \mathfrak{Z}(\eta^{1, 0, 0}_{1,0}) 
+ \ell^{2}(\ell-1) \cdot\mathfrak{Z}(\eta^{1, 0, 0}_{2,0}).  
\end{align*}
\end{itemize}
In particular, since $J_5 = J_5^1 \sqcup J_5^2 \sqcup J_5^3 \sqcup J_5^4$,
\begin{align*}
\sum_{\eta \in J_{5}}\mathfrak{Z}(\eta  )  = 
(\ell+1) \cdot 
\mathfrak{Z}(\eta^{0,0,0}_{0,0}) 
&+ (\ell+1)(\ell^{2}-1) \cdot \mathfrak{Z}(
\eta^{1, 0, 0}_{1, 0}   ) 
+ \ell^{2}(\ell^{2}-1) \cdot \mathfrak{Z}(   
\eta^{1, 0, 0}_{2,0}   ) 
\\
&+ \ell(\ell+1) \cdot \mathfrak{Z}(   
\eta^{0, 0, 1}_{0, 1}   ) 
+ \ell(\ell^{2}-1) \cdot \mathfrak{Z}(  \eta^{0,1,1}_{1,1}  ) 
+ \ell^{2}(\ell^{2}-1) \cdot\mathfrak{Z}(  \eta^{1, 0, 1}_{1,1}  ) 
\\
&+ \ell^2(\ell^2-1) \cdot \frakZ(\eta_{2,1}^{-\ell,1,1})
+ \ell^2(\ell^2-1) \cdot \mathfrak{Z}(\phi_{\infty} \otimes  \mathrm{ch}(\eta^{0 ,1,1}_{2,1} K_{0,0})).
\end{align*}

\end{lem}
\begin{proof} \

\begin{itemize}
\item[(1)] 
Put 
\[
h_{z,w} \defeq (\begin{pmatrix} \ell^3 & z \\ & 1 \end{pmatrix},
\begin{pmatrix} \ell^2 & \ell w \\ & \ell \end{pmatrix}). 
\]
For representatives in $J^1_5$, we have
\[
(A^{x,y,z,w}_1, B^u_1)
=
h_{z,w} \cdot
\eta^{\ell x, y, u-w}_{2,1}. 
\]
Lemma \ref{conjugation lemma appendix} implies that 
\begin{align*}
\sum_{\eta \in J_{5}^{1}}\mathfrak{Z}(\eta) 
= \sum_{x,w,u=0}^{\ell-1}\sum_{y=0}^{\ell^{2}-1}
\mathfrak{Z}(\eta^{\ell x, y, u-w}_{2,1}), 
\end{align*}
and  Corollary \ref{cor:eta-conjugation}(2) shows that 
\begin{align*}
\sum_{x,w,u=0}^{\ell-1}\sum_{y=0}^{\ell^{2}-1}
\mathfrak{Z}(\eta^{\ell x, y, u-w}_{2,1}) 
= \ell \cdot \sum_{x=0}^{\ell-1}\sum_{y=0}^{\ell^{2}-1}
\mathfrak{Z}(\eta^{\ell x, y, 0}_{2,0}) 
+ \ell(\ell-1) \cdot \sum_{x=0}^{\ell-1}\sum_{y=0}^{\ell^{2}-1}
\mathfrak{Z}(\eta^{\ell x, y, 1}_{2,1}). 
\end{align*}
Corollary \ref{cor:rel-eta}(1) shows that the sum $\sum_{x=0}^{\ell-1}\sum_{y=0}^{\ell^{2}-1}
\mathfrak{Z}(\eta^{\ell x, y, 0}_{2,0})$ is equal to the sum
\begin{align*}
\mathfrak{Z}(\eta^{0,0, 0}_{0,0}) + (\ell^{2}-1) \cdot \mathfrak{Z}(\eta^{1,0, 0}_{1,0}) 
+ \ell^{2}(\ell-1) \cdot \mathfrak{Z}(\eta^{1,0, 0}_{2,0}), 
\end{align*}
where the second term happens when $\ell\mid y$ and  $(x,y)\neq (0,0)$, and the third term happens when $\ell\nmid y$. 
Corollary \ref{cor:rel-eta-iota} shows that 
\[
\sum_{x=0}^{\ell-1}\sum_{y=0}^{\ell^{2}-1}
\mathfrak{Z}(\eta^{\ell x, y, 1}_{2,1})
= \sum_{x=0}^{\ell-1}\sum_{y=0}^{\ell^{2}-1}
\mathfrak{Z}(\iota^{\ell (x+y), y, 0}_{2,0}) = 
\sum_{x=0}^{\ell-1}\sum_{y=0}^{\ell^{2}-1}
\mathfrak{Z}(\iota^{\ell x, y, 0}_{2,0}). 
\]

\item[(2)] 
Put 
\[
h_{z,w} \defeq (\begin{pmatrix} \ell^3 & z \\ & 1 \end{pmatrix},
\begin{pmatrix} \ell^2 & \ell w \\ & \ell \end{pmatrix}). 
\]
For representatives in $J^2_5$, we have
\[
(A^{x,y,z,w}_1, B_2)
=
h_{z,w} \cdot
\iota^{\ell x, y, -w}_{2,-1}. 
\]
Lemma \ref{conjugation lemma appendix} implies that 
\begin{align*}
\sum_{\eta \in J_{5}^{2}}\mathfrak{Z}(\eta) 
= \sum_{x,w=0}^{\ell-1}\sum_{y=0}^{\ell^{2}-1}
\mathfrak{Z}(\iota^{\ell x, y, -w}_{2,-1}), 
\end{align*}
and Corollary \ref{cor:rel-iota} implies that  
\[
\sum_{x,w=0}^{\ell-1}\sum_{y=0}^{\ell^{2}-1}
\mathfrak{Z}(\iota^{\ell x, y, -w}_{2,-1}) = \ell \cdot \sum_{x=0}^{\ell-1}\sum_{y=0}^{\ell^{2}-1}
\mathfrak{Z}(\iota^{\ell x, y, 0}_{2,0}). 
\]

\item[(3)] 

Put 
\[
h_{z} \defeq (\begin{pmatrix} \ell^3 & z \\ & 1 \end{pmatrix},
\begin{pmatrix} \ell &  \\ & \ell^{2} \end{pmatrix}). 
\]
For representatives in $J^3_5$, we have
\[
(A^{x,y,z}_1, B_1^{u})
=
h_{z} \cdot
\theta^{y,\ell x, u}_{2,1}. 
\]
Lemma \ref{conjugation lemma appendix} implies that 
\begin{align*}
\sum_{\eta \in J_{5}^{3}}\mathfrak{Z}(\eta) 
= \sum_{x,u=0}^{\ell-1}\sum_{y=0}^{\ell^{2}-1}
\mathfrak{Z}(\theta^{y,\ell x, u}_{2,1}), 
\end{align*}
and Corollaries \ref{cor:rel-theta} and \ref{cor:rel-theta-iota} (see Remark \ref{remark:theta})  show that  
\[
\sum_{x,u=0}^{\ell-1}\sum_{y=0}^{\ell^{2}-1}
\mathfrak{Z}(\theta^{y,\ell x, u}_{2,1}) 
= \ell \cdot \sum_{x=0}^{\ell-1}\sum_{y=0}^{\ell^{2}-1}
\mathfrak{Z}(
\theta^{y,\ell x, 0}_{2,0}) 
= \ell \cdot \sum_{x=0}^{\ell-1}\sum_{y=0}^{\ell^{2}-1}
\mathfrak{Z}(
\iota^{\ell x, y,0}_{2,0}) 
= \sum_{\eta \in J_{5}^{2}}\mathfrak{Z}(\eta). 
\]

\item[(4)] 

Put 
\[
h_{z} \defeq (\begin{pmatrix} \ell^3 & z \\ & 1 \end{pmatrix},
\begin{pmatrix} \ell &  \\ & \ell^{2} \end{pmatrix}). 
\]
For representatives in $J^4_5$, we have
\[
(A^{x,y,z}_1, B_2)
=
h_{z} \cdot
\eta^{y,\ell x, 0}_{2,0}. 
\]
Lemma \ref{conjugation lemma appendix} implies that 
\begin{align*}
\sum_{\eta \in J_{5}^{4}}\mathfrak{Z}(\eta) 
= \sum_{x=0}^{\ell-1}\sum_{y=0}^{\ell^{2}-1}
\mathfrak{Z}(\eta^{y,\ell x, 0}_{2,0}), 
\end{align*}
and Corollary \ref{cor:rel-eta} implies that 
\begin{align*}
\sum_{x=0}^{\ell-1}\sum_{y=0}^{\ell^{2}-1}
\mathfrak{Z}(\eta^{y,\ell x, 0}_{2,0}) 
=  \mathfrak{Z}(\eta^{0,0,0}_{0,0}) 
+ (\ell^{2}-1) \cdot \mathfrak{Z}(\eta^{1, 0, 0}_{1,0}) 
+ \ell^{2}(\ell-1) \cdot \mathfrak{Z}(\eta^{1, 0, 0}_{2,0}). 
\end{align*}
\end{itemize}
\end{proof}



\subsubsection{$J_{6}$}

By definition, we have 
\begin{align*}
J_6 &= \{
(A^{x,y,z}_1, \ell^{2} \cdot I_{2}) \colon x,y \in \{0, \ldots, \ell^{2}-1\}, z \in \{0, \ldots, \ell^{4}-1\}\}, 
\end{align*}
where to simplify notation (we only use the following notation in this subsubsection) we write
\[
A^{x,y,z}_1 = \begin{pmatrix}
 \ell^4 & \ell^{2}x & \ell^{2}y & z \\
 & \ell^2 &  & y \\
 & & \ell^{2} & -x  \\
 & & & 1
\end{pmatrix} \,\,\,  \textrm{ and } \,\,\, I_{2} = \begin{pmatrix} 1 &  \\ 
& 1 
\end{pmatrix}. 
\]

\begin{lem}\label{lemma:J_{6}}
\begin{align*}
\sum_{\eta \in J_{6}}\mathfrak{Z}(\eta) 
=  \mathfrak{Z}(\eta^{0,0,0}_{0,0}) 
+ (\ell^{2}-1) \cdot \mathfrak{Z}(\eta^{1, 0, 0}_{1,0}) 
+ \ell^{2}(\ell^{2}-1) \cdot \mathfrak{Z}(\eta^{1, 0, 0}_{2,0}).  
\end{align*}
\end{lem}
\begin{proof}
Put 
\[
h_{z} \defeq (\begin{pmatrix} \ell^4 & z \\ & 1 \end{pmatrix},
\begin{pmatrix} \ell^{2} &  \\ & \ell^{2} \end{pmatrix}). 
\]
For representatives in $J_{6}$, we have
\[
(A^{x,y,z}_1, B_2)
=
h_{z} \cdot
\eta^{x, y, 0}_{2,0}. 
\]
Lemma \ref{conjugation lemma appendix} implies that 
\begin{align*}
\sum_{\eta \in J_{6}}\mathfrak{Z}(\eta) 
= \sum_{x,y=0}^{\ell^{2}-1}
\mathfrak{Z}(\eta^{x, y, 0}_{2,0}), 
\end{align*}
and  Corollary \ref{cor:rel-eta} shows that 
\begin{align*}
\sum_{x,y=0}^{\ell^{2}-1}
\mathfrak{Z}(\eta^{x, y, 0}_{2,0}) = \mathfrak{Z}(\eta^{0,0,0}_{0,0}) 
+ (\ell^{2}-1) \cdot \mathfrak{Z}(\eta^{1, 0, 0}_{1,0}) 
+ \ell^{2}(\ell^{2}-1) \cdot \mathfrak{Z}(\eta^{1, 0, 0}_{2,0}).  
\end{align*}
where the second term happens when $\ell\mid x,\ell\mid y, (x,y)\neq (0,0)$, and the third term happens when $\ell\nmid x$ or $\ell\nmid y$.
\end{proof}

\bibliography{references}
\bibliographystyle{amsalpha} 

\end{document}